\documentclass[11pt,reqno,letter]{amsart}
\usepackage{multigerbe}
\title{Bigerbes}
\author{Chris Kottke}
\address{New College of Florida, Division of Natural Sciences}
\email{ckottke@ncf.edu}

\author{Richard Melrose}
\address{MIT, Department of Mathematics}
\email{rbm@math.mit.edu}
\begin{document}

\begin{abstract}

The bigerbes introduced here give a refinement of the notion of 2-gerbes,
representing degree four integral cohomology classes of a space. Defined in
terms of bisimplicial line bundles, bigerbes have a symmetry with respect
to which they form `bundle 2-gerbes' in two ways; this structure replaces
higher associativity conditions.  We provide natural examples, including a
Brylinski-McLaughlin bigerbe associated to a principal $G$-bundle for a
simply connected simple Lie group.  This represents the first Pontryagin
class of the bundle, and is the obstruction to the lifting problem on the
associated principal bundle over the loop space to the structure group
consisting of a central extension of the loop group; in particular,
trivializations of this bigerbe for a spin manifold are in bijection with
string structures on the original manifold. Other natural examples
represent `decomposable' 4-classes arising as cup products, a universal
bigerbe on $K(\bbZ,4)$ involving its based double loop space, and the
representation of any 4-class on a space by a bigerbe involving its free
double loop space. The generalization to `multigerbes' of
arbitrary degree is also described.
\end{abstract}

\maketitle
\tableofcontents

\section*{Introduction} \label{S:intro}
\bgroup
\renewcommand{\theequation}{\arabic{equation}}

Gerbes provide a (more or less) geometric representation of integral
cohomology 3-classes on a space \cite{giraud,brylinski}. Bundle gerbes,
introduced by Murray in \cite{Murray1}, are particularly geometric and have
a well-known application in the form of the `lifting bundle gerbe',
representing the obstruction to the extension of a principal $G$-bundle to
a principal bundle with structure group a $\UU(1)$ central extension of $G$. Here we
present a direct extension of the notion of a bundle gerbe to obtain a
similar representation of integral 4-classes. These {\em bigerbes} are
special cases, in a sense more rigid, of the bundle 2-gerbes as defined by
Stevenson \cite{Stevenson2001}, which in turn are a more geometric version
of 2-gerbes as defined by Breen \cite{breen}. In particular our bigerbes
induce bundle 2-gerbes in two ways. One application of this notion is to
Brylinski-McLaughlin (bi)gerbes, corresponding to the existence of an
extension of the principal bundle over the loop space induced by a
principal $G$-bundle over the original space, to a bundle with structure
group a central extension of the loop group \cite{B-MI}.

A gerbe may be defined as a simplicial object
\cite{MurrayStevenson,Stevenson2001}. We work in the context of locally
split maps, which is to say continuous maps $\pi:Y\longrightarrow X$, with
local right inverses over an open cover of the topological space $X$. Such a map determines an
associated simplicial space, $Y^{[\bullet]}$, over $X$, formed from the
fiber products $Y^{[k]} = Y\times_X \cdots \times_X Y$:
\begin{equation}
\begin{tikzcd}[sep=scriptsize]
Y & Y^{[2]} \arltwo & Y^{[3]} \arlthree & \ar[l,thick,dotted, no head].
\end{tikzcd}
\label{rbm.42}\end{equation}
This constitutes a contravariant functor $\Delta \to \mathsf{Top}/X$,
where $\Delta$ denotes the simplex category with objects 
the sets $\mathbf{n} = \set{1,\ldots,n}$
for $n \in \bbN_0$ with morphisms the order preserving maps
between these, and $\mathsf{Top}/X$ denotes the category of spaces with commuting maps to $X.$
Functions on $Y^{[\bullet]}$ admit a simplicial differential, denoted by $d$,
by taking the
alternating sum of the pull-backs, and this operation extends to line (or
circle) bundles and sections thereof by taking the alternating tensor product of the pull-backs.

A bundle gerbe on $X$ is specified in terms of the simplicial space
\eqref{rbm.42} by the prescription of a complex line bundle $L$ over
$Y^{[2]}$ such that $dL$ over $Y^{[3]}$ has a section $s$ which pulls back
to be the canonical section of $d^2L$ over $Y^{[4]}$. The important special
case of the lifting bundle gerbe is obtained when $\pi : E \to X$ is 
a principal $G$-bundle; then 
there is a natural map $ E^{[2]}\longrightarrow G $,
and the line bundle is the pull-back of the line bundle over $G$
associated to a given central extension of $G$ by $\bbC^*$ or $\UU(1)$.

Our notion of a bigerbe is based on a {\em split square} of maps. This
is a commutative square of locally split maps 
\begin{equation}
\begin{tikzcd}
	Y_2 \ar[d] & \ar[l] \ar[d] W \\
	X & \ar[l] Y_1
\end{tikzcd}
\label{rbm.44}\end{equation}
with the additional property that the induced map 
\begin{equation*}
W\longrightarrow Y_1\times_XY_2
\label{rbm.45}\end{equation*}
is also locally split (and in particular surjective).

Such a split square induces a bisimplicial space $W^{[\bullet,\bullet]}$ over $X:$
\begin{equation}
\begin{tikzcd}[sep=small]
	\ar[d,dotted, no head] &  \ar[d,dotted, no head] & \ar[d,dotted, no head] & \ar[d,dotted, no head] 
	\\ Y_2^{[3]} \ardthree & W^{[1,3]} \ardthree \ar[l] & W^{[2,3]} \ardthree \arltwo & W^{[3,3]} \ardthree \arlthree & \ar[l,dotted,no head]
	\\ Y_2^{[2]} \ardtwo & W^{[1,2]} \ardtwo \ar[l] & W^{[2,2]} \ardtwo \arltwo & W^{[3,2]} \ardtwo \arlthree & \ar[l,dotted,no head]
	\\ Y_2 \ar[d] & W^{[1,1]} \ar[d] \ar[l] & W^{[2,1]} \ar[d] \arltwo & W^{[3,1]} \ar[d] \arlthree & \ar[l,dotted,no head]
	\\ X  & Y_1 \ar[l]  & Y_1^{[2]} \arltwo  & Y_1^{[3]} \arlthree & \ar[l,dotted,no head]
\end{tikzcd}
\label{rbm.46}\end{equation}
where the left column and bottom row are the standard simplicial spaces as in
\eqref{rbm.42} and the interior spaces are 
given inductively by
\begin{equation}
	W^{[m,n]} = W^{[m,1]}\times_{Y_1^{[n]}} \times \cdots \times_{Y_1^{[n]}} W^{[m,1]} 
	\cong 
	W^{[1,n]}\times_{Y_2^{[m]}} \times \cdots \times_{Y_2^{[m]}} W^{[1,n]}.
\label{rbm.47}
\end{equation}
The result is a commutative diagram in which all rows and columns are
simplicial spaces. There are then two commuting simplicial differentials,
$d_1$ and $d_2$, corresponding to the horizontal and vertical maps, respectively.

\begin{defnstar}\label{rbm.60} A {\em bigerbe} on the bisimplicial space
  \eqref{rbm.46} corresponding to a locally split square \eqref{rbm.44} is
  specified by a (locally trivial) complex line bundle $L$ over $W^{[2,2]}$ with $d_1L$ and
  $d_2L$, over $W^{[3,2]}$ and $W^{[2,3]}$ respectively, having trivializing
    sections $s_i$, for $i=1,2$, such that $ds_i$ is the canonical
    trivialization of $d^2_iL$ and $d_2s_1=d_1s_2$.
\end{defnstar}
As for bundle gerbes, there are straightforward notions of products, inverses, pullbacks, and morphisms
of bigerbes, for which the characteristic 4-class defined below behaves naturally.

As noted above, among the natural examples is the Brylinski-McLaughlin
bigerbe. Suppose that $E \to X$ is a principal $G$-bundle over a manifold
with structure group a compact, connected, simply connected, simple Lie
group. 
Then
\begin{equation}
\begin{tikzcd}
 	E\ar[d]& P E\ar[l]\ar[d]\\
	X&\ar[l]P X
\end{tikzcd}
\label{rbm.49}\end{equation}
is a split square where $ P X$ and $ P E$ are the respective (based)
path spaces, the vertical arrows are projections and the horizontal arrows
are the end-point maps. In the resulting bisimplicial space $W^{[2,1]}=\Omega
E$ is the based loop space of $E$ which is a principal bundle with
structure group the based loop group $\Omega G$ of $G$. The central
extensions  
\begin{equation}
\begin{tikzcd}[sep=scriptsize]
	1 \ar[r] & \UU(1)\ar[r]&\widehat{\Omega G}\ar[r]&\Omega G \ar[r] & 1
\end{tikzcd}
\label{rbm.50}\end{equation}
are classified by $H^3(G;\bbZ) = H^3_G(G; \bbZ) =\bbZ$, and  the associated line bundle for such a
central extension pulls back over $W^{[2,2]}=\Omega E^{[2]}$ to a line
bundle $Q$ determining a bigerbe. Here the triviality of $d_1Q$ is the
multiplicativity of the central extension, as for a lifting gerbe, whereas
the (consistent) triviality of $d_2Q$ corresponds to the so-called `fusion' property of
the central extension with respect to certain configurations of loops \cite{Stolz-Teichner2005,MR3493404,KM-equivalence}, and which is equivalent to the gerbe property with respect to the path fibration $P E^{[2]} \to E^{[2]}$. 
Incorporation of an additional `figure-of-eight' condition as in
\cite{KM-equivalence,MR3391882} --- a condition related to the simplicial
space of products of $X$ discussed in \S\ref{S:prod_simp} --- promotes this
to a bigerbe involving free loops and paths, representing the obstruction
of the lift of $LE \to LX$ to a `loop-fusion' $\wh{LG}$-bundle which is
discussed further below.

There are various 2-gerbe versions of this in the literature.  In
\cite{B-MI}, Brylinski and McLaughlin define a 2-gerbe in the sense of
Breen by pulling back the canonical gerbe on $G$ (corresponding to the
given class in $H^3(G; \bbZ)$) to $E^{[2]}$ by the difference map,
particularly in the universal case where $X = BG$. In \cite{CJMSW}, and
later \cite{MR3055700}, a similar construction was used to produce a bundle
2-gerbe in the sense of Stevenson.  Furthermore, in \cite{B-MI}, the
authors discuss a correspondence between the 2-gerbe and the problem of
extending the structure group of the free loop space $LX$ from $LG$ to
$\wh {LG}$. The bigerbe above demonstrates this correspondence explicitly.

\medskip

Returning to the simplicial space, \eqref{rbm.42}, arising from any locally
split map, the simplicial differentials extend to the \v Cech cochain
spaces over the $Y^{[k]}$ \blue{with respect to a system of `admissible' covers,
as defined in \S\ref{S:covers} and \S\ref{S:cech}, which can be taken to be arbitrarily fine.}
Then the simplicial
complex
\begin{equation}
\begin{tikzcd}[sep=scriptsize]
	0 \ar[r] &\vC^\ell(X,A)\ar[r, "d"]&\vC^\ell(Y,A)\ar[r, "d"]&\vC^\ell(Y^{[2]},A)\ar[r, "d"]&\dots
\end{tikzcd}
\label{rbm.51}\end{equation}
is exact (see\ Proposition~\ref{T:exact_simplicial}), with a homotopy inverse arising from local sections over an open
cover.
The simplicial differential commutes with the \v Cech differential
resulting in a double complex.

For a bundle gerbe, the representative $c(L)$ of the Chern class of $L$ can be chosen to be a pure cocycle:
$\delta c(L)=dc(L)=0$. From the exactness of the simplicial differential
this class descends:
\begin{equation}
c(L)= - d\beta ,\quad \delta \beta =d\alpha
\quad \text{for some $\beta \in \vC^1(Y; \bbC^\ast)$, $\alpha \in \vC^2(X; \bbC^\ast)$}
\label{rbm.52}\end{equation}
and then $\DD(L)\in \vH^3(X;\bbZ)$, the image of $[\alpha
]\in\vH^2(X;\bbC^\ast)$ under the Bockstein isomorphism, is the Dixmier-Douady
class of the gerbe. This is not the original definition of the
Dixmier-Douady class of a bundle gerbe as in \cite{Murray1, MurrayStevenson};
we show that it is equivalent below in Proposition~\ref{P:Consistent} and use
the simplicial characterization to prove that a
locally split map $\pi : Y \to X$ supports a bundle gerbe with a given 3-class
on $X$ if and only if the class vanishes when pulled back to $Y$ (Theorem~\ref{T:representability_three_class}), and also to classify trivializations
of bundle gerbes (in Proposition~\ref{P:gerbe_triv}).

For a general bigerbe there is a similar \v Cech analysis in terms of the
triple complex, formed by the three commuting differentials $\delta$, $d_1$
and $d_2$, on the \v Cech spaces over the bisimplicial space
\eqref{rbm.46}. Now the Chern class $c(L) \in\vC^1(W^{[2,2]};\bbC^\ast)$ can again be
chosen to be a pure cocycle:
$ \delta c(L) =d_1c(L) =d_2c(L) =0$.
As a consequence it descends to a cocycle on $Y^{[2]}$: 
\begin{multline}
c(L)=d_2\beta_1,\quad d_1\beta_1=0,\quad \beta_1 \in\vC^1(W^{[2,1]};\bbC^\ast)\\
\delta \beta_1 = - d _2\lambda _1,\quad d_1\lambda_1=0,\quad \lambda_1 \in\vC^2(Y_1^{[2]};\bbC^\ast),
\label{rbm.54}\end{multline}
essentially as for the gerbe. Thus the image of $[\lambda_1]$ in $\vH^3(Y^{[2]};\bbZ)$ (under the Bockstein
isomorphism) is the Dixmier-Douady class of $L$ as a gerbe over $Y_1^{[2]}$.
Significantly  however, $\lambda _1$ is naturally a \emph{simplicial}
cocycle --- a pure cocycle in the $\delta,d_1$ complex --- and so $d\lambda_1 = 0$.

In view of this, the simplicial class further descends under $d_1$:
\begin{equation}
\lambda_1 =-d_1\mu_1,\quad \delta \mu_1 = d_1\gamma,\quad \delta \gamma
=0,\quad \mu_1\in\vC^2(Y_1;\bbC^\ast),\ \gamma \in\vC^3(X;\bbC^\ast).
\label{rbm.55}\end{equation}
The Bockstein image $G(L) \in \vH^4(X; \bbZ)$ of $[\gamma]$ is the characteristic 4-class
associated to the bigerbe.

The symmetry of the bigerbe allows $Y_1$ and $Y_2$ to be interchanged, but
this also reverses the sign of $G(L)$.

\begin{thmstar}[Thm~\ref{T:bigerbe_DD_class}, Thm.~\ref{T:bigerbe_representability}]
\mbox{}
\begin{enumerate}
[{\normalfont (i)}]
\item The characteristic 4-class of a bigerbe is natural with respect to pullbacks, morphisms, products, and inverses. It vanishes if and only if the bigerbe admits a trivialization, and two bigerbes have the same 4-class if and only if they are stably isomorphic.
\item 
The bisimplicial space generated by a split
  square, as in \eqref{rbm.44}, supports a bigerbe for a given class in
  $\vH^4(X;\bbZ)$ if and only if this class lifts to the $Y_i$ to be trivial,
  with primitives which when pulled back to $W$ have exact difference.
\end{enumerate}
 \end{thmstar}

For the Brylinski-McLaughlin bigerbe $Q$ associated to a principal bundle $E \to X$ with
structure group a compact, connected, simply connected and simple Lie group $G$,
the 4-class $G(Q) = [\gamma]$ is the 
transgression to $X$ of the 3-class on $G$ corresponding to
a central extension $\wh {LG}$ of $L G$:
\begin{equation}
\begin{gathered}
\begin{tikzcd}
	X&\ar[l] E&\arltwo E^{[2]}\ar[r,"q"]&G,
\end{tikzcd}
\\
\ \alpha  \in H^3(G;\bbZ), \quad 
q^*\alpha =d\mu,\quad d\gamma  =\delta \mu.
\end{gathered}
\label{rbm.57}\end{equation}
\begin{thmstar}[Thm.~\ref{T:BM_bigerbe}, Thm.~\ref{T:lf_lifts_trivns}]
  The Brylinski-McLaughlin bigerbe $Q \to LE^{[2]}$ has characteristic class $G(Q) = p_1(E) \in \vH^4(X; \bbZ)$, the vanishing of which is equivalent to the existence of a `loop-fusion' (meaning fusion and figure-of-eight, see \S\ref{S:prod_simp}) $\wh{LG}$ lift  of the $LG$-bundle $LE \to LX$.
Such lifts, which are equivalent to certain trivializations of $Q$, 
are classified by $\vH^3(X; \bbZ)$.
\end{thmstar}

In particular this applies to the spin frame bundle of a spin manifold.
There it represents the obstruction to a lift of the principal loop spin bundle
over the loop space to a loop-fusion principal bundle for the basic
central extension of the loop spin group. The obstruction is then the
Pontryagin class, usually denoted $\frac12 p_1$ because of its relation to
the Pontryagin class of the oriented orthogonal frame bundle, of the spin
bundle \cite{MR3493404,CJMSW,KM-equivalence}.
Such loop-fusion lifts are, by the above theorem, in bijection with so-called
`string structures' on the manifold (see Corollary~\ref{C:string_loop_spin}).

In addition to the Brylinski-McLaughlin bigerbes, we provide other natural
examples of bigerbes representing `decomposable' 4-classes which are the
cup product of either 2-classes or a 1-class and a 3-class (see
\S\ref{S:decomp_bigerbes}). 
Moreover we show that for a simply connected and locally contractible space
$X$, every 4-class is represented by a bigerbe with respect to the locally
split square in which the $Y_i$ are the based path spaces $PX$ and $W = PPX$ is
the based mapping space of the square into $X$ (see \S\ref{S:Path-bigerbes}).
In particular $K(\bbZ, 4)$ supports a universal bigerbe.
Likewise, for $X$ not necessarily simply connected, we show that every 4-class
is represented by a bigerbe using the free path spaces; in particular
$W^{[2,2]} = L L X$ is the double free loop space in this case.

There is a direct extension of bigerbes to `multigerbes',
higher versions of (bi)gerbes in which the locally split squares are replaced by $n$-cubes,
the line bundles satisfy simplicial conditions with respect to $n$ commuting differentials;
these represent cohomology classes of degree $n+2$. 
The symmetry of the (multi)simplicial conditions replaces the ever higher and more complex
`associativity' conditions associated to higher gerbes.
This extension is quite straightforward, and for this reason, and since we are unaware
of examples apart from decomposable and path multigerbes, we only outline the theory
briefly at the end of this paper.

In order to restrict attention to a simple category of topological spaces,
and to avoid expanding the paper further, we do not develop the theory of
connections on bigerbes here, though this will be done in a future work. We
also do not discuss here the bigerbe analogue of bundle gerbe modules or
the related theory of generalized morphisms due to Waldorf
\cite{waldorf2007more}.

Section~\ref{S:cechtheory}
below contains a discussion of covers and locally split maps,
which is the context for the rest of the paper, and our notation for \v Cech
theory is introduced in \S\ref{S:cech}.

Section~\ref{S:bundle_gerbes_sec} is a review of Murray's theory of bundle
gerbes (without connections), as developed from the \v Cech-simplicial point of
view, with the basic properties of bundle gerbes over split maps recalled in
\S\ref{S:bundle_gerbes}. 
The extension of the \v Cech cohomology complex to a
bicomplex over the simplicial space of a split map in \S\ref{S:simplicial_cech}
leads to an alternative definition of the Dixmier-Douady class for a bundle gerbe
in \S\ref{S:DD_class}, the classification of gerbes over a given split map in
\S\ref{S:rep_of_three}, and the classification of trivializations in
\S\ref{S:gerbe_triv}. 
Examples of bundle gerbes are recalled in
\S\ref{S:gerbe_exs}. 

A `product-simplicial' condition, which we refer to as `doubling', on bundle
gerbes is defined in \S\ref{Simp.gerbe} with particular application to the
free loop space, and the connection with results from \cite{MR3391882} is
described in \S\ref{S:lf_cohom}.

In \S\ref{S:lssquares} the basic properties of locally split squares of
maps are given, leading to the definition of bigerbes in
\S\ref{S:bigerbes}. The characteristic 4-class of a bigerbe is obtained in
\S\ref{S:bigerbe_four_class} and conversely \S\ref{S:rep_four_class}
contains a necessary and sufficient condition for representability of a
4-class over a given locally split square. 

Section \S\ref{S:bigerbe_examples} is devoted to examples, with explicit
bigerbes corresponding to decomposable classes constructed in
\S\ref{S:decomp_bigerbes}, extending some of the results of \cite{MR2674880}.
After a brief discussion of doubling for bigerbes in \S\ref{S:prod_simp_bigerbe},
our main application of bigerbes --- the Brylinski-McLaughlin
lifting gerbe --- is discussed in \S\ref{S:bm-bigerbe} and its relation
to string structures is discussed in \S\ref{S:aside}.
Further examples of path bigerbes can be found in \S\ref{S:Path-bigerbes},
and finally, we end with a brief account of multigerbes in \S\ref{S:multi}.

\subsection*{Acknowledgements} 
This material is based in part on work supported by the National Science Foundation
under Grant No.\ DMS-1440140 while the authors were in residence at the
Mathematical Sciences Research Institute in Berkeley, California, during the
Fall 2019 semester.
The first author was also supported under NSF grant DMS-1811995, and the authors would like to
acknowledge helpful conversations during the development of this material with
Konrad Waldorf, Mathai Varghese, and Ezra Getzler.

\blue{
The authors are grateful to Andr\'e Henriques for pointing out errors in the published version
of the paper which are corrected by this version, in particular by the systematic use of admissible covers throughout.}
\egroup 
\section{\v Cech theory}\label{S:cechtheory}
\subsection{Covers and locally split maps}\label{S:covers}

Since we will make substantial use of \v Cech theory, we start with some
conventions on spaces, open covers, and maps.
We work throughout in the standard topologist's category of compactly generated Hausdorff spaces
and continuous maps, without additional conditions unless explicitly noted.

An open cover of a topological space $X$ is a collection of open sets,
$\ecU$, for which $X = \bigcup_{U \in \ecU} U$.  Such a cover defines an {\em \'etale
space} by taking the disjoint union
\[ 
	\Et{\ecU} = \bigsqcup_{U \in \ecU} U\to X
\]
with the map to $X$ consisting of the inclusion map on each $U$. 

Note that
since the individual sets may not be connected, it is not generally possible 
to recover the collection $\ecU$ from $\Et{\ecU}$ without specifying additional
indexing information.
We regard $X$ as its own minimal cover.

If $\ecU$ and $\ecV$ are covers of $X$ and $Y$ then a {\em map of covers} is
a continuous map $g : \Et{\ecU} \to \Et{\ecV}$ where each element $U \in \ecU$
is mapped to a specific element $V \in \ecV$. In particular, there is
an underlying map of index sets.

\begin{defn}\label{rbm.72} A continuous map $\pi : Y \longrightarrow X$ of
  topological spaces is {\em locally split} if it admits continuous local
  sections; thus $\pi$ is surjective and there exists a cover $\ecU$ of $X$
  with respect to which the local sections constitute a continuous map of
  covers $s : \Et{\ecU} \to Y$ such that $\pi\circ s : \Et{\ecU} \to X$ is inclusion of
  the cover in $X.$ In particular the {\em inclusion map} of covers 
  $\Et{\ecU}\longrightarrow X$ 
is itself
 locally split.
\end{defn}

If $\ecU$ and $\ecV$ are covers of the same space $X$
then a map of covers 
such that
\begin{equation}
\begin{tikzcd}
\Et{\ecV}\ar[r] \ar[dr] & \Et{\ecU}
    \ar[d] \\ &X \end{tikzcd}
\label{E:refinement} \end{equation}
commutes makes $\ecV$ a {\em refinement} of $\ecU$; often in the literature
the underlying map of index sets is omitted but we always retain it, even if
implicitly. In this way the covers of $X$ define a 
category with refinements as morphisms.
Observe also that if $\ecV$ is a cover of $\Et{\ecU}$ considered as a
space, then $\ecV$ is a cover of $X$ as well, the composite map $\Et{\ecV}\to X$ is an inclusion map of covers, and $\ecV$
constitutes a refinement of $\ecU$.

If $f : Y \to X$ is a continuous map then the pullback
\begin{equation}
	f^\inv \ecU = \set{f^\inv(U) : U \in \ecU}
	\label{E:pullback_fixed_cover}
\end{equation}
is a cover of $Y$ and $f$ lifts to a well defined map of covers $f:\Et{f^\inv \ecU} \cong f^\inv \Et{\ecU}
\to \Et{\ecU}$. If $\ecU$ and $\ecV$ are both covers of $X$, then
\begin{equation}\label{E:cover_mutual_refinement}
	\ecU \cap \ecV= \set{U \cap V : U \in \ecU, V \in \ecV}
	\end{equation}
is a cover of $X$ mutually refining $\ecU$ and $\ecV$; indeed,  
this is the same thing as pulling back $\ecV$ to a cover of $\Et{\ecU}$ by the
inclusion map to $X$ or vice versa, and $\Et{\ecU \cap \ecV} \cong \Et{\ecU} \times_X \Et{\ecV}$.
%
%
Note that 
\[
	\ecU^{(\ell)} 
	= \underbrace{\ecU \cap \cdots \cap \ecU}_{\text{$\ell$ times}}
\]
gives the cover by $\ell$-fold intersections of sets in $\ecU$ (we refer to $\Et{\ecU^{(\ell)}}$ as a {\em \v Cech space})
and that there is a canonical identification 
\begin{equation}
	f^\inv (\ecU^{(\ell)}) \cong (f^\inv \ecU)^{(\ell)} 
	\label{E:pullback_fixed_cover_ell}
\end{equation}
of covers over $Y$ whenever $\ecU$ is a cover of  $X$ and $f : Y \to X$ is continuous.

\subsection{{\v Cech} cohomology} \label{S:cech}

The standard definition of \v Cech cohomology proceeds by fixing an open
cover $\ecU$ of $X$ and taking the homology $\vH^\bullet_{\ecU}(X; A)$ of the
cochain complex $(\vC^\bullet_{\ecU}(X; A),\delta) = (\Gamma(\ecU^{(\bullet
  +1)}; A),\delta)$, where $A$ is a sheaf of abelian groups on $X$,
$\Gamma(\ecU^{(\bullet+1)}; A)$ denotes the group of local sections of $A$
on the intersection cover $\ecU^{(\bullet+1)}$  and $\delta :
\Gamma(\ecU^{(\bullet)}; A) \to \Gamma(\ecU^{(\bullet+1)}; A)$ is given by
the alternating sum of the pullbacks by the various inclusion maps
$\ecU^{(\bullet+1)} \to \ecU^{(\bullet)}$.  
For our purposes, $A$ will always be a fixed topological abelian group
such as $\bbC$, $\bbC^\ast$, $\bbZ$, or $\UU(1)$ and we will work on the
sheaf of continuous maps to $A$, so that 
\[
	\vC^\bullet_{\ecU}(X; A) = \cC(\Et{\cU^{(\bullet+1)}}; A)
\]
is a space of continuous maps to $A$ from the \'etale space.
The full \v Cech cohomology is defined as the direct limit
\begin{equation}
	\vH^\bullet(X; A) = \lim_{\ecU} \vH^\bullet_{\ecU}(X; A)
\label{rbm.23}\end{equation}
under refinement.
\blue{
Here the limit is taken with respect to the directed set of covers of $X$, where $\ecU < \ecV$ if there exists
a refinement $\Et{\ecU} \to \Et{\ecV}$ inducing a pullback $\vC^\bullet_{\ecV}(X; A) \to \vC^\bullet_{\ecU}(X; A)$; while there may be many such distinct refinements, the 
associated pullbacks are chain homotopic and the
induced maps on cohomology agree.
}
We use the standard terminology of cochains, cocycles and coboundaries for elements of $\vC^\bullet(X; A)$,
and also borrow the adjectives {\em closed} and {\em exact} from de Rham theory
for cocycles and coboundaries, respectively.

\begin{added}
\begin{prop}\label{P:Cech_map}
Fix a topological abelian group $A$, and suppose a continuous map
  $f:Y \to X$ of topological spaces 
lifts to a map of covers $\wt f : \Et{\ecV} \to \Et{\ecU}$ for covers $\ecV$ and $\ecU$ of $Y$ and $X$, respectively. 
Then this
induces a chain map
\[
	\wt f^\ast : \vC_{\ecU}^\bullet(X; A) \to \vC_{\ecV}^\bullet(Y; A)
\]
which descends to the pull-back functor $f^\ast : \vH^\bullet(X; A) \to
\vH^\bullet(Y; A)$ on cohomology. 
\end{prop}

\begin{proof}
The map of covers induces a map of covers $\Et{\ecV^{(\ell)}} \to \Et{\ecU^{(\ell)}}$ for each $\ell$, giving 
a chain map $\vC^\ell_{\ecV}(Y; A) \to \vC^\ell_{\ecU}(X; A)$, and 
the natural maps $\Et{\ecV^{(\ell)}} \to \Et{\ecU^{(\ell)}}$ commute with the inclusions
$\Et{\ecU^{(\ell)}} \hookrightarrow \Et{\ecU^{(\ell-1)}}$
in each factor of the $\ell$-fold intersections into the original covers,
so $\wt f^\ast$ commutes with the \v Cech differential $\delta$ and defines a chain morphism.
Furthermore, 
this is well-defined with respect to refinement; specifically, if $\Et{\ecU'} \to \Et{\ecU}$ 
and $\Et{\ecV'} \to \Et{\ecV}$ are refinements such that $\wt f$ lifts to a map of covers $\wt f : \Et{\ecV'} \to \Et{\ecU'}$
then 
\[
\begin{tikzcd}
	\Et{{\ecV'}^{(\ell)}} \ar[r, "\wt f"] \ar[d]& \Et{{\ecU'}^{(\ell)}} \ar[d] 
	\\ \Et{{\ecV}^{(\ell)}} \ar[r, "\wt f"] & \Et{\ecU^{(\ell)}}	
\end{tikzcd}
\]
commutes, descending in the direct limit in cohomology to the map $f^\ast : \vH^\bullet(Y; A) \to \vH^\bullet(X; A)$.
\end{proof}
\end{added}
We proceed to define a more limited form of pullback on chains with respect to the sections
of a locally split map.
This is essential to the exactness of the simplicial complex \eqref{rbm.51} of \v Cech chains
on the simplicial space induced from a locally split map mentioned in the introduction and
discussed in detail in \S\ref{S:simplicial_cech}.
\blue{
\begin{defn}
For a locally split map $\pi : Y \longrightarrow X$,
an \emph{admissible pair of covers} consists of covers $\ecU$ of $X$ and $\ecV$ of $Y$
along with maps of covers $\wt \pi : \Et{\ecV} \to \Et{\ecU}$ and $\wt s : \Et{\ecU} \to \Et{\ecV}$
such that 
\begin{enumerate}
[{\normalfont (i)}]
\item $\wt \pi$ is a lift of $\pi$,
\item $\wt s$ is a lift of a local section $s: \Et{\ecU} \to Y$,
\item $\wt s$ is a section (right inverse) of $\wt \pi$, i.e. $\wt \pi \wt s = \Id_\ecU : \Et{\ecU} \to \Et{\ecU}$, and
\item
\label{I:admissible_pair_int_prop}
for each $U_1,U_2 \in \cU$ such that $U_1 \cap U_2 \neq \emptyset$, there exists $V_2 \in \wt\pi^\inv(U_2) \subset \cV$ such that $\wt s_{U_1} : U_1 \cap U_2 \to V_2$, i.e., the section
defined on $U_1$ restricts over $U_2$ to a map into some $V_2$ over $U_2$.
\end{enumerate}
The key consequence of \eqref{I:admissible_pair_int_prop} is that for each $\ell \geq 1$ there exist canonical 
sections $\wt s_\ell : \cU^{(\ell)} \to \cV^{(\ell)}$ of $\wt \pi : \cV^{(\ell)} \to \cU^{(\ell)}$ defined by 
\begin{equation}
	\wt s_\ell = \wt s_{U_1} : U_1 \cap U_2 \cap \cdots \cap U_\ell \to V_1 \cap V_2 \cap \cdots \cap V_\ell
	\label{E:s_ell}
\end{equation}
so that $\cU^{(\ell)}$ and $\cV^{(\ell)}$ also constitute an admissible pair.
\label{D:admissible_pair}
\end{defn}

\begin{lem}\label{Addedrbm} 
If $\pi : Y \to X$ is a locally split map with arbitrary covers $\ecU$ of $\ecV$ of $X$ and $Y$, respectively,
admitting a section $s : \Et{\ecU} \to Y$,
then there exists an admissible pair $\ecU'$ and $\ecV'$ refining $\ecU$ and $\ecV$ such that $\wt s$ lifts $s$.

In fact, if $\pi_i : Y_i \to X$, $i = 1,\ldots,N$, is a finite number of locally split maps over $X$, then 
any covers $\ecU$ of $X$ and $\ecV_i$ of $Y_i$ can together be refined to $\ecU'$ and $\ecV'_i$, respectively, constituting an admissible pair for each $i$
simultaneously, with $\wt s_i$ lifting any given sections $s_i : \Et{\ecU} \to Y_i$.
\end{lem}
\begin{rmk}
Since admissible covers may be taken to refine arbitrary covers of $X$ and $Y$,
it follows that covers participating in admissible pairs are 
\emph{final} in the directed set of all covers; 
in particular the direct
limits $\vH^\ell(X; A) = \lim_{\ecU}\vH_\ecU^\ell(X; A)$ 
and $\vH^\ell(Y; A) = \lim_{\ecV} \vH^\ell_\ecV(Y; A)$
are equivalent to the direct
limit taken over admissible covers alone \cite[Thm.\ 1, p.\ 213]{maclane}.
Equivalently, any class in the cohomology of $X$ or $Y$ may be represented as a \v Cech 
cochain with respect to a cover which is part of an admissible pair.
\end{rmk}

\begin{proof} 
We construct the refinement in two steps, first `pulling down' $\ecV$ to get
$\ecU' = s^\inv \ecV$ and then `pulling up' $\ecU'$ to get $\ecV' = \pi^\inv \ecU' \cap \ecV$.
More precisely,
if
$s_U:U\longrightarrow Y$ is
  the section of $\pi$ over $U\in \ecU$ then the sets $s_U^\inv V\subset U$ for
  $V \in \ecV$, define a cover of $U$, and hence together a refinement of $\ecU$ as a
  cover of $X$. This is the cover $\ecU'=s^{-1}\ecV$ defined above. Then $s$ lifts to a
  naturally defined map of covers $\wt s: \Et{\ecU'}\to \Et{\ecV}$ given by $s_U :s_U^\inv
  V\to V$. 
The pull-back $\pi^{-1}\ecU'$ of the new cover of $X$ defines a
  refinement $\ecV'=\pi^{-1}\ecU'\cap\ecV$ of $\ecV$ to which $\pi$ lifts as a
  map of covers $\wt \pi : \Et{\ecV'} \to \Et{\ecU'}$.
Moreover, as a consequence of the fact that $\pi s = \Id$ on each element $U \in \ecU'$,
it follows that $\wt s$ lifts to a map of covers
$\wt s:\Et{\ecU'}\to \Et{\ecV'}$, and $\wt \pi \wt s = \Id$.

To see that this satisfies the final condition of Definition~\ref{D:admissible_pair},
note that 
by definition each $U_1'=U_1\cap s_{U_1}^{-1}(V)$ for some
  elements $U_1 \in \ecU$ and $V\in \ecV$, respectively. Then $V'_1=V\cap
  \pi^{-1}(U_1')$ contains the image of $U'_1$ under the lifted section and
  the other elements $V'_2=V\cap\pi^{-1}(U'_2)$ contain the image of
  $U_1'\cap U_2'$ under $s_{U'_1}.$ 

Given a finite number of locally split maps $\pi_i : Y_i \to X$, we set $\ecU' = \bigcap_i s_i^\inv \ecV_i$,
which refines each $s_i^\inv \ecV_i$ and hence lifts the map of covers $\wt s_i$ defined above,
followed by $\ecV_i' = \pi_i^\inv \ecU' \cap \ecV$ for each $i$.
\end{proof}

In preparation for \S\ref{S:simplicial_cech} we make the following observations.
\begin{lem}
Denote by $Y^{[k]} = Y\times_X \cdots \times_X Y$ the $k$-fold fiber product of $\pi : Y \to X$
and fix an admissible pair of covers $(\ecU,\ecV)$ for $(X,Y)$.
Then for each $k \geq 2$,
\begin{enumerate}
[{\normalfont (i)}]
\item 
$\ecV^{[k]} := \ecV\times_{\ecU} \cdots \times_{\ecU} \ecV$, the $k$-fold fiber product of $\wt \pi : \ecV \to \ecU$,
is a cover of $Y^{[k]}$,
\item
The projection maps $\pi_j : Y^{[k]} \to Y^{[k-1]}$, for $0 \leq j \leq k-1$, lift canonically to maps of covers
$\wt \pi_j : \Et{(\ecV^{[k]})^{(\ell)}} \to \Et{(\ecV^{[k-1]})^{(\ell)}}$ for each $\ell$,
\item 
\label{I:admissible_cover_fib_prod_sections}
The sections $\wt s_\ell$ of Definition~\ref{D:admissible_pair}.\eqref{I:admissible_pair_int_prop} determine a map of covers 
$\wt s_\ell : \Et{(\ecV^{[k-1]})^{(\ell)}} \to \Et{(\ecV^{[k]})^{(\ell)}}$ for each $\ell$ such that
\begin{equation}
	\wt \pi_j \wt s_\ell = \begin{cases} \Id & j = 0 \\ \wt s_\ell \wt \pi_{j-1} & 1 \leq j \leq k-1 \end{cases}
	\label{E:simplicial_chain_contraction}
\end{equation}
\end{enumerate}
\label{L:admissible_cover_fib_prod}
\end{lem}

\begin{proof}
It is clear that $\cV^{[k]}$ is an open cover of $Y^{[k]}$.
Note that there is a canonical identification $(\cV^{[k]})^{(\ell)} \cong (\cV^{(\ell)})^{[k]}$, the latter fiber product taken over $\cU^{(\ell)}$; 
indeed, for $k = \ell = 2$ this amounts to observing that 
\[
	(V_1 \times_{U_1} V'_1) \cap (V_2 \times_{U_2} V'_2) \cong (V_1 \cap V_2) \times_{U_1 \cap U_2} (V'_1 \cap V'_2)
\]
are both identified with pairs of points $(y,y')$ such that $\pi(y) = \pi(y') \in U_1 \cap U_2$ with $y \in V_1 \cap V_2$ and $y' \in V'_1 \cap V'_2$,
with similar identification holding in general.

With this observation, the lifts $\wt \pi_j$ are simply the fiber product projections $(\cV^{(\ell)})^{[k]} \to (\cV^{(\ell)})^{[k-1]}$, defined pointwise by
$(y_0,\ldots,y_{k-1}) \mapsto (y_0,\ldots, \widehat{y_j},\ldots,y_{k-1})$, with $\wh \cdot$ denoting omission of the corresponding variable.
The section $\wt s_1 :\Et{\cV^{[k-1]}} \to \Et{\cV^{[k]}}$ is defined by $\wt s \circ \wt \pi \times 1$ (where $\wt \pi : \cV^{[k-1]} \to \cU$ and $\wt s : \cU \to \cV$) acting pointwise by 
\[
	(y_0,\ldots,y_{k-2}) \mapsto \big(\wt s(\wt \pi(y_0)),y_0,\ldots,y_{k-2}\big)
\]
and $\Et{(\cV^{[k-1]})^{(\ell)}} \to \Et{(\cV^{[k]})^{(\ell)}}$ is defined similarly by $\wt s_{\ell}\circ \wt \pi \times 1$, using the identification $(\cV^{[k]})^{(\ell)} \cong (\cV^{(\ell)})^{[k]}$,
from which \eqref{E:simplicial_chain_contraction}
follows directly.
\end{proof}

\begin{lem}
Given an arbitrary cover $\cO$ of $Y^{[k]}$, there exists an admissible pair of covers $(\cU, \cV)$ for $\pi : Y \to X$ such that $\cV^{[k]}$ refines $\cO$.
Thus the covers of the form $\cV^{[k]}$ for admissible $(\cU,\cV)$ are final in the directed set of all covers on $Y^{[k]}$, and in particular every cohomology class 
in $\vH^\bullet(Y^{[k]}; A) = \lim_{\cO} \vH^\bullet_\cO(Y^{[k]}; A)$ has a representative in $\vC^\bullet_{\cV^{[k]}}(Y^{[k]}; A)$ for some admissible pair.
\label{L:refine_Yk}
\end{lem}
\begin{proof}
Since $Y^{[k]}$ is topologized as a subspace of the product $Y^k$,
for each $O \in \cO$ there is some open $\wt O \subset Y^k$ in the product topology such that $O = \wt O \cap Y^{[k]}$, and $\wt O$ is a union of sets of the form $O_1 \times \cdots \times O_k$
for open $O_i \subset Y$.
Taking the set of all such $O_i$ over $O \in \cO$ gives a cover $\cV'$ of $Y$ such that $\cV'\times_X \cdots \times_X \cV'$ is a refinement of $\cO$, and then invoking
Lemma~\ref{Addedrbm} produces an admissible pair such that $\cV^{[k]} = \cV\times_\cU \cdots \times_\cU \cV$ refines $\cV'\times_X \cdots \times_X \cV'$ and therefore $\cO$.
\end{proof}
}

\section{Bundle gerbes} \label{S:bundle_gerbes_sec}
\subsection{Simplicial line bundles} \label{S:bundle_gerbes}

We recall the notion of a bundle gerbe \cite{Murray1}, which for our purposes
is most efficiently defined in terms of simplicial line bundles.

Denote by $Y^{[k]}$ the $k$-fold fiber product $Y\times_{\pi} \cdots
\times_{\pi} Y$, with projection maps $\pi_j : Y^{[k]} \to Y^{[k-1]}$, $j =
0,1,\ldots,k-1$, where $\pi_j(y_0,\ldots,y_{k-1}) =
(y_0,\ldots,\wh{y_j},\ldots, y_{k-1})$ omits the $j$th factor
enumerated from $0$.
Then 
\begin{equation}
\begin{tikzcd}[column sep=small]
	X&\ar[l]Y & Y^{[2]} \arltwo & Y^{[3]} \arlthree & \ar[l,thick,dotted, no head]
\end{tikzcd}
\label{rbm.73}\end{equation}
is a {\em simplicial space} with face maps $\pi_j$ and degeneracy maps
the fiber diagonal maps $Y^{[k-1]} \to Y^{[k]}$ (of which we will not make use).
%
More precisely, $Y^{[\bullet]}$ is a {\em simplicial space over $X$}, meaning
that all maps commute with the projections $\pi : Y^{[k]} \to X$, and $X$ itself
may be regarded as an augmentation in \eqref{rbm.73}.
For notational convenience, we set $Y^{[1]} = Y$ and $Y^{[0]} = X$, with $\pi_0
= \pi : Y^{[1]} \to Y^{[0]}$.

\begin{rmk} Our enumeration (which is geometrically natural here)
differs unfortunately from the standard convention for simplicial spaces, under which one
would typically write $Y_0 = Y$ (as the image of the $0$ simplex), $Y_1 =
Y^{[2]}$ (as the image of the 1-simplex), etc., augmented by $Y_{-1} = X$.
For consistency we use this alternative convention throughout, and beg the pardon of
readers who would prefer to use the standard one.
\end{rmk}

Given a complex line bundle $L \to Y^{[k]}$, its differential is defined to be the line bundle
\begin{equation}
	d L := \bigotimes_{i=0}^{k} \pi_i^\ast L^{(-1)^i} \to Y^{[k+1]}.
	\label{E:line_bundle_differential}
\end{equation}
Using the commutation relations between the $\pi_j$, it follows that $d^2 L = d(dL)$ 
is canonically trivial over $Y^{[k+2]}$.

\begin{rmk}
While we will mostly work with complex line bundles, we could equivalently take
$L$ to be a principal $\bbC^\ast$ or $\UU(1)$ bundle. 
At times we will use these objects interchangeably without further elaboration.
\end{rmk}

A {\em bundle gerbe} $(L,Y,X)$ as defined by Murray is equivalent to a {\em
  simplicial line bundle} on the simplicial space $Y^{[\bullet]}$ in the
sense of Brylinski and McLaughlin \cite{B-MI}; this consists of a complex
line bundle $L \to Y^{[2]}$ along with a trivialization of the bundle
\[
	d L = \pi_0^\ast L \otimes \pi_1^\ast L^{-1}\otimes \pi_2^\ast L
\]
over $Y^{[3]}$, which in turn induces the canonical trivialization of $d^2L$
when pulled back over $Y^{[4]}$. The trivialization of $dL \to Y^{[3]}$
is equivalent to the `gerbe (or groupoid) product' isomorphism
\[
	\phi : \pi_2^\ast L \otimes \pi_0^\ast L \stackrel \cong \to \pi_1^\ast L,
\]
which multiplies (composes) pairs of respective elements in the fibers
$L_{y_0,y_1}$ and $L_{y_1,y_2}$ to get elements in $L_{y_0,y_2}$; the
condition that the trivialization coincide with the canonical one on $d^2\,
L$ over $Y^{[4]}$ is equivalent to associativity of this product.

A bundle gerbe $(L,Y,X)$ is {\em trivial} if there exists a bundle $L' \to Y$
and an isomorphism $L \cong dL'$ on $Y^{[2]}$; such an isomorphism is called a {\em trivialization} of $L$. 

If 
$(L, Y, X)$ 
is a bundle gerbe on $X$, then its {\em pullback} by a continuous map $f : X' \to X$ is the bundle gerbe $(\wt f^\ast L, f^\ast Y, X')$;
here we use the naturality $f^\ast(Y\times_X Y) \cong f^\ast(Y) \times_{X'} f^\ast(Y)$ and denote the resulting
map $f^\ast Y^{[2]} \to Y^{[2]}$ by $\wt f$. 
%
Likewise, the {\em product} of two bundle gerbes $(L_i, Y_i, X)$, $i = 1, 2$ on $X$ is given by
\[
	(L_1 \otimes L_2, Y_1 \times_X Y_2, X)
\]
where $L_1 \otimes L_2$ is shorthand for $\pr_1^\ast L_1 \otimes \pr_2^\ast L_2$ with
$\pr_i : Y_1 \times_X Y_2 \to Y_i$ denoting the projections from the fiber product.
It is straightforward to verify that this is a bundle gerbe, which we denote for simplicity as $L_1\otimes L_2$.
The definitions of product and pullback implicitly use the following
standard result which we record for later use.
\begin{lem}
Pullbacks and fiber products of locally split maps are locally split. More precisely, if $\pi : Y \to X$
is locally split and $f : X' \to X$ is continuous, then $f^\ast Y \to X'$ is locally split, and if $\pi' : Y' \to X$
is another locally split map, then $\pi\times\pi' : Y\times_X Y' \to X$ is locally split.
\label{L:ls_pullback_prod}
\end{lem}

More generally, a {\em (strong) morphism} $(L',Y',X') \to (L, Y,X)$ of
bundle gerbes consists of a map $\wt f : Y' \to Y$ covering a map $f : X'\to X$ 
along with an
isomorphism $L \cong L'$ over $\wt f^{[2]} : {Y'}^{[2]} \to Y^{[2]}$ which
intertwines the sections of $dL$ and $dL'$; a {\em (strong) isomorphism} is
a morphism for which $X = X'$ and $f = \Id$.
In particular, a morphism $f : (L', Y', X') \to (L, Y,X)$ is equivalent to an isomorphism $(L', Y', X') \cong (f^\ast L, f^\ast Y,X')$.

Finally, two gerbes 
$(L_i, Y_i, X)$, $i = 1,2$ 
are said to be {\em stably isomorphic} if $L_1\otimes L_2^\inv$ is trivial, or equivalently, there exist trivial gerbes $(T_i, Z_i, X)$ such that
\[
	L_1\otimes T_1 \cong L_2 \otimes T_2,
\]
in the sense of a strong isomorphism over a space $Z^{[2]}$ where $Z \to X$
admits maps to $Y_1$, $Y_2$, $Z_1$, and $Z_2$. 
This is strictly weaker than an isomorphism as defined above, and was introduced in \cite{MurrayStevenson} in order
to obtain a classification of bundle gerbes up to stable isomorphism by their Dixmier-Douady class.

\begin{rmk}
There is a weaker notion of gerbe morphism due to Waldorf \cite{waldorf2007more}, which naturally incorporates the theory of gerbe modules,
and has the property that the invertible morphisms are precisely the stable isomorphisms; moreover a trivialization
becomes the same thing as an isomorphism to a canonical trivial gerbe over $X$.
Because we leave the generalization to bigerbes of Waldorf's morphisms to a future work, we will not pursue this further
here.
\end{rmk}

\subsection{Simplicial \v Cech theory} \label{S:simplicial_cech}

A primary motivation for the definition of gerbes is that they represent a
class in $\vH^3(X; \bbZ) \cong \vH^2(X; \bbC^\ast)$ --- the
Dixmier-Douady class; for a bundle gerbe $(L,Y,X)$ this will be denoted
by $\DD(L)$. This class is natural with respect to products, pullbacks, and
inverses, and it determines $(L,Y,X)$ up to stable isomorphism
\cite{MurrayStevenson}. We give an alternative (though not necessarily simpler)
derivation of these facts based on a closer study of simplicial spaces; this
approach is used in the generalization to bigerbes below. In doing so we
identify the 3-classes on $X$ which can be represented by a bundle gerbe with
respect to a given locally split map $Y \to X$, and recover the classification
of the trivializations of a trivial bundle gerbe.

Consider the simplicial space, \eqref{rbm.73}, consisting of the fiber
products of a locally split map $\pi : Y \to X$,
\blue{
and fix an admissible pair of covers $(\ecU,\ecV)$.
Proposition~\ref{P:Cech_map} and Lemma~\ref{L:admissible_cover_fib_prod}} show that the induced maps $\pi_j^\ast :
\vC_{\blue{\ecV^{[k]}}}^\bullet(Y^{[k]}; A) \to \vC_{\blue{\cV^{[k]}}}^\bullet(Y^{[k+1]}; A)$ are chain maps and
the {\em simplicial differential} on \v Cech cochains is defined by
\[
d = \sum_{j=0}^{k} (-1)^j \pi_j^\ast : \vC_{\blue{\ecV^{[k]}}}^\bullet(Y^{[k]}; A) \to \vC_{{\ecV^{[k+1]}}}^\bullet(Y^{[k+1]}; A),
\]
so $d^2 = 0$ and $d \delta = \delta d.$ 

Thus, $(\vC^\bullet_{\blue{\ecV^{[\bullet]}}}(Y^{[\bullet]}; A), d, \delta)$ forms a double complex 
\begin{equation}
\begin{tikzcd}[sep=small]
	 {} & {} & {} & 
	\\\vC^0(Y^{[2]}) \ar[r,"\delta"]\ar[u,"d"] &\vC^1(Y^{[2]}) \ar[r, "\delta"]\ar[u,"d"] & \vC^2(Y^{[2]}) \ar[r,"\delta"]\ar[u,"d"] & {}
	\\\vC^0(Y) \ar[r,"\delta"]\ar[u,"d"] &\vC^1(Y) \ar[r, "\delta"]\ar[u,"d"] & \vC^2(Y) \ar[r,"\delta"]\ar[u,"d"] & {}
	\\\vC^0(X) \ar[r,"\delta"] \ar[u,"d"]&\vC^1(X) \ar[r, "\delta"] \ar[u,"d"]& \vC^2(X) \ar[r,"\delta"]\ar[u,"d"] & {}
	\\0 \ar[u,"d"]& 0 \ar[u,"d"]& 0 \ar[u,"d"] & {}
\end{tikzcd}
	\label{E:bundle_gerbe_double_complex}
\end{equation}
\blue{
where we have omitted the coefficient group and covers from the notation for simplicity.
}
We take the row $\vC^\bullet(Y)$ to have
vertical degree 0 in \eqref{E:bundle_gerbe_double_complex} (corresponding to its true simplicial
degree), so the row
$\vC^\bullet(X)$ has degree $-1$.
Related versions of this \v Cech-simplicial double complex appear more generally in algebraic geometry \cite{friedlander}, and 
are also discussed in \cite{B-MI} in the context of simplicial gerbes.

\begin{conv} \label{Conv:multicomplex} Our convention for double (and
  higher) complexes is that the two differentials {\em commute}, as above.
  This necessitates the introduction of a sign (depending on an ordering of
  the differentials) in the total differential, which we take to be
\[
	D = \delta + (-1)^{p} d \quad \text{on}\quad  \vC^{p}(Y^{[q]}; A).
\]
Another possible sign convention is given by changing the formal order of $d$
and $\delta$, namely $D' = d + (-1)^{q+1} \delta$ (recalling that $Y^{[q]}$
has vertical degree $q-1$). This is intertwined with $D$ via the automorphism
$(-1)^{p(q+1)}$ of the double complex.

In general, whenever we have a multicomplex $C^{p_1,\ldots,p_k}$ with $k$
commuting differentials $d_1, d_2,\ldots, d_{k}$, the total
differential will be defined inductively by
\[
\begin{gathered}
	D_k = D_{k-1} + (-1)^{p_1 +\cdots+p_{k-1}} d_k
	\\= d_1 + (-1)^{p_1} d_2 + (-1)^{p_1+p_2} d_3 + \cdots + (-1)^{p_1 +\cdots+p_{k-1}} d_k.
\end{gathered}
\]
Switching the order of the indices requires composing with an automorphism
given in each degree by an appropriate power of $-1$ as above.
\end{conv}

\begin{prop} \blue{For an admissible pair $(\ecU,\ecV)$,} the simplicial chain complex
\begin{equation}
	0 \to \vC_{\blue{\ecU}}^\ell(X; A) \stackrel d \to \vC_{\blue{\ecV}}^\ell(Y; A) \stackrel d \to \vC_{\blue{\ecV^{[2]}}}^\ell(Y^{[2]}; A) \stackrel d \to \cdots
	\label{E:exact_simplicial}
\end{equation}
is exact. In particular, a collection of local sections $s :
\Et{\ecU} \to Y$ determines a chain homotopy contraction \blue{via the maps in
Lemma~\ref{L:admissible_cover_fib_prod}.}
\label{T:exact_simplicial}
\end{prop}
\noindent Compare the exactness of the de Rham complex in \S8 of \cite{Murray1}.
This may be understood as a manifestation of the fact that the geometric realization of the simplicial
set $Y^{[\bullet]}$ is known to be homotopy equivalent to $X$.

\blue{
\begin{proof}
The chain contraction is defined by 
\[
	\wt s_\ell^\ast : \vC^\ell_{\ecV^{[k]}}(Y^{[k]}; A) \to \vC^\ell_{\ecV^{[k-1]}}(Y^{[k-1]}; A)
\]
for $k \geq 2$,
where $\wt s_\ell$ is the map of covers in Lemma~\ref{L:admissible_cover_fib_prod}.\eqref{I:admissible_cover_fib_prod_sections},
and at the bottom by $\wt s_\ell : \vC^\ell_\ecV(Y; A) \to \vC^\ell_\ecU(X; A)$ where $\wt s_\ell$ is as in Definition~\ref{D:admissible_pair}.\eqref{I:admissible_pair_int_prop}.
That $\wt s_\ell^\ast$ forms a chain homotopy retraction follows directly from \eqref{E:simplicial_chain_contraction}.
\end{proof}
}
\begin{rmk} It is important that the maps $s_\ell^\ast$ do {\em not}
  generally commute with the \v Cech differential; in other words, we do
  not obtain a chain map, and in particular we do not claim that
  $s_\ell^\ast$ descends to cohomology.
Indeed the lift, $s_{\ell}$, of $s$ in \eqref{E:s_ell}
  corresponds to an (arbitrary) preference for the map corresponding to the
  first factor, $U_1$, in the $\ell$-fold intersection.
\end{rmk}

\subsection{Dixmier-Douady class of a gerbe} \label{S:DD_class}
The setting for 
our analysis of the \v Cech cohomology of a bundle gerbe is 
the truncated complex
\begin{equation}
\begin{tikzcd}[sep=small]
	 0& 0 & 0& 
	\\\vZ_{\blue{\ecV^{[2]}}}^0(Y^{[2]}; A) \ar[r,"\delta"]\ar[u,"d"] &\vZ_{\blue{\ecV^{[2]}}}^1(Y^{[2]}; A) \ar[r, "\delta"]\ar[u,"d"] & \vZ_{\blue{\ecV^{[2]}}}^2(Y^{[2]}; A) \ar[r,"\delta"]\ar[u,"d"] & {}
	\\\vC_{\blue{\ecV}}^0(Y; A) \ar[r,"\delta"]\ar[u,"d"] &\vC_{\blue{\ecV}}^1(Y; A) \ar[r, "\delta"]\ar[u,"d"] & \vC_{\blue{\ecV}}^2(Y; A) \ar[r,"\delta"]\ar[u,"d"] & {}
\end{tikzcd}
	\label{E:gerbe_double_cplx}
\end{equation}
where we have omitted the bottom row of \eqref{E:bundle_gerbe_double_complex}, and where
\[
\begin{aligned}
	\vZ_{\blue{\ecV^{[2]}}}^\ell(Y^{[2]}; A) &:= \kernel \bset{ d : \vC_{\blue{\ecV^{[2]}}}^\ell(Y^{[2]}; A) \to \vC_{\blue{\ecV^{[3]}}}^\ell(Y^{[3]}; A)}
	\\&= \img\bset{d : \vC_{\blue{\ecV}}^\ell(Y; A) \to \vC_{\blue{\ecV^{[2]}}}^\ell(Y^{[2]}; A)}.
\end{aligned}
\]
Denote by
\[
	\vH^\bullet_Z(Y^{[2]}; A) := \blue{\lim_{\cV^{[2]}}}H_{\blue{\ecV^{[2]}}}^\bullet(\vZ^\bullet(Y^{[2]}; A), \delta)
\]
the \v Cech cohomology of the simplicially trivial classes on $Y^{[2]}$, or the horizontal
cohomology of the top row in \eqref{E:gerbe_double_cplx},
\blue{taken in the direct limit over covers of the form $\ecV^{[2]}$ for admissible pairs}.
For later use we note the following result.
\begin{lem}
There is a natural Bockstein isomorphism $\vH^\bullet_Z(Y^{[2]}; \bbC^\ast) \cong \vH^{\bullet + 1}_Z(Y^{[2]}; \bbZ)$.
\label{L:bockstein_simplicial}
\end{lem}
\begin{proof}
Regarding the chain complexes $\vZ_{\blue{\ecV^{[2]}}}^\bullet(Y^{[2]}; A)$ for an abelian group
$A$ as the image under $d$ of $\vC_{\blue{\ecV}}^\bullet(Y; A)$, it follows both that the
coefficient sequence
\begin{equation}
	0 \to \vZ_{\blue{\ecV^{[2]}}}^\bullet(Y^{[2]}; \bbZ) \to \vZ_{\blue{\ecV^{[2]}}}^\bullet(Y^{[2]}; \bbC) \stackrel{\exp(2\pi i \cdot)}{\to} \vZ_{\blue{\ecV^{[2]}}}^\bullet(Y^{[2]}; \bbC^\ast) \to 0
	\label{E:Bockstein_ses}
\end{equation}
is short exact, and that $\vZ_{\blue{\ecV^{[2]}}}^\bullet(Y^{[2]}; \bbC)$ is acyclic, from which
the long exact sequence for \eqref{E:Bockstein_ses} degenerates to the
Bockstein isomorphism.
\end{proof}

\begin{thm}
The total cohomology of the double complex \eqref{E:gerbe_double_cplx} is isomorphic to $\vH_\ecU^\bullet(X; A)$.
\label{T:total_cohom_just_X}
\end{thm}
\begin{proof}
Owing to exactness of the columns, the $(d,\delta)$ spectral sequence of \eqref{E:gerbe_double_cplx} 
degenerates at the $E_1$ page to 
\[
\begin{tikzcd}[row sep=2ex, column sep=2ex]
	 0& 0 & 0& 
	\\0  & 0  & 0 & {}
	\\\vC_\ecU^0(X) \ar[r,"\delta"] &\vC_\ecU^1(X) \ar[r, "\delta"] & \vC_\ecU^2(X) \ar[r,"\delta"] & {}
\end{tikzcd}
\]
and therefore stabilizes at $E_2$ to the cohomology $\vH_\ecU^\bullet(X; A)$.
\end{proof}

Next we will show that a bundle gerbe is represented by a pure cocycle in the double complex
\eqref{E:gerbe_double_cplx} concentrated at $\vZ^1(Y^{[2]}; \bbC^\ast)$, so
with $\delta c(L) = 0$ and $dc(L) = 0$ and this descends to the Dixmier-Douady class.

\begin{prop}\label{P:bundle_gerbe_double_class}
 A bundle gerbe $(L, Y,X)$ has Chern class represented by $c(L) \in \vZ_{\ecV^{[2]}}^1(Y^{[2]}; \bbC^\ast)$
\blue{for some admissible pair of covers $(\ecU,\ecV)$};
in particular,
\begin{equation}
c(L)\in\vH_Z^1(Y^{[2]}; \bbC^\ast)\cong \vH_Z^2(Y^{[2]}; \bbZ).
\label{rbm.75}\end{equation}
Conversely, any such class determines a bundle gerbe, and $L$
admits a trivialization if and only if $[c(L)]\in d\vH^1(Y; \bbC^\ast)$.
\end{prop}

\begin{proof} As a complex line bundle, $L$ gives rise to a Chern cocycle $c(L) \in \vC_\cO^1(Y^{[2]}, \bbC^\ast)$ \blue{for some cover $\cO$ of $Y^{[2]}$,
which by Lemma~\ref{L:admissible_cover_fib_prod} may be assumed to be of the form $\cO = \ecV^{[2]}$ for an admissible pair}, and $dc(L) \in \vC_{\ecV^{[3]}}^1(Y^{[3]}, \bbC^\ast)$
represents the bundle $dL$ on $Y^{[3]}$. 
The trivialization of $dL$ is encoded by $\gamma \in \vC_{\blue{\ecV^{[3]}}}^0(Y^{[3]}, \bbC^\ast)$ 
such that $\delta \gamma = dc(L)$, and the fact that this induces the canonical trivialization of $d^2 L$ on $Y^{[4]}$ means
that $d \gamma = 0$. 
Thus by exactness we can alter $c(L)$ by $\delta$ applied to a $d$-preimage of $\gamma$ to arrange that $dc(L) = 0$. 

Altering $c(L) \in \vZ_{\blue{\ecV^{[2]}}}^1(Y^{[2]}; \bbC^\ast)$ by $\delta \beta$ for $\beta \in \vZ_{\blue{\ecV^{[2]}}}^0(Y^{[2]}; \bbC^\ast)$ amounts
to applying an automorphism to $L \to Y^{[2]}$ which does not change the trivialization of $dL \to Y^{[3]}$, 
so the Chern class in $\vH^1_Z(Y^{[2]}; \bbC^\ast)$ is well-defined.

Conversely, given $\alpha \in \vZ^1_{\blue{\ecV^{[2]}}}(Y^{[2]}; \bbC^\ast)$ representing a
cocycle with respect to some fixed open cover \blue{$\ecV^{[2]}$ of
  $Y^{[2]}$ associated to an admissible pair}, the usual construction uses
$\alpha$ on \blue{$(\ecV^{[2]})^{(2)}$} to assemble
a line bundle $L \to Y^{[2]}$ out of trivial bundles on $V \in \blue{\ecV^{[2]}}$.
Then since $d\alpha = 0$ it follows that $dL$ is assembled trivially out of trivial bundles 
on the open cover $\ecV^{[3]}$ of $Y^{[3]}$, and hence
is globally trivial (with the trivialization agreeing with the canonical one on $d^2 L$).

Finally, $L$ admits a trivialization $L \cong dQ$ for some $Q \to Y$, if and only if $c(L) = dc(Q) \in \vH^1_Z(Y^{[2]}; \bbC^\ast)$ where $c(Q) \in H_{\blue{\ecV}}^1(Y; \bbC^\ast)$
is the Chern class of $Q$.
\end{proof}

\begin{defn} The {\em Dixmier-Douady class} of a bundle gerbe $(L, Y,X)$ is
  the image $\DD(L)\in\vH^2(X;\bbC^\ast) = \lim_\ecU \vH^2_{\ecU}(X; \bbC^\ast) \cong \vH^3(X; \bbZ)$ of the
  hypercohomology class of $c(L) \in \vZ_{\blue{\ecV^{[2]}}}^1(Y^{[2]}; \bbC^\ast)$ in the
  double complex \eqref{E:gerbe_double_cplx} and is obtained
explicitly by a zig-zag construction
\begin{equation}
\begin{tikzcd}[sep=small]
	 0
	 \\{c(L)} \ar[cm bar-to,u] \ar[cm bar-to,r]& 0
	\\ \beta \ar[cm bar-to,u, "-d"] \ar[cm bar-to,r,"\delta"] & \delta\beta \ar[cm bar-to,u] \ar[cm bar-to,r] & 0
	\\ & \DD(L) \ar[cm bar-to,u,"d"] \ar[cm bar-to,r] & 0 \ar[cm bar-to,u]
\end{tikzcd}
	\label{E:gerbe_zig_zag}
\end{equation}
\label{D:dixmier_douady}
\end{defn}
\noindent The sign in $- d\beta = c(L)$ arises from the fact that the total
differential $D$ involves the term $-d$ on that column according to
Convention~\ref{Conv:multicomplex}. 

Note that if $\pi_i : Y_i \to X_i$, $i = 1,2$ are locally split maps which
are intertwined by $f : X_1 \to X_2$ and $\wt f: Y_1 \to Y_2,$ 
then \v Cech cochain maps determined by $\wt f^{[k]}
: Y_1^{[k]} \to Y_2^{[k]}$ as in Proposition~\ref{P:Cech_map} together form
a morphism $\vC^\bullet(Y_2^{[\bullet]}; A) \to
\vC^\bullet(Y_1^{[\bullet]}; A)$ of double complexes, in that the
$(\wt f^{[k]})^\ast$ commute with $\delta$ and $d$.
\begin{added}
Indeed, we may fix an admissible pair $(\ecU_2,\ecV_2)$ for $(Y_2,X_2)$ 
and then invoke Lemma~\ref{Addedrbm} to obtain a refinement 
of $(f^\inv \ecU_2,\wt f^\inv \ecV_2)$
by an admissible pair
$(\ecU_1,\ecV_1)$,
and then it follows that there are maps of covers $\wt f^{[k]} : \Et{\ecV_1^{[k]}} \to \Et{\ecV_2^{[k]}}$
the pullbacks of which intertwine the two double complexes.
\end{added}

\begin{prop}\label{P:Consistent} The Dixmier-Douady class as defined above coincides with the
  definition given by Murray and is natural with respect to inverse,
  product, and pullback; it vanishes if and only if the gerbe is trivial.
\end{prop}

\begin{proof} 
The (well-known) naturality properties follow directly from the preceding remark.
To see the coincidence of our definition with that of Murray
given in \cite{Murray1}, we first recall the latter.

  Suppose $s : \ecU \to Y$ is a set of local sections of the locally split
  map, and consider the pullback $L' = (s^2)^\ast L$ to $\ecU^{(2)}$ of $L$
  via the map $s^2 : \ecU^{(2)} \to Y^{[2]}.$ Since $L$ is locally trivial,
  this cover can be refined so that $L'$ is trivial over each component,
  and so has a nonvanishing section $\sigma : \ecU^{(2)} \to L'.$ The
  trivialization of $dL \to Y^{[3]}$ pulls back to give a trivialization of
  $\delta L' = (s^3)^\ast dL \to \ecU^{(3)}$ which allows $g := \delta
  \sigma$ to be regarded as a cochain $g : \ecU^{(3)} \to \bbC^\ast$ and the
  associativity condition over $Y^{[4]}$ implies that $g$ is closed, hence
  $[g] \in \vH_{\ecU}^2(X; \bbC^\ast)\cong H^3(X; \bbZ)$ is defined to be the
  Dixmier-Douady class.

To see that this is equivalent to Definition~\ref{D:dixmier_douady}, it suffices to show that
$[g]$ represents the image of $c(L)$ in the total cohomology of the double
complex $(\vC^\bullet(\Et{\ecU^{(\bullet)}};\bbC^\ast), \delta, d)$, where we use
$\Et{\ecU} \to X$ itself as the locally split map.
For convenience we suppose that $\ecU$ is a `good cover', meaning that each element
of $\ecU^{(\ell)}$ is contractible for each $\ell$.
Note that by this contractibility,
the \v Cech cohomology $\vH^\bullet(\Et{\ecU^{(\ell)}}; \bbC^\ast)$ of the space $\Et{\ecU^{(\ell)}}$ is trivial except in degree $0$
where
\[
	\vH^0(\Et{\ecU^{(\ell)}}; \bbC^\ast) = \Gamma(\ecU^{(\ell)}; \bbC^\ast) = \vC^{\ell-1}_{\ecU}(X; \bbC^\ast).
\]
Thus the $(\delta, d)$ spectral sequence of the double complex
\[
\begin{tikzcd}[sep=small]
	{} &{}  &{}  &
	\\ \vC^0(\Et{\ecU^{(3)}}) \ar[u, "d"] \ar[r, "\delta"] & \vC^1(\Et{\ecU^{(3)}}) \ar[u, "d"] \ar[r, "\delta"] &\vC^2(\Et{\ecU^{(3)}}) \ar[u, "d"] \ar[r, "\delta"] & {}
	\\ \vC^0(\Et{\ecU^{(2)}}) \ar[u, "d"] \ar[r, "\delta"] & \vC^1(\Et{\ecU^{(2)}}) \ar[u, "d"] \ar[r, "\delta"] & \vC^2(\Et{\ecU^{(2)}}) \ar[u, "d"] \ar[r, "\delta"] &{}
	\\ \vC^0(\Et{\ecU}) \ar[u, "d"] \ar[r, "\delta"] & \vC^1(\Et{\ecU}) \ar[u, "d"] \ar[r, "\delta"] &\vC^2(\Et{\ecU}) \ar[u, "d"] \ar[r, "\delta"] &{}
\end{tikzcd}
\]
degenerates at the $E_1$ page to 
\[
\begin{tikzcd}[sep=small]
	{} &  & 
	\\ \vH^0(\Et{\ecU^{(3)}}) = \vC_{\ecU}^2(X) \ar[u, "d=\delta"] & 0 & 0
	\\ \vH^0(\Et{\ecU^{(2)}}) = \vC_{\ecU}^1(X) \ar[u, "d=\delta"] & 0 & 0
	\\ \vH^0(\Et{\ecU}) = \vC_{\ecU}^0(X) \ar[u, "d=\delta"] & 0 & 0
\end{tikzcd}
\]
with the simplicial differential now identified with the \v Cech differential
on $\vC_{\ecU}^\bullet(X)$, and then stabilizes at $E_2$ to give
$\vH_{\ecU}^\bullet(X; \bbC^\ast) \cong H^{\bullet+1}(X; \bbZ)$.
The image of $[L'] \in \vC^1(\Et{\ecU^{(2)}}; \bbC^\ast)$ in the total cohomology
$H^3(X; \bbZ)$ is therefore equivalently represented by its image in $\vH^0(\ecU^{(3)}; \bbC^\ast) = \vC_{\ecU}^2(X; \bbC^\ast)$
on the $E_1$ page above, and Murray's construction gives an explicit zig-zag
\[
\begin{tikzcd}[sep=small]
	0
	\\ g \ar[cm bar-to,u] \ar[cm bar-to,r] & 0
	\\ \sigma \ar[cm bar-to,u] \ar[cm bar-to,r] & {[L']} \ar[cm bar-to,u] \ar[cm bar-to,r] & 0
\end{tikzcd}
\]
realizing $[g]$ as $\DD(L)$.
\end{proof}

\subsection{Representability of 3-classes} \label{S:rep_of_three}

We proceed to give a characterization of the classes in $H^3(X; \bbZ)$ which are represented
by bundle gerbes $(L,Y,X)$ for a given locally split map $Y \to X.$
Note that the augmented double complex
\begin{equation}
\begin{tikzcd}[row sep=2ex, column sep=2ex]
	 0& 0 & 0& 
	\\\vZ_{\ecV^{[2]}}^0(Y^{[2]}) \ar[r,"\delta"]\ar[u,"d"] &\vZ_{\ecV^{[2]}}^1(Y^{[2]}) \ar[r, "\delta"]\ar[u,"d"] & \vZ_{\ecV^{[2]}}^2(Y^{[2]}) \ar[r,"\delta"]\ar[u,"d"] & {}
	\\\vC_{\ecV}^0(Y) \ar[r,"\delta"]\ar[u,"d"] &\vC_{\ecV}^1(Y) \ar[r, "\delta"]\ar[u,"d"] & \vC_{\ecV}^2(Y) \ar[r,"\delta"]\ar[u,"d"] & {}
	\\\vC_{\ecU}^0(X) \ar[r,"\delta"]\ar[u,"d"] &\vC_{\ecU}^1(X) \ar[r, "\delta"]\ar[u,"d"] & \vC_{\ecU}^2(X) \ar[r,"\delta"]\ar[u,"d"] & {}
	\\0 \ar[u] & 0 \ar[u] & 0 \ar[u] & {}
\end{tikzcd}
	\label{E:aug_double_complex}
\end{equation}
\blue{
for any admissible pair $(\ecU,\ecV)$
}
has exact columns and therefore trivial total cohomology. 
Since the $(\delta, d)$ spectral sequence of this complex (beginning with the horizontal differential)
must necessarily stabilize at the $E_3$ page (as there are only three rows),
it follows that
the $E_2$ differentials
are necessarily isomorphisms, which we record in the following form 
\blue{
after passing to the direct limit over covers in cohomology.
}

\begin{thm}
There are isomorphisms
\begin{equation}
	\kernel\bset{\pi^\ast : \vH^{k+1}(X; A) \to \vH^{k+1}(Y; A)} \cong \vH_Z^{k}(Y^{[2]}; A)/ \vH^k(Y; A)
	\label{E:SS_isom}
\end{equation}
for each $k \in \bbN$ and coefficient group $A$; these isomorphisms are natural with respect to pullback by maps
\[
\begin{tikzcd}
 	Y_1 \ar[r,"\wt f"] \ar[d, "\pi_1"] & Y_2 \ar[d,"\pi_2"]
	\\ X_1 \ar[r, "f"] & X_2
\end{tikzcd}
\]
of locally split spaces, and also with respect to the Bockstein
isomorphisms $\vH_Z^{k+1}(Y^{[2]}; \bbZ) \cong \vH_Z^k(Y^{[2]}; \bbC^\ast)$
and $\vH^{k+1}(X; \bbZ) \cong \vH^k(X; \bbC^\ast)$.
\label{T:bundle_gerbe_cohom_iso}
\end{thm}
\begin{rmk}
It is reasonable to call the isomorphism \eqref{E:SS_isom} {\em transgression} from classes on $X$ to (equivalence classes of) classes on $Y^{[2]}$; the map
is realized at the chain level by the zig-zag \eqref{E:gerbe_zig_zag}.
\end{rmk}

In particular there is a natural isomorphism
\begin{equation}
\begin{gathered}
	\kernel\set{\pi^\ast : \vH^3(X; \bbZ) \to \vH^3(Y; \bbZ)} \cong \vH_Z^2(Y^{[2]}; \bbZ)/ \vH^2(Y; \bbZ)
	\\\cong \vH_Z^1(Y^{[2]}; \bbC^\ast)/ \vH^1(Y; \bbC^\ast)
	\label{E:bundle_gerbe_cohom_iso}
\end{gathered}
\end{equation}
under which
the Chern class $c(L) \in
\vH_Z^1(Y^{[2]}; \bbC^\ast)/d\vH^1(Y; \bbC^\ast)$
is the image 
of 
$\DD(L)$ 
in $\vH^3(X; \bbZ)$.
We see again that $\DD(L) = 0$ if and only if $c(L) \in d\vH^1(Y; \bbC^\ast)$, which
by Proposition~\ref{P:bundle_gerbe_double_class} holds precisely when $L$ is
trivial.
In combination with Proposition~\ref{P:bundle_gerbe_double_class} this proves the following result.

\begin{thm}
A class $\alpha \in H^3(X; \bbZ)$ is represented by a bundle gerbe $(L, Y, X)$ for a given locally split map $\pi : Y \to X$
if and only if $\pi^\ast \alpha = 0 \in H^3(Y; \bbZ)$.
\label{T:representability_three_class}
\end{thm}

\begin{rmk} Another direct (and more geometric) way to show
  Theorem~\ref{T:representability_three_class} is to use $B\PU(H)$ as a
  $K(\bbZ, 3)$, where $H$ is an infinite dimensional separable Hilbert
  space. Here $\PU(H) = \UU(H)/\UU(1)$ denotes the projective unitary group
  and by Kuiper's theorem $\UU(H)$ is contractible, making $\PU(H)$ a
  $K(\bbZ,2)$. Thus $\alpha \in H^3(X; \bbZ)$ is classified by a map (up to homotopy) to
$B\PU(H)$ and represented by a $\PU(H)$ bundle $E \to X$. If $\pi^\ast
\alpha = 0 \in H^3(Y; \bbZ)$, it follows that $\pi^\ast E \to Y$ admits a
global section $s : Y \to \pi^\ast E$.  Then on $Y^{[2]}$ the shift map
composed with $s^{[2]}$ determines a map $\chi : Y^{[2]} \to \PU(H)$, along
which the universal line bundle can be pulled back to give a simplicial
bundle $L = \chi^\ast \UU(H) \to Y^{[2]}$ with $\DD(L) = \alpha$.
\end{rmk}

\subsection{Classification of trivializations} \label{S:gerbe_triv}
Suppose $(L, Y, X)$ is a trivial gerbe.
There is an action on the set of trivializations of $L$ by $H^2(X; \bbZ)$ (in the form of equivalence classes of line bundles) as follows.
Given 
a line bundle 
$P \to Y$ 
trivializing $L$, so $dP \cong L$, and $\alpha \in H^2(X; \bbZ)$
representing a line bundle $Q \to X$,
the bundle $P \otimes \pi^\ast Q = P \otimes dQ \to Y$ is another trivialization
of $L$ in light of the fact that $d^2 Q$ is canonically trivial.
\begin{prop}
Let $(L,Y,X)$ be a trivial gerbe. 
Then the set of trivializations of $L$ is a torsor for the group $\img\set{\pi^\ast :H^2(X; \bbZ) \to H^2(Y; \bbZ)}$.
\label{P:gerbe_triv}
\end{prop}
\begin{proof}
Clearly the action of $H^2(X; \bbZ)$ factors through its image in $H^2(Y;
\bbZ)$; to see that this image acts transitively suppose $P \to Y$ and $P' \to
Y$ are two trivializations of $L$, represented by classes $[P], [P'] \in
\vH^2(Y; \bbZ)$.
Then $d([P] - [P']) = 0 \in H^2_Z(Y^{[2]}; \bbZ)$, and from the $(\delta,d)$
spectral sequence for \eqref{E:aug_double_complex}, the $E_2$ term associated
to $\vC^2(Y; \bbZ)$ of which must vanish identically, it follows that
\begin{equation}
	\kernel\set{d: \vH^2(Y; \bbZ) \to \vH^2_Z(Y^{[2]}; \bbZ)}
	= \img\set{\pi^\ast : \vH^2(X; \bbZ) \to \vH^2(Y; \bbZ)},
	\label{E:triv_image}
\end{equation}
and hence $[P] - [P'] = \pi^\ast [Q]$ for some $[Q] \in \vH^2(X; \bbZ)$, represented
by a line bundle $Q \to X$.
\end{proof}

\subsection{Decomposable and universal gerbes} \label{S:gerbe_exs}

One consequence of Theorem~\ref{T:representability_three_class} is the existence of the {\em decomposable
  gerbes} of \cite{MR2674880}. Given a 3-class on $X$ which is the cup
product $\alpha \cup \beta$ of $\alpha \in H^2(X; \bbZ)$ and $\beta \in
H^1(X; \bbZ)$, we may take $\pi : Y \to X$ to be the circle bundle with
Chern class $c(Y) = \alpha$, and then since $Y$ is canonically trivial when
pulled back to itself, it follows that $\pi^\ast (\alpha \cup \beta) =
0\cup \pi^\ast \beta = 0 \in H^3(Y; \bbZ)$, so by
Theorem~\ref{T:representability_three_class} the following is immediate.

\begin{prop}
For every $\alpha \in H^2(X; \bbZ)$ and $\beta \in H^1(X; \bbZ)$, the circle bundle $Y \to X$
with $c(Y) = \alpha$ supports a bundle gerbe $(L, Y, X)$ with $\DD(L) = \alpha \cup \beta$. 
\label{P:decomp_gerbe}
\end{prop}

\begin{rmk} Note that \cite{MR2674880} goes further for $X$ a smooth manifold by 
constructing a connection 
on the gerbe from a connection on $Y$ and a function $u \in
C^\infty(X; \UU(1))$ representing $\beta$.
\end{rmk}

In fact, the image of $\alpha \in H^2(X; \bbZ)$ in $H_Z^1(Y^{[2]};
\bbZ)/H^1(Y; \bbZ)$ with respect to the isomorphism
\eqref{E:bundle_gerbe_cohom_iso} has a geometric interpretation
that will be of use in the construction of decomposable bigerbes in
\S\ref{S:decomp_bigerbes}.

\begin{lem}
Let $\pi : Y \to X$ be a circle bundle with $c(Y) = \alpha \in H^2(X; \bbZ)$. Then the image of $\alpha$
under the isomorphism \eqref{E:bundle_gerbe_cohom_iso} coincides with the pullback to $Y^{[2]}$ of the generator
of $H^1(\UU(1); \bbZ)$ by the {\em shift map}
\begin{equation}
	\chi : Y^{[2]} \to \UU(1),
	\qquad y_2 = \chi(y_1,y_2)y_1 \quad \text{for $(y_1,y_2) \in Y^{[2]}$.}
	\label{E:shift_map}
\end{equation}
\label{L:cohom_transgr_circle_bundle}
\end{lem}
\begin{proof}
This is easiest to see with $\UU(1)$ coefficients. With respect to the
isomorphism $H^1(\UU(1); \bbZ) \cong H^0(\UU(1); \UU(1))$ the generator
corresponds to the identity map, so it suffices to show that the image of
$c(Y) \in H^1(X; \UU(1))$ is represented by $\chi$, itself regarded as a class in
$H^0(Y^{[2]}; \UU(1))$. 

Let $\alpha \in \vC_{\ecU}^1(X; \UU(1))$ represent $c(Y)$; explicitly, we may take $\alpha$ to be defined
with respect to a cover $i : \ecU \to X$ with respect to which $Y$ is (locally) trivialized by $h : i^\ast Y \to \ecU \times \UU(1)$, 
and we may abuse notation to write $\alpha = \delta h$, meaning $\alpha :
\ecU^{(2)} \to \UU(1)$ is defined so that $\delta h = 1\times \alpha :
\ecU^{(2)} \times \UU(1) \to \ecU^{(2)} \times \UU(1)$. Now $\pi^\ast Y = Y^{[2]} \to Y$ is globally trivialized by $1\times \chi$, to which $\pi^\ast h$ may be compared
to write $\pi^\ast h = \gamma \chi$ for $\gamma \in \vC^0_{\pi^\inv \ecU}(Y; \UU(1))$ and then 
\[
	d\alpha = \pi^\ast \alpha = \delta \pi^\ast h = \delta \gamma \delta \chi = \delta \gamma \in \vC^1_{\pi^\inv \ecU}(Y; \UU(1)).
\]
Finally, it follows by a straightforward computation 
that 
$d \gamma = \chi^\inv$ in $\vC^0_{\pi^\inv \ecU}(Y^{[2]}; \UU(1))$,
and then the result follows in observance of Convention~\ref{Conv:multicomplex}.
\end{proof}

It also follows from Theorem~\ref{T:representability_three_class} that for a
connected, locally contractible, space $X$, a gerbe can be constructed
representing any integral three class using the (based) path fibration $ P
X \to X$. Indeed, the hypotheses on $X$ imply that the end-point map $PX \to X$ is locally split, and since $ PX$ is contractible, any 3-class on $X$ vanishes
when lifted to $ PX$. The fiber product $ P^{[2]} X$ may be identified with
the based loop space $ \Omega X$, and the isomorphism
\eqref{E:bundle_gerbe_cohom_iso} takes the form 
\[ 
	H^3(X; \bbZ) \cong \vH_Z^1(\Omega X; \bbC^\ast), 
\]
from which we recover the following well-known result.
\begin{thm}\label{T:univ_gerbe} For a connected,
  locally contractible, space $X$ each $\alpha \in H^3(X; \bbZ)$,
  corresponds to a unique bundle gerbe $L \to \Omega X$ (up to simplicial
  isomorphisms of the line bundle) with $\DD(L) = \alpha$.
\end{thm}
\noindent This `canonical gerbe' on the loop space goes back at least to Brylinski \cite{brylinski}.
Murray defines a bundle gerbe version in \cite{Murray1} under the assumption that $X$ is 2-connected;
a hypothesis which is removed in \cite{CJMSW}.

In particular, since $K(\bbZ,3)$ may be realized as a CW complex, its path
space carries a universal gerbe.

The simplicial structure on $\Omega X$ coming from $P^{[k]} X$ is related to
what has been called the {\em fusion product} in the literature
\cite{Stolz-Teichner2005, MR2980505, MR3493404, MR3391882}.
A point $\gamma = (\gamma_0,\gamma_1,\gamma_2) \in P^{[3]} X$ consists of three
paths with common endpoints and so defines three loops, $\ell_i = \pi_i
\gamma \in \Omega X = P^{[2]} X$, $i = 0,1,2$, by the simplicial maps, and
we say $\ell_1 = (\gamma_0,\gamma_2)$ is the {\em fusion product} of
$\ell_2=(\gamma_0,\gamma_1)$ and $\ell_0 = (\gamma_1,\gamma_2)$.

A {\em fusion structure} on a line bundle $L \to \Omega X$ is a collection of associative isomorphisms
\[
	L_{\ell_1} \cong L_{\ell_2} \otimes L_{\ell_0}
\]
for all such triples, which is equivalent to a simplicial line bundle structure
on $L$ with respect to $P^{[\bullet]} X$.
In this language then, Theorem~\ref{T:univ_gerbe} shows that fusion line
bundles on $\Omega X$, which are equivalent to bundle gerbes $(L, PX, X)$, are
classified by $H^3(X; \bbZ)$ (see also Waldorf's related results in \cite{MR2980505}).

\section{Doubling and the free loop space} \label{S:prod_simp}
\subsection{Simplicial bundle gerbes and figure-of-eight} \label{Simp.gerbe}
Replacing the simplicial line bundle in the definition of a bundle gerbe
with a bundle gerbe over $X_2$ of a simplicial space $X_\bullet$ leads to the notion of a {\em simplicial
  bundle gerbe}, which has been defined by Stevenson \cite{Stevenson2001}
and is the setting for his definition of bundle 2-gerbes.  
Here we consider a more limited `product-simplicial' version, which we
call simply \emph{doubled}, of this theory, not yet to obtain a version of
2-gerbes as we shall do in \S\ref{S:bundle_bigerbes} below, but rather to
promote the examples of bundle gerbes involving the based loop space
$\Omega X$ to those involving the free (unbased) loop space $L X$ by
satisfying an additional condition with respect to the simplicial space
$\set{X^k : k \in \bbN}$ of products, with face maps the projections; this
space is often denoted by $EX$ in the literature.

While we specialize to this simplicial space of products below, we proceed for the moment in some generality for 
an arbitrary simplicial space $X_\bullet$, where we continue to use our unusual enumeration
 convention.
Suppose then that $(L, Y_2, X_2)$ is a bundle gerbe over $X_2$.
Using products, inverses and pullbacks, we may define the gerbe
\[
	\pa L := \pi_0^\ast L \otimes \pi_1^\ast L^\inv \otimes \pi_2^\ast L
\]
over $X_3$, where $\pi_j : X_3 \to X_2$ for $j = 0,1,2$ are the face maps of the simplicial space.
\begin{defn}
A {\em simplicial trivialization} of a bundle gerbe $L$ over $X_2$ is a trivialization of the bundle gerbe $\pa L$
over $X_3$.
It follows by naturality that for such a gerbe the Dixmier-Douady class $\DD(L) \in H^3(X_2; \bbZ)$ satisfies
\begin{equation}
	\pa \DD(L) := \sum_{j=0}^2 (-1)^j \pi_j^\ast \DD(L) = \DD(\pa L) = 0 \in H^3(X_3; \bbZ).
	\label{E:simp_trivn_DD}
\end{equation}
Note that the gerbe $\pa L$ is defined a priori with respect to the locally split map
\begin{equation}
	\pi_0^\ast Y_2\times_{X_3} \pi_1^\ast Y_2 \times_{X_3} \pi_2^\ast Y_2 \to X_3.
	\label{E:simplicial_product_space}
\end{equation}
However, using the notion of gerbe morphism, we may specialize to the setting in which there exists a locally split map $Y_3\to X_3$ for some fixed space $Y_3$, along
with lifts $\wt \pi_j : Y_3 \to Y_2$ of the $\pi_j : X_3 \to X_2$ for $j = 0,1,2$. 
Indeed, it then follows that $Y_3$ maps through the product space \eqref{E:simplicial_product_space},
and we may require that
\begin{equation}
	\wt \pi_0^\ast L \otimes \wt \pi_1^\ast L^\inv \otimes \wt \pi_2^\ast L \to Y_3^{[2]}
	\label{E:alt_paL}
\end{equation}
is trivial as a bundle gerbe over $X_3$, where we continue to denote the extensions of $\wt \pi_j$ as maps from
$Y_3^{[2]}$ to $Y_2^{[2]}$ by the same notation.
When such data is available, we will abuse notation by referring to \eqref{E:alt_paL} itself as $\pa L$ (as these are (strongly) isomorphic as bundle gerbes over $X_3$) and 
a trivialization of \eqref{E:alt_paL} is a simplicial trivialization of $L$.
Explicitly, this then is the data of a line bundle $S \to Y_3$ such that $dS \cong \pa L$, as summarized
in the diagram
\begin{equation}
\begin{tikzcd}[sep=small]
&L\ar[r]& Y_2^{[2]}\ardtwo & \arlthree Y_3^{[2]}\ardtwo&\pa L\cong dS\ar[l]
& \\
&& Y_2\ar[d] & \arlthree Y_3\ar[d]&\ar[l]S\\
&X_1&\arltwo X_2& \arlthree X_3.
\end{tikzcd}
\label{rbm.30}\end{equation}
By naturality of the Dixmier-Douady class, the conclusion \eqref{E:simp_trivn_DD} remains valid.
\label{D:simplicial_trivn}
\end{defn}

\begin{rmk}
We do not require that the split maps $Y_\bullet \to X_\bullet$ 
be compatible by the $\pi_j$ in the sense of \S\ref{S:cech}.
We also do {\em not} require that $Y_\bullet$ extend to form (part of) a
simplicial space over $X_\bullet$, as indeed our example of interest will
not.  By contrast, in the setting of the bigerbes defined in
\S\ref{S:bundle_bigerbes} below, we will employ a {\em bisimplicial space} of {\em compatible} locally
split maps.
\end{rmk}

We now specialize to the case in which $X_\bullet = X^\bullet$ consists of products of a fixed space $X$.
As a special case of the fiber product
construction over the unique map $\pi : X \to \ast$ to a 1-point space,
this map is globally split, with section $s : \ast \mapsto x_\ast \in X$
for any choice of $x_\ast \in X$. 
The \v Cech theory constructions of
\S\ref{S:covers} and \S\ref{S:simplicial_cech} give a map $s^\ast : \vC_{\blue{\ecU^k}}^\bullet(X^k; A) \to
\vC_{\blue{\ecU^{k-1}}}^\bullet(X^{k-1}; A)$ which in this case does commute with the \v Cech
differential (since $s$ is global), and hence descends to a chain homotopy
contraction for each $\ell$ of the cohomology complex
\begin{equation}
	0 \to \vH^\ell(X; A) \stackrel \pa \to \vH^\ell(X^2; A) \stackrel \pa \to \vH^\ell(X^3; A) \stackrel \pa \to \cdots
	\label{E:product_cohomology_seq}
\end{equation}
which is therefore exact. 
(This is a reflection of the well-known fact that the geometric realization $\abs{E X}$ of the simplicial set $EX$ is contractible.) 
Indeed, denoting by $s = s\times 1\times \cdots \times 1 : X^k \to X^{k+1}$ the map $(x_0,\ldots,x_{k-1}) \mapsto (x_\ast,x_0,\ldots,x_{k-1})$, it follows that
$s^\ast \pa + \pa s^\ast = 1$ on $\vH^\bullet(X^k; A)$.
Note that throughout this section and below, we denote this product simplicial differential
by $\pa = \sum_{j=0}^{k-1} (-1)^j \pi_j^\ast$ rather than $d$ to avoid
confusion whenever both appear together.

As a consequence of \eqref{E:product_cohomology_seq} and \eqref{E:simp_trivn_DD}, 
we have the following result.

\begin{prop}
For a gerbe $(L, Y_2, X^2)$ with simplicial trivialization
over $X^3$, the Dixmier-Douady class of $L$ descends from $X^2$ to $X$ itself, so
\[
	\DD(L) \in H^3(X; \bbZ)
\]
is well-defined.
\label{P:prod_simp_gerbe}
\end{prop}
\noindent
We refer to such a gerbe as a \emph{doubled gerbe}.
\begin{rmk}
When $X_\bullet = X^{[\bullet]}$ is a more general simplicial space of fiber products of a locally split map $X = X_1 \to X_0$,
Stevenson in \cite{Stevenson2001} defines additional conditions for a {\em simplicial gerbe}, including higher associativity
conditions over $X_4$ and $X_5$ under which the class of a bundle gerbe further descends
to a degree four cohomology class on $X_0$ 
and such
an object is defined to be a {\em bundle 2-gerbe} on $X_0$.
Here we only use the simplicial condition to descend the 3-class from $X_2 = X^2$ to $X_1 = X$, and will not make use of these
additional conditions.
\end{rmk}

The locally split map of present interest consists of the free (unbased) path space 
\[
	IX = \cC([0,1]; X) \to X^2,
	\quad \gamma \mapsto \big(\gamma(0),\gamma(1)\big),
\]
mapping to $X^2$ by the evaluation map on both endpoints. 
Instead of the fiber products of the pullbacks of $IX$ to $X^3$ 
we will take $Y_3 = IX$ also,
with evaluation map 
\[
	IX \to X^3, 
	\quad \gamma \mapsto \big(\gamma(0),\gamma(\tfrac 1 2),\gamma(1)\big),
\]
mapping to the midpoint as well as the endpoints.
For disambiguation, 
we will often distinguish these two incarnations of the free path space by writing
them as $I_{2}X$ and $I_{3}X$, respectively.
The three liftings $\wt \pi_i$ of the projection maps $\pi_j : X^3 \to X^2$
taking $\gamma\in I_{3}X$ to $I_{2}X$ are obtained by reparameterizing to obtain
the three paths
\[
	\wt \pi_1 \gamma(t) = \gamma(t),
	\quad
	\wt \pi_2\gamma(t) = \gamma(\tfrac12t),
	\quad \text{and} \quad
	\wt \pi_0\gamma(t) = \gamma\big(\tfrac12(1+t)\big). 
\]
\begin{rmk}
While it is possible to continue to the right, with $Y_n = I_n X$ mapping to $X^n$
by evaluating along $n$ points, the need to reparameterize paths to define the lifts $\wt \pi_i$
means that these do not satisfy the simplicial relations, so we do not in fact
obtain simplicial spaces $Y_\bullet^{[k]}$ over $X_\bullet$. 
In particular,
the associated maps $\pa$ on line bundles do not form a complex, i.e.,
$\pa^2 L$ is not canonically trivial, except at the bottom level.
\end{rmk}

Observe that $Y_2^{[2]} = I^{[2]}_{2}X$ may be naturally identified
with the free loop space, $LX = \cC(\bbR/2\pi\bbZ; X)$, while the space $Y_3^{[2]} = I^{[2]}_{3}X$ consists of pairs of paths which coincide at their midpoint in addition
to their endpoints. 
The latter may be identified with those loops $\ell$ in $LX$ for which $\ell(\pi/2) = \ell(3\pi/2)$, 
which we call {\em figure-of-eight} loops, and we accordingly denote the {\em figure-of-eight loop (sub)space} by
\[
	L_8X \cong I^{[2]}_{3}X.
\]

In fact, in this case the product doubling condition for a gerbe over $X^2$ can be strengthened.
\begin{lem}
A gerbe $(L,IX, X^2)$ or $(L, IX, X^3)$ is trivial if and only if $L \to
LX$ (resp. $L \to L_8 X$) is trivial as a line bundle.
In particular, a gerbe $(L, I X, X^2)$ is doubled if and only if
$\pa L \to L_8X$ is a trivial line bundle.
\label{L:prod_simp_LX}
\end{lem}

\begin{proof}
Retraction of paths onto their initial points determines a deformation retract of
$I_kX$ onto $X$, with respect to which the two simplicial maps
$$
\begin{tikzcd}[sep=scriptsize]I^{[2]}_kX \arrtwo & I_kX\end{tikzcd}
$$
both become identified with the evaluation map $I_k^{[2]} X \to X$ at a single parameter value.
Thus every line bundle $P \to I_kX$ is isomorphic to a bundle $Q$ pulled back
from $X$, and then $dP \cong Q \otimes Q^\inv \to I^{[2]}_k X$ is isomorphic to a
trivial bundle.
This result is independent of $k$.

Alternatively, we may use the equality \eqref{E:triv_image} proved in
Proposition~\ref{P:gerbe_triv}, which here takes the form 
\[
\begin{gathered}
	\kernel\set{d : H^2(I_kX; \bbZ) \cong H^2(X; \bbZ) \to H_Z^2(I_k^{[2]}X; \bbZ)}
	\\= \img\set{\pi^\ast \cong \Delta^\ast : H^2(X^k; \bbZ) \to H^2(X; \bbZ) \cong H^2(I_kX; \bbZ)}
	= H^2(X; \bbZ),
\end{gathered}
\]
since with respect to the retraction $IX \simeq X$ the map $\pi : I_kX \to X^k$
is identified with the diagonal map $\Delta : X \to X^k$. 
Since $\Delta^\ast$ is surjective on cohomology, it follows that $d
\equiv 0 : H^2(I_kX; \bbZ) \to H^2_Z(I_k^{[2]}; \bbZ)$, so every trivial gerbe
is in fact trivial as a line bundle.
\end{proof}
\begin{rmk}
In other words, there are no `nontrivial trivial gerbes' with respect to the path spaces.
This seems at first confusing in light of Proposition~\ref{P:gerbe_triv}, since
the classifying set $\img\set{\pi^\ast : H^2(X^2; \bbZ) \to H^2(IX; \bbZ)} \cong H^2(X; \bbZ)$ of gerbe trivializtions may well be nontrivial, yet these facts are not inconsistent.
Indeed, while the only trivial gerbe with respect to $IX \to X^2$ is the equivalence class
of the trivial line bundle on $LX$, the set of {\em gerbe trivializations} of this trivial gerbe
may itself be nontrivial.

On the other hand, we could restrict consideration to {\em doubled
trivializations}, meaning line bundles $P \to I_2X$ with $dP = L$ such that
$\pa P \to I_3X$ is a trivial bundle.
The set of these doubled trivializations is indeed trivial, since
under the retractions $I_k X \simeq X$, the reparameterization maps $\wt \pi_j
: I_3X \to I_2 X$ become the identity, and the operator $\pa \cong \Id^\ast -
\Id^\ast + \Id^\ast$ likewise becomes the identity, so triviality of $\pa P$
implies triviality of $P$ itself.

The extension of this notion of doubling will be important in the
setting of the Brylinski-McLaughlin bigerbe in \S\ref{S:bm-bigerbe}.
\end{rmk}

From the point of view of fusion line bundles on loop space, the
doubling property corresponds to the `figure-of-eight' condition,
as defined in \cite{MR3391882,KM-equivalence}.
The following definition is therefore just a repackaging of the above in a different language.
\begin{defn}
A {\em loop-fusion} structure on a line bundle $L \to LX$ is a fusion structure, meaning a trivialization of $dL \to I^{[3]}_{2}X$ inducing the canonical trivialization of $d^2 L \to I^{[4]}_{2}X$, along with the {\em figure-of-eight} condition that $\pa L \to L_8X \cong I^{[2]}_3 X$ is trivial as a line bundle.
An isomorphism of loop-fusion line bundles is a line bundle isomorphism which intertwines the fusion structures.
\label{D:loop-fusion_bundle}
\end{defn}

\begin{thm}
The following are naturally in bijection:
\begin{enumerate}
[{\normalfont (i)}]
\item The set of doubled gerbes $(L, IX, X^2)$ up to strong isomorphism.
\item The set of loop-fusion line bundles on $LX$ up to isomorphism.
\item $H^3(X; \bbZ)$.
\end{enumerate}
\label{T:lf_bijection}
\end{thm}
\begin{proof}
Equivalence of the first two is a consequence of Lemma~\ref{L:prod_simp_LX} and Definition~\ref{D:loop-fusion_bundle}. 
Doubled gerbes $(L, IX, X^2)$ are classified by their Dixmier-Douady class,
which descends to $X$, as noted above,
and that every element in $H^3(X; \bbZ)$ is represented by a doubled gerbe (equivalently, loop-fusion line bundle) $L \to LX$ follows from Theorem~\ref{T:lf_trans} below.
\end{proof}

\begin{rmk}
The figure-of-eight structure is weaker than other conditions that have
been considered in the literature, such as thin homotopy equivariance in
\cite{MR2980505}, or reparameterization equivariance in
\cite{KM-equivalence}, which likewise identify categories of fusion line
bundles on $LX$ with gerbes on $X$.
\end{rmk}

\subsection{Loop-fusion cohomology} \label{S:lf_cohom}

In fact, applying the above considerations to \v Cech theory in place of line bundles leads to a general result, which recovers
the main theorem in our previous paper \cite{MR3391882}.
There we defined {\em loop-fusion cohomology} on $L X$, which in the present
language is equivalent to the group
\[
	\vH_{\mathrm{lf}}^\ell(L X; A) = \kernel \set{\pa : \vH^\ell_Z(LX; A) \to \vH_Z^\ell(L_8X; A)}.
\]
In particular, the set $\vH_{\mathrm{lf}}^2(LX; \bbZ)$ classifies loop-fusion line bundles up to isomorphism.
\begin{thm}[\cite{MR3391882}]
For each $\ell \in \bbN$ and topological abelian group $A$, there is an isomorphism 
\[
	\vH^\ell(X; A) \cong \vH_{\mathrm{lf}}^{\ell-1}(L X; A).
\]
\label{T:lf_trans}
\end{thm}
\noindent
It is additionally shown in \cite{MR3391882} that the isomorphism descends via
the forgetful map $\vH^\bullet_{\mathrm{lf}}(L X; A) \to \vH^\bullet(LX;
A)$ to the transgression homomorphism $\vH^\ell(X; A) \to \vH^{\ell-1}(L X; A)$;
recall that the latter
is defined by composing the pullback along the evaluation map $S^1 \times LX \to X$ with the pushforward along the projection $S^1\times LX \to LX$ (given by cap product with
the fundamental class of $S^1$).

\begin{proof}
%
The result follows from naturality of the isomorphism \eqref{E:SS_isom}, and exactness of \eqref{E:product_cohomology_seq}.
Applied to the three maps from $I_3^{[k]}X$ to $I_2^{[k]}X$ this yields an isomorphism (omitting
the coefficient group for brevity)
\begin{equation}
\begin{gathered}
	\kernel \pa \cap \kernel\pi^\ast \subset \vH^\ell(X^2) 
	\\ \cong \kernel \set{\pa : \vH^{\ell-1}_Z(I_2^{[2]}X)/d\vH^{\ell-1}(I_2 X) 
	\to \vH^{\ell-1}_Z(I_3^{[2]}X)/d\vH^{\ell-1}(I_3 X)}
	\label{E:lf_trans_initial_iso}
\end{gathered}
\end{equation}
However, as noted in the proof of Lemma~\ref{L:prod_simp_LX}, the deformation
retraction of the free path spaces $I_k X$ onto $X$ implies that $d :
\vH^{\ell-1}(I_kX; A) \to \vH^{\ell-1}(I_k^{[2]} X; A)$ is trivial, so the
quotients in \eqref{E:lf_trans_initial_iso} disappear.
Moreover, by exactness of \eqref{E:product_cohomology_seq}, the kernel of $\pa$
in $\vH^\ell(X^2; A)$ is the image of $\vH^\ell(X; A)$ and this is
automatically in the kernel of $\pi^\ast : \vH^\ell(X^2; A) \to \vH^\ell(I_2X; A)$
under the retraction $I_2X \simeq X$, so \eqref{E:lf_trans_initial_iso}
simplifies to
\[
	\vH^\ell(X; A) \cong \kernel \set{\pa : \vH^{\ell-1}_Z(I_2^{[2]}X; A) \to \vH^{\ell-1}_Z(I_3^{[2]}X; A)} = \vH_{\mathrm{lf}}^{\ell-1}(L X; A)
\]
as claimed.
\end{proof}

\section{Bundle bigerbes} \label{S:bundle_bigerbes}
\subsection{Locally split squares} \label{S:lssquares}

Bigerbes as introduced below are based on the following notion.
\begin{defn}
A commutative diagram 
\begin{equation}
\begin{tikzcd}
	Y_2 \ar[d,"\pi_2"'] & W \ar[l] \ar[d] 
	\\ X & Y_1 \ar[l,"\pi_1"]
\end{tikzcd}
\label{E:lssquare}
\end{equation}
is a {\em locally split square} if $Y_i \to X$, $i = 1,2$ and the induced
map $W \to Y_1\times_X Y_2$ are locally split.
\label{D:lssquare}
\end{defn}
\noindent There is manifest symmetry in the definition.
\blue{Note that in fact all four maps in \eqref{E:lssquare} are locally split; indeed, if $s_1 : \Et{\ecU} \to Y_1$ are local sections of $\pi_1$ for some cover $\ecU$ of $X$,
then $1\times (\wt s_1 \circ \pi_2) : \Et{\pi_2^\inv\ecU} \to Y_1\times_X Y_2$ form local sections of the projection $Y_1\times_X Y_2 \to Y_2$; moreover,
refining this cover $\pi_2^\inv \ecU$ of $Y_2$ if necessary, these may then be composed with local sections of $W \to Y_1 \times_X Y_2$ to obtain local sections of $W \to Y_2$, and similarly for $Y_1$.}


\begin{added}
\begin{defn}
Given a locally split square as above, an \emph{admissible set of covers} is a set of covers $\cU_X$, $\cU_{Y_1}$, $\cU_{Y_2}$, and $\cU_{W}$ for $X$, $Y_1$, $Y_2$, and $W$, respectively
for which each pair $(\cU_X, \cU_{Y_1})$, $(\cU_X, \cU_{Y_2})$, $(\cU_{Y_1}, \cU_W)$, and $(\cU_{Y_2}, \cU_W)$ is an admissible pair and such that the following four diagrams
commute:
\[
\begin{tikzcd}[sep=small]
	\Et{\cU_{Y_2}} \ar[d] & \Et{\cU_{W}} \ar[l] \ar[d]
	\\ \Et{\cU_X} & \Et{\cU_{Y_1}} \ar[l]
\end{tikzcd}
\quad
\begin{tikzcd}[sep=small]
	\Et{\cU_{Y_2}} \ar[r] & \Et{\cU_W}
	\\ \Et{\cU_X} \ar[r] \ar[u] & \Et{\cU_{Y_1}} \ar[u]
\end{tikzcd}
\]\[
\begin{tikzcd}[sep=small]
	\Et{\cU_{Y_2}} \ar[d] \ar[r] & \Et{\cU_W}  \ar[d]
	\\ \Et{\cU_X} \ar[r]  & \Et{\cU_{Y_1}}
\end{tikzcd}
\quad
\begin{tikzcd}[sep=small]
	\Et{\cU_{Y_2}}  & \Et{\cU_W}  \ar[l] 
	\\ \Et{\cU_X} \ar[u]   & \Et{\cU_{Y_1}} \ar[l] \ar[u]
\end{tikzcd}
\]
\label{D:admissible_set}
\end{defn}

\begin{lem}
Given any set of covers for the spaces in a locally split square $(X, Y_1, Y_2, W)$, there exist refinements constituting an admissible set.

In fact, given any finite set of locally split squares over $X$ which may share some of the same maps $\pi_i : Y_i \to X$ and any initial set of covers of these,
there exists a set of covers for all spaces which forms an admissible set for each square seperately, and for which the section maps lift any initially prescribed local sections of the maps $Y_i \to X$ or $W \to Y_1 \times_X Y_2$.  
\label{L:admissible_set_exists}
\end{lem}
\begin{proof}
Let $\cU'_\bullet$, $\bullet \in \set{X,Y_1,Y_2,W}$ be the initial arbitrary set of covers for a given square.
Without loss of generality by the stronger statement of Lemma~\ref{Addedrbm}, we may assume that $(\cU'_X, \cU'_{Y_i})$, $i =1,2$ form an admissible pair lifting
prescribed local sections $s_1$ and $s_2$ of $\pi_1$ and $\pi_2$.
%
Under this condition, it follows that $\cU'_{Y_i}$ support sections $\sigma_i$ of the fiber product projections $Y_1 \times_X Y_2 \to Y_i$ 
defined by 
$\sigma_1 := 1 \times (s_2 \circ \wt \pi_1)$
and 
$\sigma_2 := (s_1 \circ \wt \pi_2) \times 1$,
respectively (here $1$ denotes the inclusion of the cover $1 : \cU'_{Y_i} \to Y_i$), giving a diagram
\[
\begin{tikzcd}
	\Et{\cU'_{Y_2}} \ar[d,shift left,"\wt \pi_2"] \ar[r, "\sigma_2"] & Y_1\times_X Y_2 
	\\\Et{\cU'_X} \ar[u,shift left, "\wt s_2"] \ar[r, shift left,"\wt s_1"] & \Et{\cU'_{Y_1}} \ar[u,swap, "\sigma_1"] \ar[l,shift left, "\wt \pi_1"]
\end{tikzcd}
\]
for which the following three squares commute:
\[
\begin{tikzcd}[sep=small]
	\Et{\cU'_{Y_2}} \ar[r] & Y_1\times_X Y_2 
	\\\Et{\cU'_X} \ar[r] \ar[u]  & \Et{\cU'_{Y_1}}  \ar[u]
\end{tikzcd}
\quad 
\begin{tikzcd}[sep=small]
	\Et{\cU'_{Y_2}} \ar[r] \ar[d] & Y_1\times_X Y_2  \ar[d]
	\\\Et{\cU'_X} \ar[r]  & Y_1  
\end{tikzcd}
\quad
\begin{tikzcd}[sep=small]
	{Y_2}  & Y_1\times_X Y_2 \ar[l]
	\\\Et{\cU'_X} \ar[u]  & \Et{\cU'_{Y_1}} \ar[u] \ar[l]
\end{tikzcd}
\]
We also fix an initial cover $\cU'_{Y_1\times_X Y_2}$ of $Y_1 \times_X Y_2$ which supports a prescribed section $t : \Et{\cU'_{Y_1\times_X \times Y_2}} \to W$ of the locally split map $p : W \to Y_1\times_X Y_2$. 

This is the starting point, and wet set $\cU^0_\bullet = \cU'_\bullet$, for $\bullet \in \set{X,Y_1,Y_2,Y_1\times_X Y_2,W}$, where the superscript denotes the step in the process. 
As in the proof of Lemma~\ref{Addedrbm}, we proceed to refine the covers in two steps, first ``pulling down'' by the sections, and then ``pulling up'' by the projections.

\noindent\textit{Step 1}: Starting at the top right, we subsequently set $\cU^1_W = \cU^0_W$, $\cU^1_{Y_1\times_X Y_2} = \cU^0_{Y_1\times_X Y_2} \cap t^\inv(\cU^1_W)$, 
$\cU^1_{Y_i} = \cU^0_{Y_i} \cap \sigma_i^\inv(\cU^1_{Y_1\times_X Y_2})$, and 
\begin{equation}
\begin{aligned}
	\cU^1_X &= \cU^0_X \cap (\sigma_1 \circ \wt s_1)^\inv(\cU^1_{Y_1 \times_X Y_2})
	\cong \cU^0_X \cap \wt s_1^\inv(\cU^1_{Y_1})
	\\&\cong \cU^0_X \cap (\sigma_2 \circ \wt s_2)^\inv(\cU^1_{Y_1 \times_X Y_2})
	\cong \cU^0_X \cap \wt s_2^\inv(\cU^1_{Y_2})
\end{aligned}
	\label{E:ad_set_step1_isos}
\end{equation}
after which we have maps
of covers constituting the commutative diagram
\begin{equation}
\begin{tikzcd}
	&& \Et{\cU^1_W}
	\\ \Et{\cU^1_{Y_2}} \ar[r, "\wt \sigma_2"] & \Et{\cU^1_{Y_1\times_X Y_2}} \ar[ur, "\wt t"]
	\\ \Et{\cU^1_X} \ar[u, "\wt s_2"] \ar[r, "\wt s_1"] & \Et{\cU^1_{Y_1}} \ar[u, "\wt \sigma_1"]
\end{tikzcd}
	\label{E:square_stage_1}
\end{equation}
and which each constitute sections the appropriate map $Y_i \to X$, $Y_1\times_X Y_2 \to Y_i$, or $W \to Y_1\times_X Y_2$. 
The only nontrivial assertion is the horizontal identifications in \eqref{E:ad_set_step1_isos}, but this follows easily from viewing the refinements as pullbacks in light of naturality of pullback.
Alternatively it can be checked directly: given $U \in \cU^0_X$ mapping by $\wt s_1$ into $V_1 \in \cU^0_{Y_1}$, say, the latter becomes refined into sets of the form $V_1 \cap \sigma_1^\inv(W)$ 
for $W \in \cU^1_{Y_1\times_X Y_2}$, and then $U$ decomposes into sets of the form $U\cap \wt s_1^\inv(V_1 \cap \sigma_1^\inv(W)) = U \cap (\sigma_1 \circ \wt s_1)^\inv(W)$. 
For the stronger version involving multiple split squares over $X$, we note that $\cU^1_{Y_i}$ and then subsequently $\cU^1_X$ may be further refined while retaining the diagram \eqref{E:square_stage_1};
in particular, we can replace these by the mutual refinements of the Step 1 covers for each locally split square in which the spaces are involved.

\noindent\textit{Step 2}: Starting at the bottom left, we subsequently set 
$\cU^2_X = \cU^1_X$, 
$\cU^2_{Y_i} = \cU^1_{Y_i} \cap \pi_i^\inv(\cU^2_X)$, 
$\cU^2_{Y_1\times_X Y_2} = \cU^1_{Y_1\times_X Y_2} \cap (\cU^2_{Y_1} \times_{\cU^2_X} \cU^2_{Y_2})$,
and $\cU^2_W = \cU^1_W \cap p^\inv(\cU^2_{Y_1\times_X Y_2})$,
after which we have maps of covers
\[
\begin{tikzcd}
	&& \Et{\cU^2_W} \ar[dl,shift left, "\wt p"]
	\\ \Et{\cU^2_{Y_2}} \ar[r,shift left, "\wt \sigma_2"] \ar[d, shift left,"\wt \pi_2"] & \Et{\cU^2_{Y_1\times_X Y_2}} \ar[d,shift left, "\wt \pr_1"] \ar[l,shift left, "\wt \pr_2"] \ar[ur,shift left, "\wt t"]
	\\ \Et{\cU^2_X} \ar[u,shift left, "\wt s_2"] \ar[r,shift left, "\wt s_1"] & \Et{\cU^2_{Y_1}} \ar[l,shift left, "\wt \pi_1"] \ar[u,shift left, "\wt \sigma_1"]
\end{tikzcd}
\]
such that each subsequent pair is admissible.
Indeed, the only assertion to check
admissibility of the pair $(\cU^2_{Y_i}, \cU^2_{Y_1\times_X Y_2})$,
since $\cU^2_{Y_1\times_X Y_2}$ is defined in a different way here compared to Lemma~\ref{Addedrbm}.

To check this, first note that $\cU^2_{Y_1 \times_X Y_2}$ is in general finer than either $\cU^1_{Y_1\times_X Y_2} \cap \pr_i^\inv(\cU^2_{Y_i})$ for $i = 1,2$, and in fact it factors through both 
of these since $\pr_1^\inv(\cU^2_{Y_1}) = \cU^2_{Y_1} \times_X Y_2$ and $\pr_2^\inv(\cU^2_{Y_2}) = Y_1 \times_X \cU^2_{Y_2}$. 
It follows from this that $\wt \pr_i$ lift to $\cU^2_{Y_1\times_X Y_2}$ and $\wt \pi_1 \circ \wt \pr_1 = \wt \pi_2 \circ \wt \pr_2$.
To see that $\wt\sigma_i$ lift, consider for example a $V \in \cU^1_{Y_2}$ mapping by $\wt \sigma_2$ to $W \in \cU^1_{Y_1\times_X Y_2}$ after step 1. 
In Step 2, $V$ is refined into sets of the form $V \cap \pi_2^\inv(U)$ for $U \in \cU^2_X$, and we have a diagram
\[
\begin{tikzcd}[sep=small]
	V \cap \pi_2^\inv(U) \ar[r, "\wt \sigma_2"] \ar[d, "\wt \pi_2"] & W
	\\ U \ar[r, "\wt s_1"] & V' \cap \pi_1^\inv(U)
\end{tikzcd}
\]
for some $V' \cap \pi_1^\inv(U) \in \cU^2_{Y_1}$.
But then $\wt \sigma_2$ lifts canonically into a map 
\[
	\wt \sigma_2 : V \cap \pi_2^\inv(U) \to W \cap (V \times_{U} V') \in \cU^2_{Y_1\times_X Y_2}.
\]
Finally, to see that the final (pairwise intersection) property for admissibility holds, take $V_1,V_2 \in \cU^2_{Y_2}$ with non-empty intersection. 
By construction the $V_i$ are associated to $U_i \in \cU^2_X$ for $i = 1,2$, and by the intersection property for $\wt s_1$ it follows that 
$(\wt s_1)_{U_1} : U_1 \cap U_2 \to V'_1 \cap V'_2$ (more precisely, the restriction to $U_1 \cap U_2$ of $\wt s_1$ as defined on $U_1$) maps into some $V'_2 \in \cU^2_{Y_1}$ over $U_2$.
It follows then that the restriction to $V_1 \cap V_2$ of $\wt \sigma_2$ (as defined on $V_1$), in fact maps into $(W \cap V_1 \times_{U_1} V'_1) \cap (W \cap V_2 \times_{U_2} V'_2)$.
In case multiple split squares over $X$ are being considered, note that the refinements in Step 2 are mutually consistent, so there are no additional considerations to this step in this case.

The admissible set is given by $\cU_\bullet = \cU^2_\bullet$ for $\bullet \in \set{X,Y_1,Y_2,W}$, with maps of covers between the $\cU_{Y_i}$ and $\cU_W$ given by 
compositions $\wt \pr_i \circ \wt p$ and $\wt t \circ \wt \sigma_i$.
The required commutativity of the four squares follows from the like commutativity of the squares in step 2 for the spaces $X$, $Y_1$, $Y_2$, and $Y_1\times_X Y_2$. 
\end{proof}
\end{added}

As in \S\ref{S:bundle_gerbes}, let $Y_i^{[k]}$ be the $k$-fold fiber
product $Y_i \times_X \cdots \times_X Y_i$ for $i = 1,2.$ Then
$Y_1^{[\bullet]}$ and $Y_2^{[\bullet]}$ each form simplicial spaces over
$X,$ giving the bounding column and row in \eqref{E:locally_split_2d_diagram} below.

Setting $W^{[1,1]}=W$ and
\[
	W^{[1,k]} = W \times_{Y_1} \cdots \times_{Y_1} W,
	\quad \text{and} \quad
	W^{[k,1]} = W\times_{Y_2} \cdots \times_{Y_2} W,
\]
with projection maps $W^{[1,k]} \to W^{[1,k-1]}$ and $W^{[k,1]} \to
W^{[k-1,1]}$ gives simplicial spaces over $Y_1$ and $Y_2$ 
extending above and to the right of $W^{[1,1]}$ in \eqref{E:locally_split_2d_diagram}.

That the rest of the quadrant can then be filled out unambiguously by fiber
products is a consequence of the following result.

\begin{prop}
For each $n$ and $m$, there is a natural isomorphism
\[
	W^{[m,n]} := \overbrace{W^{[m,1]}\times_{Y_1^{[m]}} \cdots \times_{Y_1^{[m]}} W^{[m,1]}}^{n\text{ times}} 
	\cong 
	\overbrace{W^{[1,n]}\times_{Y_2^{[n]}} \cdots \times_{Y_2^{[n]}} W^{[1,n]}}^{m\text{ times}}.
\]
\label{P:fiber_quadrant}
\end{prop}
\begin{proof}
Both sides may be identified with the set of tuples $(w_{i,j} : 1 \leq i \leq
m,\ 1 \leq j \leq n) \in W^{mn}$ such that for each $i$,
$(w_{i,1},\ldots,w_{i,n})$ all map to a fixed $y_{1,i} \in Y_1$ and for each
$j$, $(w_{1,j},\ldots,w_{m,j})$ all map to a fixed $y_{2,j} \in Y_2$, and where
every $y_{1,i}$ and $y_{2,j}$ sit over a fixed $x \in X$.
\end{proof}

The spaces $W^{[\bullet,\bullet]}$ in the resulting diagram
\begin{equation}
\begin{tikzcd}[row sep=small, column sep=small]
	\ar[d,dotted, no head] &  \ar[d,dotted, no head] & \ar[d,dotted, no head] & \ar[d,dotted, no head] 
	\\ Y_2^{[3]} \ardthree & W^{[1,3]} \ardthree \ar[l] & W^{[2,3]} \ardthree \arltwo & W^{[3,3]} \ardthree \arlthree & \ar[l,dotted,no head]
	\\ Y_2^{[2]} \ardtwo & W^{[1,2]} \ardtwo \ar[l] & W^{[2,2]} \ardtwo \arltwo & W^{[3,2]} \ardtwo \arlthree & \ar[l,dotted,no head]
	\\ Y_2 \ar[d] & W^{[1,1]} \ar[d] \ar[l] & W^{[2,1]} \ar[d] \arltwo & W^{[3,1]} \ar[d] \arlthree & \ar[l,dotted,no head]
	\\ X  & Y_1 \ar[l]  & Y_1^{[2]} \arltwo  & Y_1^{[3]} \arlthree & \ar[l,dotted,no head]
\end{tikzcd}
	\label{E:locally_split_2d_diagram}
\end{equation}
form a {\em bisimplicial space} over $X$, meaning a functor
$\Delta^\mathrm{op}\times \Delta^\mathrm{op} \to \mathsf{Top}/X$ where
$\Delta$ is the simplex category. In particular, $W^{[m,\bullet]}$ and
$W^{[\bullet,n]}$ are simplicial spaces over $Y_1^{[m]}$ and $Y_2^{[n]}$,
respectively, and the squares commute for consistent choices of maps. For
notational convenience, we also set $W^{[k,0]} = Y_1^{[k]}$, $W^{[0,k]} =
Y_2^{[k]}$, and $W^{[0,0]} = X$.

\subsection{Bigerbes} \label{S:bigerbes}

If $L \to W^{[m,n]}$ is a line bundle over one of the spaces in
\eqref{E:locally_split_2d_diagram} then its two simplicial differentials are
\[
\begin{aligned}
	d_1L &= \bigotimes_{i = 0}^m (\pi^1_i)^\ast L^{(-1)^i} \to W^{[m+1,n]}
	\quad \text{and} 
	\\ d_2L &= \bigotimes_{i = 0}^n (\pi^2_i)^\ast L^{(-1)^i} \to W^{[m,n+1]},
\end{aligned}
\]
where $\pi^1_j : W^{[m+1,n]} \to W^{[m,n]}$, $0 \leq j \leq m$
and $\pi^2_j : W^{[m,n+1]} \to W^{[m,n]}$, $0 \leq j \leq n$ denote the fiber
projection maps.
The bundles $d_1d_1L$ and $d_2d_2 L$ are canonically trivial, and there is a
natural isomorphism $d_1d_2 L \cong d_2 d_1 L$.

\begin{defn} A {\em bigerbe} consists of a
  locally split square $(W,Y_2,Y_1,X)$, a line bundle $L \to W^{[2,2]}$,
  and trivializations of $d_1 L$ and $d_2 L$, which induce the same
  trivialization of $d_1d_2 L$ and which induce the canonical
  trivializations of $d_1^2 L$ and $d_2^2L$.
We denote the bigerbe by $(L, W, Y_2, Y_1, X)$ or simply $L.$

For reasons that will become clear below, the order of the spaces $Y_1$ and $Y_2$, or equivalently the orientation of the square \eqref{E:lssquare},
is part of the data of the bigerbe.
\label{D:bundle_bigerbe}
\end{defn}

\begin{defn}\label{D:bigerbe_trivn}
If $Q_1 \to W^{[1,2]}$ and $Q_2 \to W^{[2,1]}$ are line bundles which are
simplicial with respect to $d_2$ and $d_1$, respectively --- so $d_2 Q_1$ over $W^{[1,3]}$
  is equipped with a trivialization inducing the canonical trivialization
  of $d^2_2 Q_1$ and similarly for $Q_2$ --- then $d_1 Q_1 \otimes d_2 Q_2^\inv$ 
  has a canonical bigerbe structure.
 A bigerbe $L$ is said to be {\em trivial} if
\[ L \cong d_1
	  Q_1\otimes d_2 Q_2^\inv, 
\] 
for $Q_1$ and $Q_2$ as above, with the isomorphism identifying the bigerbe structure on $L$ with the canonical
one on $d_1 Q_1 \otimes d_2Q_2^\inv$;
such an isomorphism is referred to as a {\em trivialization} of $L$.
\end{defn}
\noindent In particular, $L$ is trivial if either
\begin{enumerate} [{\normalfont (i)}]
\item $L \cong d_1 P$ where $P \to
    W^{[1,2]}$ is a line bundle with trivialization $d_2 P \cong \ul \bbC$
    inducing the canonical trivialization of $d_2^2 P$, or
\item $L \cong
    d_2 Q$ where $Q \to W^{[2,1]}$ is a line bundle with trivialization
    $d_1 Q \cong \ul \bbC$ inducing the canonical trivialization of $d_1^2
    Q$,
\end{enumerate}
as in either case we can take the trivial bundle on the other factor.

As for ordinary bundle gerbes, we proceed to define pullbacks and products of bigerbes.
\begin{lem}
If $(W,Y_2,Y_1,X)$ and $(W',Y'_2,Y'_1,X)$ are locally split squares over $X$ and $f : X' \to X$ is a continuous map, then
\begin{enumerate}
[{\normalfont (i)}]
\item\label{I:pullback_prod_lss_one} $\big(f^\ast(W),f^\ast(Y_2),f^\ast(Y_1),X'\big)$ is a locally split square over $X'$, and
\item\label{I:pullback_prod_lss_two} $(W\times_X W',Y_2\times_X Y'_2, Y_1\times_X Y'_1,X)$ is a locally split square over $X$.
\end{enumerate}
\label{L:pullback_prod_lss}
\end{lem}
\begin{proof}
By hypothesis there are covers $\ecU_i \to X$ admitting sections $s^i : \ecU_i
\to Y_i$ of $\pi_i : Y_i\to X$ and a cover $\ecV\to Y_1\times_X Y_2$ admitting
a section $t : \ecV\to W$ of the universal map $p : W \to Y_1\times_X Y_2$. 

Pullback of these by $f$ gives covers $f^\inv \ecU_i \to X'$ and sections
$f^\ast s^i : f^\inv \ecU_i \to f^\ast Y_i$, with $f^\ast s^i = 1\times s^i
\circ f : f^\inv \ecU_i \to X'\times_X Y_i = f^\ast Y_i$, where the section is
composed with the lift $f : f^\inv \ecU_i \to \ecU_i$ and $1$ denotes the
inclusion map $f^\inv \ecU_i \to X'$ of covers.
Similarly, if $\wt f : f^\ast (Y_1\times_X Y_2) \to Y_1\times_X Y_2$
denotes the natural lift over $f$, then $\wt f^\inv \ecV\to
f^\ast(Y_1\times_X Y_2)$ supports the section $\wt f^\ast t : \wt f^\inv
\ecV\to f^\ast W$ of the natural map $f^\ast W \to f^\ast(Y_1\times_X Y_2)
\cong (f^\ast Y_1)\times_{X'} (f^\ast Y_2)$, proving
\eqref{I:pullback_prod_lss_one}.

For the fiber product, the cover $\ecU_i \cap \ecU'_i \to X$ admits sections
$s^i\times (s')^i$ of $Y_i\times_X Y'_i$, and then
$(Y_1\times_XY'_1)\times_X (Y_2 \times_X Y_2') \cong (Y_1 \times_X Y_2)
\times_X (Y_1'\times_X Y'_2)$ may be equipped with the cover $\ecV\times_X
\ecV'$, which admits the section $t \times t'$ to $W\times_X W'$, proving
\eqref{I:pullback_prod_lss_two}.
\end{proof}

\begin{defn} If $L \to W^{[2,2]}$ is a bigerbe with respect to the locally
  split square $(W,Y_2,Y_1,X)$, and $f : X' \to X$ is a continuous map,
  then the {\em pullback} of $L$ is the bigerbe $\wt f^\ast L \to
  f^\ast(W^{[2,2]})$ with respect to the locally split square 
$\big(f^\ast(W),f^\ast(Y_2),f^\ast(Y_1),X'\big)$.

If $L = (L,W,Y_2,Y_1,X)$ and $L' = (L',W',Y_2',Y_1',X)$ are bigerbes on $X$, then the
{\em product} of $L$ and $L'$ is the bigerbe 
\begin{equation*}
(L\otimes L',W\times_X W',Y_2\times_X Y'_2, Y_1\times_X Y'_1, X). 
\label{rbm.79}\end{equation*}
\label{D:pullback_product}
\end{defn}

Next we define (strong) morphisms and stable isomorphisms for bigerbes.
A {\em morphism} of locally split squares $(W',Y_2',Y_1',X') \to (W,Y_2,Y_1,X)$
is a collection of maps from each space in the first square to the
corresponding space in the second, with each of the relevant squares commuting.
As for bundle gerbes we do not require compatibility in the sense of
\S\ref{S:cech} of the locally split maps of the first square with those of the
second.
By naturality of fiber products, these maps extend to maps ${W'}^{[m,n]} \to
W^{[m,n]}$ for each $(m,n) \in \bbN_0^2$ commuting with the various fiber
projections in both directions.
By abuse of notation we will denote all such maps by a single letter, say $f:
{W'}^{[m,n]} \to W^{[m,n]}$.
\begin{defn}
If $(L,W,Y_2,Y_1,X)$ and $(L',W',Y'_2,Y'_1,X')$ are bigerbes over $X$ and $X'$ respectively, then a {\em (strong) morphism}
from $L'$ to $L$ consists of a morphism $f : (W',Y_2',Y_1',X') \to (W,Y_2,Y_1,X)$ and an isomorphism $L' \cong f^\ast L$ over ${W'}^{[2,2]}$ 
which intertwines the sections of $d_i L$ and $d_i f^\ast L$ for $i = 1,2$.
A {\em (strong) isomorphism} is a morphism for which $X = X'$ and $f = \Id : X \to X'$.

Finally, a {\em stable isomorphism} of bigerbes $L$ and $L'$ over $X$ is a (strong) isomorphism
\[
	L\otimes T \cong L'\otimes T
\]
where $T$ and $T'$ are trivial bigerbes.
\label{D:bigerbe_morphisms}
\end{defn}

\subsection{The 4-class of a bigerbe} \label{S:bigerbe_four_class}

\begin{added}
We proceed to define the cohomology 4-class of a bigerbe using \v Cech theory as in \S\ref{S:simplicial_cech} and \S\ref{S:DD_class}, starting
with the following generalization of Lemma~\ref{L:admissible_cover_fib_prod}.
\begin{lem}
Fix an admissible set of covers for a locally split square $(X,Y_1,Y_2,W)$.
%
Then 
\begin{enumerate}
[{\normalfont (i)}]
\item 
\[
	\cU_{W}^{[m,n]} := \cU_{W}^{[m,1]} \times_{\cU_{Y_1^{[m]}}} \cdots \times_{\cU_{Y_1^{[m]}}} \cU_{W}^{[m,1]} 
	\cong \cU_W^{[1,n]}\times_{\cU_{Y_2}^{[n]}} \cdots \times_{\cU_{Y_2}^{[n]}} \cU_W^{[1,n]}
\]
is a well-defined cover of $W^{[m,n]}$ for each $m$ and $n$.
\item
The various projection maps $\pi^1_i : W^{[m,n]} \to W^{[m-1,n]}$ and $\pi_j^2 : W^{[m,n]} \to W^{[m,n-1]}$ 
for $0 \leq i \leq m-1$ and $0 \leq j \leq n - 1$ lift canonically
to maps of covers 
$\wt \pi^1_i : \Et{(\cU_W^{[m,n]})^{(\ell)}} \to \Et{(\cU_W^{[m-1,n]})^{(\ell)}}$ 
and
$\wt \pi^2_j : \Et{(\cU_W^{[m,n]})^{(\ell)}} \to \Et{(\cU_W^{[m,n-1]})^{(\ell)}}$ 
for each $\ell$.
\item 
The lifted sections on the admissible set of covers determine maps of covers
$\wt s^1_\ell : \Et{(\cU_W^{[m-1,n]})^{(\ell)}} \to \Et{(\cU_W^{[m,n]})^{(\ell)}}$ 
and
$\wt s^2_\ell : \Et{(\cU_W^{[m,n-1]})^{(\ell)}} \to \Et{(\cU_W^{[m,n]})^{(\ell)}}$ 
for each $\ell$
such that $\wt s^1_\ell$ commutes with the projections $\wt \pi^2_j$ and satisfies 
\[
	\wt \pi^1_j \wt s^1_\ell = \begin{cases} 1 & j = 0 \\ \wt s^1_\ell \wt \pi^1_{j-1} & 1 \leq j \leq m \end{cases}
\]
and similarly for $\wt s^2_\ell$ with respect to $\wt \pi^1_j$ and $\wt \pi^2_j$. 
\end{enumerate}
\label{L:triple_complex}
\end{lem}

\begin{proof}
The same argument as in Proposition~\ref{P:fiber_quadrant}
shows that $\cU_W^{[m,n]}$ is well-defined, namely both sides are identified with tuples
\[
	(w_{i,j} \in W_{i,j} : 1 \leq i \leq m, \ 1 \leq j \leq n), \qquad W_{i,j} \in \cU_W
\]
such that for each $i$, $(w_{i,1},\ldots,w_{i,n})$ all map to a fixed $y_{1,i} \in V_{1,i}$ for some $V_{1,i} \in \cU_{Y_1}$,
and for each $j$, $(w_{1,j},\ldots,w_{m,j})$ all map to a fixed $y_{2,j} \in V_{2,j}$ for some $V_{2,j} \in \cU_{Y_2}$, with 
$(y_{1,1},\ldots,y_{1,m}) \in V_{1,1} \times_{U} \cdots \times_{U} V_{1,m} \in \cU_{Y_1^{[m]}}$ and $(y_{2,1},\ldots,y_{2,n}) \in V_{2,1} \times_U \cdots \times_U V_{2,n} \in \cU_{Y_2^{[n]}}$.

Then the arguments in 
Lemma~\ref{L:admissible_cover_fib_prod} apply, viewing each row or column of \eqref{E:locally_split_2d_diagram}
as a simplicial space, with commutativity of the vertical and horizontal maps of covers deriving from their commutativity at the bottom level for 
the locally split square.
\end{proof}
\end{added}

The abelian groups of \v Cech cochains
$\vC_{\blue{\ecU}}^\ell(W^{[j,k]}; A) \blue{:=\vC_{\ecU_W^{[j,k]}}^\ell(W^{[j,k]}; A)}$ 
are
defined for each $\ell$, $j$, and $k \in \bbN_0$.
Using Lemma~\ref{L:triple_complex} we may define 
\[
\begin{aligned}
	d_1 &= \sum_{j=0}^{m} (-1)^j(\pi^1_j)^\ast : \vC_{\blue{\ecU}}^{\ell}(W^{[m,n]}; A) \to \vC_{\blue{\ecU}}^\ell(W^{[m+1,n]}; A), \quad \text{and}
	\\ d_2 &= \sum_{j=0}^{n} (-1)^j(\pi^2_j)^\ast : \vC_{\blue{\ecU}}^{\ell}(W^{[m,n]}; A) \to \vC_{\blue{\ecU}}^\ell(W^{[m,n+1]}; A)
\end{aligned}
\]
which are differentials commuting with one another and with the \v Cech
differential $\delta.$ Thus $(\vC_{\blue{\ecU}}^\bullet(W^{[\bullet,\bullet]}; A), \delta, d_1, d_2)$ forms a
triple complex, and combining 
\blue{Lemma~\ref{L:triple_complex} and 
Proposition~\ref{T:exact_simplicial}}
leads to the following result.

\begin{prop} 
\blue{
An admissible set of covers determines commuting}
 homotopy contractions for the $d_i$ which commute with the other
 simplicial differential, and for each fixed $\ell$ and $k$, the subcomplex
\begin{equation*}
\pns{\kernel\bset{d_1 : \vC_{\blue{\ecU}}^\ell(W^{[k,\bullet]}) \to
 \vC_{\blue{\ecU}}^\ell(W^{[k+1,\bullet]})}, d_2}
\label{rbm.61}\end{equation*}
is exact and similarly with indices reversed.
\label{P:triple_complex_exactness}
\end{prop}

Just as a bundle gerbe has a Dixmier Douady class in $H^3(X; \bbZ)$, a bigerbe
determines a characteristic class in $H^4(X; \bbZ).$ To see this, consider the 
truncation of the 
triple complex $(\vC_{\blue{\ecU}}^\bullet(W^{[\bullet,\bullet]}; A), \delta, d_1, d_2)$
which we denote by $(\vZ_{\blue{\ecU}}^\bullet(W^{[\bullet,\bullet]}; A), \delta, d_1, d_2)$, where 
\[
	\vZ_{\blue{\ecU}}^\ell(W^{[j,k]}; A) =
	\begin{cases}
	   \vC_{\blue{\ecU}}^\ell(W^{[j,k]}; A), & [j,k] = [1,1]
	\\ \kernel d_1 \subset \vC_{\blue{\ecU}}^\ell(W^{[2,1]}; A) & [j,k] = [2,1]
	\\ \kernel d_2 \subset \vC_{\blue{\ecU}}^\ell(W^{[1,2]}; A) & [j,k] = [1,2]
	\\ \kernel d_1 \cap \kernel d_2 \subset \vC_{\blue{\ecU}}^\ell(W^{[2,2]}; A) & [j,k] = [2,2]
	\\ 0 & \text{otherwise}
	\end{cases}
\]
Suppressing the \v Cech direction, we may depict the truncated complex as
\begin{equation}
\begin{tikzcd}[row sep=small, column sep=small]
	0 & 0
	\\ \vZ_{\blue{\ecU}}^\bullet(W^{[1,2]}; A) \ar[u,"d_2"] \ar[r, "d_1"] & \vZ_{\blue{\ecU}}^\bullet(W^{[2,2]}; A) \ar[u, "d_2"] \ar[r, "d_1"] & 0
	\\ \vC_{\blue{\ecU}}^\bullet(W^{[1,1]}; A) \ar[u, "d_2"] \ar[r, "d_1"] & \vZ_{\blue{\ecU}}^\bullet(W^{[2,1]}; A) \ar[u, "d_2"] \ar[r, "d_1"] & 0.
\end{tikzcd}
	\label{E:bigerbe_triple_complex}
\end{equation}
In particular, the leftmost column and bottom row of \eqref{E:bigerbe_triple_complex} are taken to have $d_i$ degree 
$0$.
Then following Convention~\ref{Conv:multicomplex}, the total differential on \eqref{E:bigerbe_triple_complex} is
\begin{equation}
	D = \delta + (-1)^\ell d_1 + (-1)^{\ell + m + 1} d_2 
	\quad \text{on $\vZ_{\blue{\ecU}}^\ell(W^{[m,n]}; A)$}
	\label{E:triple_complex_total_diff}
\end{equation}
since $\vC_{\blue{\ecU}}^\bullet(W^{[m,n]})$ occupies the $(m-1,n-1)$ coordinate in the $(d_1,d_2)$ plane.

Employing a spectral sequence argument twice immediately yields the following result.

\begin{prop}
The triple complex 
$\big(\vZ_{\blue{\ecU}}^\bullet(W^{[\bullet,\bullet]}; A), \delta, d_1, d_2\big)$ has
total cohomology isomorphic to the ordinary cohomology $\vH_{\blue{\ecU}}^\bullet(X; A)$ of $X$.
\label{P:triple_complex_trivial}
\end{prop}

\begin{proof} The total differential \eqref{E:triple_complex_total_diff} of
  the $(\delta,d_1,d_2)$ triple complex can be written as $D = D_1 +
  (-1)^{\ell + m+1} d_2$ on $\vC_{\blue{\ecU}}^\ell(W^{[m,n]}; A)$, where $D_1 = \delta +
  (-1)^\ell d_1$ is the total differential of the $(\delta,d_1)$ double
  complex.
By exactness
of $d_2$, the total cohomology of the $(D_1,d_2)$ double complex is isomorphic
to the cohomology of the $D_1$ (double) complex $\vC_{\blue{\ecU}}^\bullet(Y_1^{[\bullet]}; A)$, which in turn is isomorphic to
$\vH_{\blue{\ecU}}^\bullet(X; A)$ as in Theorem~\ref{T:total_cohom_just_X}.
\end{proof}

\begin{lem}\label{L:Cech_bigerbe_class}
 The line bundle $L \to W^{[2,2]}$ of a bigerbe determines a
  pure cocycle $c(L) \in \vZ_{\blue{\ecU}}^1(W^{[2,2]}; \bbC^\ast)$ in the triple complex
  \eqref{E:bigerbe_triple_complex} \blue{for some admissible set of covers}, and conversely any line bundle with
  $c(L) \in \vZ_{\blue{\ecU}}^1(W^{[2,2]}; \bbC^\ast)$ determines a bigerbe. 
Moreover, the pure
  cocycle $c(L) \in \vZ_{\blue{\ecU}}^1(W^{[2,2]}; \bbC^\ast)$  is a coboundary if and
  only if $L$ admits a trivialization.
\end{lem}

\begin{proof}
The line bundle $L$ is represented 
by its `transition' Chern class 
on some cover, \blue{which by arguments along the lines of 
Lemma~\ref{L:refine_Yk} can be assumed to be of the form $\cU_{W^{[2,2]}}$ for some admissible set of covers,}
and hence by an element $c(L) \in \vC_{\blue{\ecU}}^1(W^{[2,2]}; \bbC^\ast)$ such
that $\delta c(L) = 0$. 
The simplicial trivializations of $d_i L$, $i =
1,2$ are represented by elements $\alpha_1 \in \vC_{\blue{\ecU}}^0(W^{[3,2]})$ and
$\alpha_2 \in \vC_{\blue{\ecU}}^0(W^{[2,3]})$ such that $d_i \alpha_i = 0$, $\delta
\alpha_i = d_i c(L)$, and $\pa(d_2 \alpha_1 - d_1 \alpha_2) = 0$.
In other words, 
the triple $(c(L), -\alpha_1,\alpha_2)$ forms a cocycle in the triple complex $(\vC_{\blue{\ecU}}^\bullet(W^{[\bullet,\bullet]}; \bbC^\ast),\delta, d_1,d_2)$.
Now, by exactness, we may obtain $d_i$ preimages $\beta_i$ of the $\alpha_i$, 
and then $c(L)$ can be altered by the image under $\delta$ of the $\beta_i$
to obtain a pure cocycle, which we again denote by $c(L) \in \vZ_{\blue{\ecU}}^1(W^{[2,2]}; \bbC^\ast).$

A coboundary for $c(L)$ in the triple complex consists of a triple $(\alpha, \beta, \gamma)$ where $\alpha \in \vZ_{\blue{\ecU}}^0(W^{[2,2]})$, 
$\beta \in \vZ_{\blue{\ecU}}^1(W^{[1,2]})$ and $\gamma \in \vZ_{\blue{\ecU}}^1(W^{[2,1]})$ such that $\delta \beta = 0$ and $\delta \gamma = 0$, and
\begin{equation}
	D(\alpha, \beta,\gamma) = \delta \alpha - d_1 \beta + d_2 \gamma = c(L).
	\label{E:bigerbe_trivial_coboundary}
\end{equation}
But this amounts precisely to saying that $\beta$ and $\gamma$ determine $d_1$ and $d_2$ simplicial line bundles 
$Q \to W^{[2,1]}$, 
and 
$P \to W^{[1,2]}$,
such that $L$ is isomorphic (with isomorphism determined by $\alpha$) to $d_1P \otimes d_2 Q^\inv$, i.e., $L$ is trivial.
Conversely, a trivialization of the bigerbe $L$ determines such a coboundary \eqref{E:bigerbe_trivial_coboundary}.
\end{proof}

\begin{defn}
Let $(L,W,Y_2,Y_1,X)$ be a bigerbe over $X$.
The {\em characteristic 4-class}
of $L$ is the image $G(L) \in H^4(X; \bbZ) \cong H^3(X; \bbC^\ast)$ of the hypercohomology class of $c(L) \in \vZ_{\blue{\ecU}}^1(W^{[2,2]}; \bbC^\ast)$
in the triple complex \eqref{E:bigerbe_triple_complex}.\end{defn}
\noindent For an explicit zig-zag construction of $G(L)$ from $c(L)$, see
\eqref{E:bigerbe_zig_one} and \eqref{E:bigerbe_zig_two} in the proof of
Theorem~\ref{T:bigerbe_representability} below.
\label{D:bigerbe_char_class}

Because of the need to introduce signs in the $(\delta,d_1,d_2)$ total
complex following Convention~\ref{Conv:multicomplex}, the sign of the class $G(L)
\in H^4(X; \bbZ)$ depends on the order of $Y_1$ and $Y_2$, which is to say the
orientation of the locally split square.
In particular, reversing the roles of $Y_1$ and $Y_2$ while keeping the bundle $L \to W^{[2,2]}$ fixed determines a bundle gerbe $L'$ with 
class $G(L') = -G(L)$.
Indeed, this follows from the fact that 
the total differential $D' = \delta + (-1)^\ell d_2 + (-1)^{\ell+n+1}d_1$
on $\vC_{\blue{\ecU}}^\ell(W^{[m,n]})$,
where we have interchanged the roles of $d_1$ and $d_2$, is intertwined with $D = \delta + (-1)^\ell d_2 + (-1)^{\ell+m+1}d_2$
by the automorphism $(-1)^{(m+1)(n+1)}$ of the triple complex. 
In particular this amounts to multiplication by $-1$ on $\vZ_{\blue{\ecU}}^\bullet(W^{[2,2]})$, exchanging $c(L)$ and $-c(L)$.

Alternatively, from the explicit zig-zag \eqref{E:bigerbe_zig_one} and \eqref{E:bigerbe_zig_two} it follows that $c(L)$ is the 
double transgression of $G(L) \in H^4(X; \bbZ)$ (in the sense of the isomorphism \eqref{E:SS_isom}) first to $H^3_Z(Y_1^{[2]}; \bbZ)$
and then to $H_Z^2(W^{[2,2]}; \bbZ)$.
The transgression the other way, first to $Y_2^{[2]}$ and then to $W^{[2,2]}$ has the opposite image $-c(L)$.

\begin{thm}
The characteristic 4-class $G(L)$ vanishes if and only if $L$ is trivial as a bigerbe, 
and is natural with respect to pullback, product and inverses in that
\[
	G(f^\ast L) = f^\ast G(L), 
	\quad G(L_1\otimes L_2) = G(L_1) + G(L_2),
	\quad G(L^\inv) = - G(L)
\]
A morphism $f : (L', W',Y_2',Y_1',X') \to (L, W,Y_2,Y_1,X)$
of bigerbes induces an equality $f^\ast G(L) = G(L')$, and 
two bigerbes $L$ and $L'$ over $X$ satisfy $G(L) = G(L')$ if and only if they are stably isomorphic.
\label{T:bigerbe_DD_class}
\end{thm}

\begin{proof}
That $G(L) = 0$ if and only if $L$ admits a trivialization was proved in Lemma~\ref{L:Cech_bigerbe_class}.
The pullback of a locally split square over $X$ by a continuous map $f : X' \to X$ induces
natural maps $f^\ast W^{[m,n]} \to W^{[m,n]}$ commuting with each $\pi^i_j$, and thus
a map 
$f^\ast \vC_{\blue{\ecU}}^\bullet(W^{[\bullet,\bullet]}) \to
\vC_{\blue{\ecU'}}^\bullet(f^\ast W^{[\bullet,\bullet]})$
of triple complexes, \blue{where $\ecU'_\bullet$ is an admissible set of covers refining $f^\inv(\ecU_\bullet)$.}
The naturality of $G$ with respect to pullbacks and morphisms is then a consequence
of the naturality of the spectral sequences which identify the total cohomology of the triple complex
with the cohomology of $X$ and $X'$, respectively.
Naturality with respect to products and inverses is a direct consequence of the fact that we can take $[L^\inv] = -c(L)$
and $[L_1\otimes L_2] = [\pr_1^\ast L_1] + [\pr_2^\ast L_2]$ as representatives.
Finally, if $L$ and $L'$ are stably isomorphic, then $G(L) = G(L')$ by
triviality and products, and conversely if $G(L) = G(L')$, then $L^\inv
\otimes L' = T$ is trivial, from which a stable isomorphism $L\otimes T
\cong L'$ may be constructed.
\end{proof}

\subsection{Representability of 4-classes} \label{S:rep_four_class}

To characterize those 4-classes which are represented by bigerbes over a given
locally split square we follow a similar argument to that in \S\ref{S:rep_of_three},
though it is necessary in this case to go further in a spectral sequence for the triple
complex. 
Consider the augmented triple complex
\begin{equation}
\begin{tikzcd}[row sep=small, column sep=small]
	0 & 0 & 0
	\\ \vZ^\bullet(Y_2^{[2]}) \ar[r,"d_1"] \ar[u,"d_2"] & \vZ^\bullet(W^{[1,2]}) \ar[u,"d_2"] \ar[r, "d_1"] & \vZ^\bullet(W^{[2,2]}) \ar[u, "d_2"] \ar[r, "d_1"] & 0
	\\ \vC^\bullet(Y_2) \ar[r,"d_1"] \ar[u,"d_2"] & \vC^\bullet(W^{[1,1]}) \ar[u, "d_2"] \ar[r, "d_1"] & \vZ^\bullet(W^{[2,1]}) \ar[u, "d_2"] \ar[r, "d_1"] & 0
	\\ \vC^\bullet(X) \ar[r,"d_1"] \ar[u,"d_2"] & \vC^\bullet(Y_1) \ar[u, "d_2"] \ar[r, "d_1"] & \vZ^\bullet(Y_1^{[2]}) \ar[u, "d_2"] \ar[r, "d_1"] & 0
\end{tikzcd}
	\label{E:bigerbe_triple_complex_aug}
\end{equation}
with the leftmost column and bottom row considered as degree $-1$ for $d_1$ and $d_2$, respectively.

\begin{lem}
Fix $\ell \geq 1$ and an abelian group $A$.
Suppose $[\alpha] \in \vH^\ell(X; A)$ satisfies $\pi_i^\ast [\alpha] = 0 \in \vH^\ell(Y_i; A)$ for $i = 1,2$. 
Then there is a well-defined {\em transgression class} defined by
\begin{equation}
	\Tr [\alpha] = [d_1 \beta_2 - d_2 \beta_1] \in \vH^{\ell-1}(W; A) / \big(\vH^{\ell-1}(Y_1;A)\oplus \vH^{\ell-1}(Y_2;A)\big)
	\label{E:four_class_transgression}
\end{equation}
where $\beta_i \in \vC_{\blue{\ecU}}^{\ell-1}(Y_i; A)$ are any elements satisfying $\delta\beta_i = \pi_i^\ast \alpha \in \vC_{\blue{\ecU}}^\ell(Y_i; A)$
for a representative $\alpha \in \vC_{\blue{\ecU}}^\ell(X; A)$ \blue{for an appropriate admissible set of covers $\ecU_\bullet$.}
\label{L:four_class_transgression}
\end{lem}
\begin{rmk}
This transgression can be understood as the $W^{[1,1]}$ component of the $E_2$
page differential of the $(\delta, D_{12})$ spectral sequence of $(\vC_{\blue{\ecU}}^\bullet(W^{[\bullet,\bullet]}),
\delta, D_{12})$ applied to $[\alpha]$, where we have rolled up $d_1$ and $d_2$
into a total differential $D_{12} = d_1 \pm d_2$.

In fact, to observe the sign convention discussed in
Convention~\ref{Conv:multicomplex} we should properly define $\Tr[\alpha]$
as the class $[(-1)^{\ell+1} d_2 \beta_1 + (-1)^{\ell+1} d_1 \beta_2]$ where
$\delta \beta_1 = (-1)^\ell d_1 \alpha$ and $\delta \beta_2 = (-1)^{\ell+1} d_2 \alpha$, but then cancellation of the
two factors of $(-1)^{\ell+1}$ and exchanging $\beta_1$ with $-\beta_1$ makes this equivalent to the definition given above.
\end{rmk}
\begin{proof}
With $\alpha$, $\beta_1$, and $\beta_2$ as above \blue{for a fixed admissible set of covers}, it follows that $d_1 \beta_2 - d_2 \beta_1$ is a cocycle 
since
\[
	\delta (d_1 \beta_2 - d_2 \beta_1) = d_1 \pi_2^\ast \alpha - d_2 \pi_1^\ast \alpha 
	= d_1 d_2 \alpha - d_2 d_1 \alpha = 0.
\]
Another choice of representative $\alpha' = \alpha + \delta \gamma$ can be incorporated 
as a different choice $\beta'_i = \beta_i + d_i \gamma$ of the $\beta_i$; moreover if
$\beta'_i \in \vC_{\blue{\ecU}}^{\ell-1}(Y_i)$ are another choice of bounding chains for $\pi_i\alpha$,
then $\delta(\beta_i - \beta'_i) = 0$ and
\[
	(d_1 \beta_2 - d_2 \beta_1)
	- (d_1 \beta'_2 - d_2 \beta'_1)
	= d_1(\beta_2 - \beta'_2) + d_2(\beta'_1 - \beta_1)
\]
is in the image under $\begin{bmatrix} d_2 & d_1 \end{bmatrix}$ of $\vH_{\blue{\ecU}}^{\ell-1}(Y_1) \oplus \vH_{\blue{\ecU}}^{\ell-1}(Y_2)$.
\end{proof}

\begin{thm}
A locally split square $(W,Y_2,Y_1,X)$ supports a bigerbe with a
given class $[\alpha] \in H^4(X; \bbZ)$
if and only if
\begin{enumerate}
[{\normalfont (i)}]
\item $\pi_i^\ast [\alpha] = 0 \in H^4(Y_i, \bbZ)$ for $i = 1,2$ and
\item $\Tr [\alpha] = 0 \in H^3(W; \bbZ)/\big(H^3(Y_1; \bbZ)\oplus H^3(Y_2; \bbZ)\big)$.
\end{enumerate}
\label{T:bigerbe_representability}
\end{thm}

\begin{proof}
By naturality of the Bockstein isomorphism, it suffices to work one degree lower with $\bbC^\ast$ coefficients.
Thus suppose $\alpha \in \vC_{\blue{\ecU}}^3(X; \bbC^\ast)$ represents $[\alpha]$. 
Since by hypothesis $\Tr[\alpha]$ vanishes, there exist representatives $\beta_i \in \vC_{\blue{\ecU}}^2(Y_i; \bbC^\ast)$ such that $[d_1 \beta_2 - d_2 \beta_1] = 0 \in \vH_{\blue{\ecU}}^2(W; \bbC^\ast)$;
thus $d_1\beta_2 - d_2 \beta_1 = \delta \gamma$
for $\gamma \in \vC_{\blue{\ecU}}^1(W; \bbC^\ast)$.
Then we claim $d_1 d_2 \gamma = d_2 d_1 \gamma \in \vZ_{\blue{\ecU}}^1(W^{[2,2]}; \bbC^\ast)$ is a pure cocycle and that a bigerbe $L \to W^{[2,2]}$ with 
\[
	c(L) = - d_1 d_2\gamma
\]
satisfies $G(L) = [\alpha]$.
Indeed, it is obvious that $d_i(d_1 d_2 \gamma) = 0 $ for $i = 1,2$; moreover $\delta d_1 d_2 \gamma = d_1 d_2 \delta \gamma = d_1 d_2 (d_1 \beta_2 - d_2 \beta_1) = 0$
as well, so by Lemma~\ref{L:Cech_bigerbe_class}, $d_1d_2 \gamma = - c(L) \in \vZ_{\blue{\ecU}}^1(W^{[2,2]}; \bbC^\ast)$ for a bigerbe $L \to W^{[2,2]}$.

To see that $G(L) = \alpha$, we follow the proof of Proposition~\ref{P:triple_complex_trivial}, carefully observing the sign convention
\eqref{E:triple_complex_total_diff} and observe that
\begin{equation}
\begin{tikzcd}[row sep=small]
	{c(L)}
	\\ - d_1 \gamma \ar[cm bar-to, u,"d_2"] \ar[cm bar-to, r,"D_{1}= \delta-d_1"] & (d_2 d_1 \beta_1, 0)
	\\ & (-d_1 \beta_1,0) \ar[cm bar-to, u,"-d_2"] 
\end{tikzcd}
	\label{E:bigerbe_zig_one}
\end{equation}
is a zig-zag which identifies $- d_1 \beta_1 \in \vZ_{\blue{\ecU}}^2(Y_1^{[2]}; \bbC^\ast)$ as a pure cocycle representing the image of $c(L)$ in the $E_1$ page 
of the $(d_2, D_1 = \delta \pm d_1)$ spectral sequence of the triple complex \eqref{E:bigerbe_triple_complex} which collapses to 
the $D_1$ cohomology of $\vZ_{\blue{\ecU}}^\bullet(Y_1^{[\bullet]}; \bbC^\ast)$. 
Then, as in \eqref{E:gerbe_zig_zag},
\begin{equation}
\begin{tikzcd}[sep=small]
	-d_1 \beta_1
	\\ -\beta_1 \ar[cm bar-to,u,"d_1"] \ar[cm bar-to,r,"\delta"] & -d_1 \alpha
	\\ & \alpha \ar[cm bar-to,u,"-d_1"] 
\end{tikzcd}
	\label{E:bigerbe_zig_two}
\end{equation}
is a further zig-zag which identifies $\alpha \in \vC_{\blue{\ecU}}^3(X; \bbC^\ast)$ as the image of $c(L)$ in the $E_1$ page of the $(d_1,\delta)$ spectral sequence
of $\vZ_{\blue{\ecU}}^\bullet(Y_1^{[\bullet]}; \bbC^\ast)$ representing the class $G(L) = [\alpha]$.

Conversely, to show necessity of this condition, suppose that $L \to W^{[2,2]}$ is a
bigerbe.
As shown in Lemma~\ref{L:Cech_bigerbe_class} this generates a \v
Cech cocycle, $\lambda =c(L)\in \vZ_{\blue{\ecU}}^1(W^{[2,2]}; \bbC^\ast)$ \blue{with respect to some admissible set of covers}, with values
in $\bbC^*$ which is a pure cocycle in the triple complex:
\begin{equation}
\delta \lambda =d_1\lambda =d_2\lambda =0.
\label{rbm.37}\end{equation}

Using the exactness of the simplicial complexes we may pull this back
under the two homotopy contractions $(s_1^1)^\ast$ and $(s_2^2)^\ast$ giving 
\begin{equation}
\gamma  \in \vC_{\blue{\ecU}}^1(W; \bbC^\ast),\quad d_1d_2\gamma =d_2d_1\gamma =\lambda .
\label{rbm.38}\end{equation}
Consider the \v Cech differential $\delta \gamma \in \vC_{\blue{\ecU}}^2(W; \bbC^\ast).$
The images of this, 
\begin{equation*}
d_1\delta \gamma =\delta d_1\gamma\in\vC_{\blue{\ecU}}^2(W^{[2,1]},\bbC^\ast)
\quad \text{and}\quad 
d_2\delta \gamma \in \vC_{\blue{\ecU}}^2(W^{[1,2]};\bbC^\ast)
\label{rbm.41}\end{equation*}
are pure cocycles in the triple complex, since $\lambda$ is closed. 
Thus $d_2\delta \gamma$ descends to a uniquely defined \v Cech cocycle $\mu
_2\in\vC_{\blue{\ecU}}^2(Y_2^{[2]},\bbC^\ast)$ with $d_1\mu_2=d_2\delta \gamma$. 
Note that $\delta \mu_2 = 0$ by injectivity of $d_1$ at the bottom level.
Under
$(s^2_2)^\ast$
this in turn pulls back to $\beta
_2\in\vC_{\blue{\ecU}}^2(Y_2;\bbC^\ast)$ with $d_2\beta_2=\mu_2.$ Now $d_2(\delta\gamma
-d_1\beta_2)=0$ by construction, so there is a unique $\beta_1 \in \vC_{\blue{\ecU}}^2(Y_1; \bbC^\ast)$ such that
\begin{equation}
-d_2\beta_1=\delta\gamma -d_1\beta_2.
\label{rbm.40}\end{equation}
It follows that 
$\mu_1 = d_1 \beta_1$ satisfies $d_1 \mu_1 = 0$ and $\delta \mu_1 = 0$
(by injectivity of $d_2$ on $\vC_{\blue{\ecU}}^2(Y_1^{[2]}; \bbC^\ast)$ and the fact that $\delta d_2 \mu_1 = \delta d_2d_1 \beta_1 = -\delta^2 d_1 \gamma = 0$).

Thus $\delta \beta_2$ and $\delta\beta_1$ descend, from $Y_2$ and $Y_1$
respectively, to define 
cocycles in $\vC_{\blue{\ecU}}^3(X; \bbC^\ast)$; moreover these must be the same
cocycle $\alpha \in \vC_{\blue{\ecU}}^3(X; \bbC^\ast)$
by injectivity of $d_1$ and $d_2$ and the fact that $d_1 \delta \beta_2 = d_2 \delta \beta_1$, so
this represents the 4-class of the bigerbe. This shows that the 
difference $d_1\beta _2-d_2\beta _1$ is exact and the criterion therefore holds. 
\end{proof}

There is an analogue of Proposition~\ref{P:gerbe_triv} classifying trivializations of bundle gerbes.

\begin{prop}
The trivializations of a bundle bigerbe $(L, W, Y_2, Y_1, X)$ form a torsor for the group
\begin{equation}
\begin{aligned}
	&\img\set{d_2 :\vH_Z^2(Y_1^{[2]}; \bbZ) \to \vH_Z^2(W^{[2,1]}; \bbZ)}
	\\&\oplus 
	\img\set{d_1 :\vH_Z^2(Y_2^{[2]}; \bbZ) \to \vH_Z^2(W^{[1,2]}; \bbZ)}
	\label{E:bigerbe_triv}
\end{aligned}
\end{equation}
where $\vH_Z^\bullet$ denotes the $\delta$ cohomology of the associated space in the diagram \eqref{E:bigerbe_triple_complex_aug}.
\label{P:bigerbe_trivn}
\end{prop}
\noindent In specific cases, as for the Brylinski-McLaughlin bigerbe in \S\ref{S:bm-bigerbe}, this may be simplified further.

\begin{proof} If $(Q_1,Q_2)$ and $(Q'_1,Q'_2)$ are two trivializations of a
  bigerbe $L$, then $P_1 = Q'_1 \otimes Q_1^\inv \to W^{[1,2]}$ and $P_2 =
  Q'_2 \otimes Q_2^\inv \to W^{[2,1]}$ are line bundles represented by \v
  Cech cocycles $\alpha_i = c(P_i)$ satisfying $d_1 \alpha_i = d_2 \alpha_i
  = \delta \alpha_i = 0$ for $i = 1,2$.
Exactness of the rows and columns of \eqref{E:bigerbe_triple_complex} gives the existence
of $\beta_1 \in \vZ^1(Y_2^{[2]}; \bbC^\ast)$ and $\beta_2 \in \vZ^1(Y_1^{[2]}; \bbC^\ast)$
satisfying $d_2 \beta_1 = \alpha_1$ and $d_1 \beta_2 = \alpha_2$, $d_i \beta_i = 0$ and (by injectivity) $\delta \beta_i = 0$.
It is straightforward to see that $[\alpha_1] \in \img\set{d_1 : \vH_Z^1(Y_2^{[2]}; \bbC^\ast) \to \vH_Z^1(W^{[1,2]}; \bbC^\ast)}$
is well-defined independent of choices, and similarly for $[\alpha_2]$. 
Conversely, given elements in \eqref{E:bigerbe_triv} corresponding to line bundles on $W^{[1,2]}$ and $W^{[2,1]}$
coming from simplicial line bundles $Y_2^{[2]}$ and $Y_1^{[2]}$, respectively, a trivialization $(Q_1,Q_2)$ may be altered
to give a different trivialization of the same bigerbe.
\end{proof}

\section{Examples of bigerbes} \label{S:bigerbe_examples}
\subsection{Decomposable bigerbes} \label{S:decomp_bigerbes}

As for the decomposable bundle gerbes discussed in
\S\ref{S:decomp_bigerbes}, we consider the special classes of bigerbes
corresponding to decomposable classes in $H^4(X; \bbZ).$ These are either
of the form $\alpha_1 \cup \alpha_2$ with the $\alpha_i \in H^2(X; \bbZ)$
or of the form $\rho\cup \alpha$ with $\rho\in H^1(X;\bbZ)$ and $\alpha \in
H^3(X; \bbZ).$ Stuart Johnson in his PhD thesis,
\cite{johnson2003constructions}, makes related constructions in the setting
of 2-gerbes.

From Theorem~\ref{T:bigerbe_representability} it follows that if, for $i =
1,2$, $\alpha _i\in H^2(X;\bbZ)$ and $\pi_i : Y_i \to X$, are locally split
maps such that $\pi_i^\ast \alpha_i = 0 \in H^{2}(Y_i; \bbZ)$ then the cup
product $\alpha _1\cup\alpha _2$ is represented by a bigerbe over the
locally split square $(Y_1\times_X Y_2, Y_2, Y_1, X)$.
Indeed, in \v Cech theory if $\rho _i\in\vC^1(Y_i;\bbZ)$ are primitives for the
$\pi_i^\ast \alpha_i$ then $\rho _1\cup\alpha_2$ and $\alpha_1 \cup\rho_2$ are primitives for $\alpha _1\cup\alpha _2$ on $Y_1$ and $Y_2$, respectively, and pulled
back to $Y_1\times_XY_2$ their difference, $\alpha_1\cup\rho_2-\rho_1
\cup\alpha_2$ has primitive $\rho _1\cup\rho _2$.

If the spaces $Y_i$ are the total spaces of circle bundles
representing the 2-classes the bigerbe is given quite explicitly in terms
of the classifying line bundle, for decomposed 2-forms, over the torus.

\begin{lem}\label{L:bimult_fund_bundle} The fundamental line bundle on
  $\bbT^2$ (with Chern class generating $H^2(\bbT^2; \bbZ) = \bbZ$) has a
  `bimultiplicative' representative $S \to \bbT^2$, meaning there are
  natural isomorphisms between fibers $S_{x_1 + x_2, y} \cong
  S_{x_1, y}\otimes S_{x_2,y}$ and $S_{x, y_1 +
    y_2} \cong S_{x, y_1} \otimes S_{x,y_1}$ such
  that
  \begin{equation} \begin{tikzcd}
S_{x_1+x_2,y_1+y_2}
      \ar[r] \ar[d] & S_{x_1,y_1+y_2}\otimes
      S_{x_2,y_1+y_2} \ar[d]
      \\ S_{x_1+x_2,y_1}\otimes S_{x_1+x_2,y_2}
      \ar[r] & S_{x_1,y_1}\otimes S_{x_2,y_1}\otimes
      S_{x_1,y_2}\otimes S_{x_2,y_2}
\end{tikzcd}
    \label{E:fund_line_bundle_comm}
\end{equation}
commutes.
 \end{lem}

\begin{proof} Line bundles over $\bbT^2 \cong \bbR^2/ \bbZ^2$ are naturally
  identified with $\bbZ^2$ equivariant line bundles over the universal cover,
  $\bbR^2$.
We equip the trivial bundle $\bbR^2 \times \bbC$ with the $\bbZ^2$ action
covering translation via
\begin{equation*}
(n,m)\cdot (x,y,z) = (x + n, y + m, e^{2\pi imx}z),
	\quad (n,m) \in \bbZ^2,\ (x,y,z) \in \bbR^2\times \bbC
\label{rbm.80}\end{equation*}
and let $S \to \bbT^2$ be the quotient bundle. 

The vertical and horizontal identifications in \eqref{E:fund_line_bundle_comm} correspond to the
invariance of the bilinear maps
\begin{equation*}
\begin{aligned}
(x,y_1,z_1)\times(x,y_2,z_2)&\longmapsto (x,y_1+y_2,z_1z_2),
\\(x_1,y,z_1)\times(x_2,y,z_2)&\longmapsto (x_1+x_2,y,z_1z_2)
\end{aligned}
\end{equation*}
which commute under the $\bbZ_2$ actions.
Moreover, the $\bbZ_2$-action corresponds to parallel transport along the
circles with
respect to the invariant connection $d+2\pi iy dx$,
the curvature of which is the fundamental class in $H^2(\bbT^2; \bbZ).$ 
\end{proof}

\begin{prop}\label{C:decomp_bigerbe}
For a decomposed 4-class $\alpha_1\cup\alpha _2\in H^4(X;\bbZ)$, with the
$\alpha _i\in H^2(X;\bbZ)$ represented by circle bundles
$Y_i\longrightarrow X$, the pullback under the product of the difference maps
$\chi_i:Y_i^{[2]} = Y_i\otimes Y_i\longrightarrow \UU(1)$, 
\[
L = (\chi_1\times \chi_2)^\ast S \to W^{[2,2]} = Y_1^{[2]}\times_X Y_2^{[2]},
\]
defines a bigerbe $(L,Y_1\times_X Y_2,Y_2,Y_1,X)$ with characteristic class $G(L)
= \alpha_1\cup \alpha_2$.
\end{prop}

\begin{proof} The bimultiplicative relations of Lemma~\ref{L:bimult_fund_bundle} correspond under pullback by
  $\chi_1\times \chi_2$ to the bisimplicial conditions for $L$, and
  that $G(L) = \alpha_1 \cup \alpha_2$ is a consequence of
  Lemma~\ref{L:cohom_transgr_circle_bundle}.
\end{proof}

Similarly if $\rho \in H^1(X; \bbZ)$ and $\alpha \in H^3(X; \bbZ)$ the
representability condition is satisfied by the fiber product square given
by any locally split maps $\pi_i : Y_i \to X$, $i = 1,2$ such that
$\pi_1^\ast \rho = 0 \in H^{1}(Y_1; \bbZ)$ and $\pi_2^\ast \alpha = 0 \in
H^3(Y_2; \bbZ)$.

Taking the `logarithmic' covering $\wt X\to X$ corresponding to $\rho$,
meaning the pullback of the universal cover of $\UU(1)$ by a homotopy class
of maps $X \to \UU(1)$ representing $\rho$ and a bundle gerbe $(L, Y, X)$
with $\DD(L) = \alpha \in H^3(X; \bbZ)$, there is again a direct
construction of a bigerbe for the fiber product square.

\begin{prop}\label{T:geometric_bigerbe_3plus1}
 If $(L,Y,X)$ is a bundle gerbe with Dixmier-Douady class $\alpha \in H^3(X; \bbZ)$ and $\wt X \to X$ is the
  logarithmic cover corresponding to a class $[\rho]\in H^1(X;\bbZ)$
    represented by $\rho:X\longrightarrow \UU(1)$ then the line bundle
\[
	L^{\chi} \to \wt X^{[2]} \times_X Y^{[2]},
\]
where $\chi:\wt X\times_X\wt X\longrightarrow \bbZ$ is the fiber-shift map,
defines a bigerbe 
\[
	(L^\chi,\wt X \times_X Y,Y, \wt X, X) 
	\quad \text{with} \quad
	G(L^\chi)=\rho \cup\alpha \in H^4(X; \bbZ).
\]
\end{prop}

\begin{proof}
We view the covering space $\wt X \to X$ as a principal $\bbZ$ bundle, and then the shift map
\[
	\chi : \wt X^{[2]} \to \bbZ
\]
defines the collective bundle $L^\chi$ on $\wt X^{[2]}\times_X Y^{[2]}$ given by the tensor product $L^n$ over $\chi^{-1}(n)$. 

The bisimplicial space is
\[
	W^{[m,n]} = \wt X^{[m]} \times_X Y^{[n]}.
\]
and the line bundle $L^\chi$ is simplicial in the $d_2$ direction, with
trivializing section of $d_2 (L^\chi) = (d_2L)^\chi$ over $W^{[2,3]}$ given
by $s^\chi$, and the $d_1$ differential of $L^\chi$ is given by
\[
d_1
(L^\chi) = L^{d_1\chi} = L^0
\]
so is canonically trivial. Thus this is indeed a bigerbe.

To see that $G(L^\chi) = \rho \cup \alpha$, observe that representative
cocycles $c(L) \in \vC^2(Y^{[2]}; \bbZ)$ and $\chi \in \vC^0(\wt X^{[2]};
\bbZ)$ pull back to $\vC^\bullet(\wt X^{[2]}\times_X Y^{[2]}; \bbZ)$ by the
fiber product projections, and their cup product $\chi \cup c(L) \in
\vC^2(\wt X^{[2]}\times_X Y^{[2]}; \bbZ)$ represents the transgression
image of $\rho \cup \alpha$ from $X$.
Then $\chi \cup c(L) = nc(L) = [L^n]$ locally on components $\chi^\inv(n)$, so the result follows.
\end{proof}

\subsection{Doubling for bigerbes} \label{S:prod_simp_bigerbe}
As in \S\ref{S:prod_simp}, we may incorporate an additional simplicial
structure with respect to the space of products $EX_\bullet =
X^\bullet$ in order to promote examples of bigerbes involving based loop
spaces to examples involving free loop spaces.

\begin{defn}
A bigerbe $L$ on $X^2$ will be said to be {\em double} if the bigerbe
\[
	\pa L = \pi_0^\ast L \otimes \pi_1^\ast L^\inv \otimes \pi_2^\ast L
\]
is trivial on $X^3$ with respect to the three projection maps $\pi_i : X^3 \to X^2$.
In the absence of additional data, $\pa L$ is defined with respect to the
bisimplicial space over $X^3$ obtained by the fiber products of the three pullbacks
of the bisimplicial space $W_2^{[\bullet,\bullet]}$ over
$X^2$.
However, as for gerbes above, it will typically be the case that $X^3$ carries
a natural split square and induced bisimplicial space $W_3^{[\bullet,\bullet]}$ 
along with maps%
$\begin{tikzcd}[inner sep=0,sep=small] W_3^{[\bullet,\bullet]} \arrthree &W_2^{[\bullet,\bullet]}\end{tikzcd}$ 
over the projections $\begin{tikzcd}[inner sep=0,sep=small] X^3 \arrthree & X^2\end{tikzcd}$.
In this case, it suffices that $\pa L \to W_3^{[2,2]}$, defined by pulling back
along the three maps $W_3^{[2,2]} \to W_2^{[2,2]}$ and taking the alternating
product, admits a bigerbe trivialization.
\label{D:product_simplicial_bigerbe}
\end{defn}

As a special case relevant in our primary example, a bigerbe is a double if
$\pa L$ is itself trivial as a line bundle over $W_3^{[2,2]}$.
Naturality of the bigerbe characteristic class and exactness of the sequence
\eqref{E:product_cohomology_seq} together lead to the following analogue
of Proposition~\ref{P:prod_simp_gerbe}.
\begin{prop}
The characteristic 4-class, $G(L),$ of a double bigerbe $L$ on $X^2$,  descends
from $H^4(X^2; \bbZ)$ to $H^4(X; \bbZ)$.
\label{P:prod_simp_bigerbe}
\end{prop}

In \S\ref{S:Path-bigerbes} we will also consider a similar condition with
respect to the bisimplicial space of products $X^{m,n} = X^{mn}$, with the
two sets of projections $\pi^1_j : X^{m,n} \to X^{m-1,n}$ and $\pi^2_j :
X^{m,n} \to X^{m,n-1}$.
A bigerbe $L$ over $X^4 = X^{2,2}$ is {\em quadruple} if $\pa_1 L$
and $\pa_2 L$ are respectively trivial over $X^{1,2} = X^2$ and $X^{2,1} =
X^2$.
The natural differentials $\pa_1$ and $\pa_2$, defined on cohomology
$\vH^\ell(X^{\bullet,\bullet}, A)$, commute, and from exactness of these we
obtain the following result.

\begin{prop}
For a quadruple bigerbe $L$ on $X^4$, the characteristic 4-class $G(L)$ descends
from $H^4(X^4; \bbZ)$ to $H^4(X; \bbZ)$.
\label{P:prod_bisimp_bigerbe}
\end{prop}

\subsection{Brylinski-McLaughlin bigerbes} \label{S:bm-bigerbe}

Next we turn to our main application. The loop space of a principal $G$
bundle over a manifold is a principal bundle over the loop space with
structure group the loop group of $G.$ The Brylinski-McLaughlin bigerbe
captures the obstruction to lifting this bundle to a (loop-fusion) principal
bundle for a central extension of the loop group.
While the version involving based path and loop spaces is simpler, we
focus from the beginning on the doubled version involving free path and
loop spaces, as this gives the results of primary interest.
Note that this theory most naturally involves $\UU(1)$ principal bundles in place of line bundles, 
which we shall use for the remainder of the section without further comment.

Let $G$ be a compact, simple, connected and simply connected group.
As is well-known (see for instance \cite{Pressley-Segal1}), there is a classification of $\UU(1)$ central extensions
\begin{equation}
	1 \to \UU(1) \to \wh{LG} \to LG \to 1
	\label{E:LG_ce}
\end{equation}
of the loop group $LG$ by $H^3(G; \bbZ) \cong H_G^3(G; \bbZ) \cong \bbZ$.
These extensions descend to the quotient $LG/G \cong \Omega G$ and so the
classification of central extensions of the based loop group $\Omega G$ is
equivalent.

Forgetting the group structure for the moment, such a central extension may be
viewed as a circle bundle over $LG \cong I^{[2]} G$, the Chern class
$c(\wh{LG}) \in H^2(LG; \bbZ)$ of which is the transgression of the defining class in
$H^3(G; \bbZ)$.
As such, it follows from Theorem~\ref{T:lf_trans} that $c(\wh{LG})$ has a loop-fusion refinement. 
In the equivalent language of the loop-fusion structures of Definition~\ref{D:loop-fusion_bundle}, we may restate this as follows.
\begin{thm}
As a $\UU(1)$-bundle, $\wh{LG}\to LG$ has a canonical loop-fusion structure, meaning a trivialization of $d\wh{LG} \to I^{[3]} G$ inducing
the canonical trivialization of $d^2 \wh{LG} \to I^{[4]} G$ and a trivialization of $\pa \wh{LG} \to L_8 G$.
\label{T:LG_lf}
\end{thm}
\begin{rmk}
In fact, the additional structure that promotes a general $\UU(1)$-principal bundle over $LG$ to a central extension is also a simplicial one.
Indeed, as noted by Brylinski and McLaughlin in \cite{B-MI} and attributed to Grothendieck, a $\UU(1)$ central extension
of any group $H$ is equivalent to a simplicial circle bundle with respect to the simplicial space $BH_\bullet$ defined by 
$BH_k = H^{k-1}$ 
with the face maps $H^{k+1} \to H^k$ given by
\[
	\pi_i : (h_0,h_1,\ldots,h_k) \mapsto \begin{cases} (h_1,\ldots,h_k), & i = 0 \\ (h_0,\ldots, h_{i-1}h_i,h_{i+1},\ldots,h_k), & 1 \leq i \leq k \\ (h_0,\ldots,h_{k-1}), & i = k. \end{cases}
\]
Thus, given a circle bundle $Q \to H = BH_2$, a trivialization of $\pa Q =
\pi_0^\ast Q\otimes \pi_1^\ast Q^\inv \otimes \pi_2^\ast Q$ inducing the canonical
trivialization of $d^2 Q \to H^3$ equips $Q$ with the (associative) multiplicative structure
of a $\UU(1)$ central extension of $H$ and vice versa.

For the groups under consideration, we believe it can be shown that the classes
in $H^3_G(G; \bbZ) = H^4(|BG|; \bbZ)$ are represented by cohomology classes
$\alpha \in H^3(G = BG_2; \bbZ)$ satisfying $\pa \alpha = 0 \in H^3(G^2 = BG_3;
\bbZ)$, and that the corresponding gerbe $(W, PG, G)$, with circle bundle $W
\to \Omega G$,  admits a simplicial structure with respect to $B\Omega
G_\bullet$, and thus a central extension of $\Omega G$. 
Further considering a doubled structure with respect to $G_\bullet =
G^\bullet$ gives rise to the central extensions of $LG$.
For reasons of space, and since the theory of central extensions of $LG$ is
already well-known, we will not elaborate further on this point.
\end{rmk}

To define the Brylinski-McLaughlin bigerbe, let $X$ be a connected manifold with 
principal $G$-bundle $E \to X$.
One case of particular interest is the spin frame bundle over a spin manifold of dimension $\geq 5$.

\begin{lem}\label{L:BM_lssquare} With vertical maps projections and
evaluation at end- and  mid-points in the horizontal directions,
the diagrams
%
\begin{equation}\label{E:BM_square}
\begin{tikzcd}
 	E^k \ar[d] & I_{k} E \ar[l,"\pi"] \ar[d] 
	\\ X^k & I_k X \ar[l, "\pi"]
\end{tikzcd}
\end{equation}
for $k \geq 2$ are locally split squares.
\end{lem}

\begin{proof} 
The maps are locally trivial fiber bundles, so all maps are locally split, and
$I_kE$ is likewise a fiber bundle over the fiber product $I_kX \times_{X^k}
E^k$, which is the space of paths in $X$ along with prescribed points in $E$
over the endpoints (for $k= 2$) and midpoint (for $k = 3$) of the path.
A connection on $E$ gives a horizontal lift of each path segment in $X$ given an 
initial point in $E$, and from the connectedness of $G$ this can be concatenated with
a path in the fiber from the endpoint of the lifted path segment to any other prescribed point
in the same fiber; 
this can be done for each segment defined between the $k$ marked points of the path on $X$.
This construction can be carried out locally continuously, so giving a local section of $I_kE$
over $E^k \times_{X^k}I_k X$.
\end{proof}

In the resulting bisimplicial diagrams we may write $IE^{[2]}$, etc., without
risk of confusion in light of the canonical isomorphisms
\[
\begin{gathered}
	 (I E)^{[2]} = I E \times_{I X} I E \cong I (E^{[2]}) = I (E\times_X E),
	\\ (L E)^{[2]} = L E \times_{L X} L E \cong L (E^{[2]}) = L (E\times_X E),
	\quad \text{etc.}
\end{gathered}
\]

Filling out the bisimplicial space for $k = 2$ by fiber products leads to the diagram
\begin{equation}
\begin{tikzcd}[sep=small]
{} \ar[d,no head,thick,dotted] & {} \ar[d,no head,thick,dotted] & {} \ar[d,no head,thick,dotted] & {} \ar[d,no head,thick,dotted]
\\
(E^{[3]})^2 \ardthree & IE^{[3]} \ar[l] \ardthree & LE^{[3]} \arltwo\ardthree&I^{[3]}E^{[3]}\ardthree\arlthree & {} \ar[l,no head, thick,dotted]
\\
(E^{[2]})^2 \ardtwo & I E^{[2]} \ar[l] \ardtwo & L E^{[2]} \arltwo \ardtwo
&I^{[3]} E^{[2]}\ardtwo\arlthree&{} \ar[l,no head, thick,dotted]
\\
E^2 \ar[d] & IE \ar[d] \ar[l] & LE \ar[d] \arltwo
&I^{[3]}E\arlthree\ar[d]&{} \ar[l,no head,thick,dotted]
	\\
X^2 & IX \ar[l] & LX \arltwo &I^{[3]} X\arlthree&{} \ar[l,no head,thick,dotted]
\end{tikzcd}
	\label{E:BM_diagram}
\end{equation}

The third column of \eqref{E:BM_diagram} is the simplicial space generated by the
fibration $L E \to L X$, itself a principal bundle with structure group $L G$,
and thus supports a lifting bundle gerbe
\[
	Q = \chi^\ast \wh{LG} \to L E^{[2]},
\]
where
\[
	\chi : LE^{[2]} \ni 
	\big(l_1(\theta),l_2(\theta)\big) \mapsto \ell(\theta) \in LG,
	\quad l_2(\theta) = \ell(\theta)l_1(\theta)
\]
is the shift map of the principal bundle, and we consider $\wh{LG} \to LG$ as a $\UU(1)$-bundle.
The other columns are likewise the simplicial spaces of principal bundles, with structure groups $G$ and $I^{[k]} G$ for $k \geq 1$,
and we denote their associated shift maps by the same letter.

\begin{thm}
Given a central extension \eqref{E:LG_ce} of level $\ell \in \bbZ = H_G^3(G; \bbZ)$,
the lifting bundle gerbe $Q \to LE^{[2]}$ is the double bigerbe $(Q, IE, E^2, IX, X^2)$ 
with characteristic class 
\begin{equation}
	G(Q) = \ell\,p_1(E) \in H^4(X; \bbZ),
	\label{E:BM_char_class}
\end{equation}
where $p_1(E)$ is the first Pontryagin class of $E$.
\label{T:BM_bigerbe}
\end{thm}
\noindent The bigerbe $(Q,IE, E^2, IX, X^2)$ will be called the {\em Brylinski-McLaughlin} bigerbe.

\begin{proof} It follows immediately from the lifting gerbe construction
  that $Q$ is vertically simplicial, and the simplicial condition in the
  horizontal direction follows from naturality of the shift map and
  Theorem~\ref{T:LG_lf}.

Indeed, unwinding the definitions reveals that $d_1 \chi^\ast \wh{LG} =
\chi^\ast d \wh{LG}$, which admits the trivialization noted in
Theorem~\ref{T:LG_lf} inducing the canonical trivialization of $\chi^\ast
d^2 \wh{LG} = d_1^2 \chi^\ast \wh{LG}$.

For doubling, we define $\pa Q$ with respect to the locally split square \eqref{E:BM_square} for $k = 3$,
the induced bisimplicial space of which sits in the diagram
\begin{equation}
\begin{tikzcd}[sep=small]
%
{} \ar[d,no head,thick,dotted] & {} \ar[d,no head,thick,dotted] & {} \ar[d,no head,thick,dotted] & {} \ar[d,no head,thick,dotted]
\\
(E^{[3]})^3 \ardthree & IE^{[3]} \ar[l] \ardthree & L_8 E^{[3]}
        \arltwo\ardthree&I^{[3]}E^{[3]}\ardthree\arlthree & {} \ar[l,no head, thick,dotted]
\\
(E^{[2]})^3 \ardtwo & IE^{[2]} \ar[l] \ardtwo & L_8 E^{[2]} \arltwo \ardtwo
&I^{[3]} E^{[2]}\ardtwo\arlthree&{} \ar[l,no head, thick,dotted]
\\
E^3 \ar[d] & IE \ar[d] \ar[l] & L_8 E \ar[d] \arltwo
&I^{[3]} E\arlthree\ar[d]&{} \ar[l,no head,thick,dotted]
	\\
X^3 & IX \ar[l] & L_8 X \arltwo &I^{[3]} X\arlthree&{} \ar[l,no head,thick,dotted]
\end{tikzcd}
	\label{E:BM_diagram_8}
\end{equation}
with the obvious maps to \eqref{E:BM_diagram}, and 
where we have omitted the subscript $3$ on the path spaces and made the identification $L_8 X \cong I^{[2]}_{3} X$, etc.
Once again, it follows that $\pa Q = \pa \chi^\ast \wh{LG} = \chi^\ast \pa \wh{LG}$ with
respect to the shift map for the principal $L_8 G$-bundle $L_8 E \to L_8 X$,
which admits a trivialization in light of Theorem~\ref{T:LG_lf}.
Thus $Q$ is product simplicial in the sense of Definition~\ref{D:product_simplicial_bigerbe} and its characteristic class
descends to $H^4(X; \bbZ)$.

Observe that this characteristic class can be obtained in two steps, first by regressing the Chern class $c(Q)$ from
$H^2_Z(LE^{[2]}; \bbZ)$ to $H^3_Z(LX; \bbZ)$ and then to $H^4(X; \bbZ)$.
The image of $c(Q)$ in $H^3_Z(LX; \bbZ)$ is, essentially by definition,
the Dixmier-Douady class of the lifting bundle gerbe (or more precisely,
its loop-fusion refinement), and it is well-known that this class is the
transgression of the Pontryagin class of $E$ on $X$ (multiplied by $\ell$
in the case of higher levels), so from Theorem~\ref{T:lf_trans} we obtain
\eqref{E:BM_char_class}.

\end{proof}

Without the vertical simplicial condition, the gerbe $(Q, LE^{[2]}, LX)$
represents the obstruction to lifting the $LG$ bundle $LE \to LX$ to an $\wh
{LG}$ bundle $\wh{LE} \to LX$.
The enhancement of this data to a bigerbe carries additional information,
which is formalized in the following definition.
\begin{defn}
Let $E \to X$ be a principal $G$-bundle for $G$ a simple, connected and simply connected Lie group, and fix a central extension \eqref{E:LG_ce} of $LG$
of level $\ell \in \bbZ = H_G^3(G; \bbZ)$.
A {\em loop-fusion $\wh{LG}$ lift} of $LE \to LX$ is a principal $\wh{LG}$-bundle $\wh{LE} \to LX$ lifting $LE$ with the property that $\wh{LE} \to LE$ is loop-fusion as a $\UU(1)$-bundle;
in other words, $d\wh{LE} \to I^{[3]} E$ has a trivialization inducing the canonical trivialization of $d^2\wh{LE} \to I^{[4]} E$ and $\pa \wh{LE} \to L_8 E$ admits a trivialization.
\label{D:lf_lift}
\end{defn}
\noindent Without the additional figure-of-eight structure, such fusion lifts have been considered
by Waldorf in \cite{MR3493404}, and with stronger conditions (high regularity and equivariance with respect
to diffeomorphisms of $S^1$) by the authors in \cite{KM-equivalence}.

\begin{thm} Loop fusion $\wh{LG}$ lifts of $LE \to LX$ are in bijection
  with doubled trivializations of the Brylinski-McLaughlin bigerbe; they
  exist if and only if $p_1(E)$ vanishes, and then form a torsor for
  $H^3(X; \bbZ)$.  \label{T:lf_lifts_trivns}
\end{thm}

\begin{proof}
Here by a doubled trivialization we mean a trivialization
$(P_1,P_2)$ in the sense of Definition~\ref{D:bigerbe_trivn} with the
additional property that $\pa P_i$ is trivial on the bisimplicial space
\eqref{E:BM_diagram_8}. 
As argued in the proof of Lemma~\ref{L:prod_simp_LX}, the retractions for any
space $Y$ of $I_{k} Y$ onto $Y$ itself for each $k$ give a homotopy with
respect to which $\pa = \Id$ as an operator from line bundles on $I_{2} Y$ to
those $I_{3} Y$; in particular the condition that $\pa P_1$ is trivial for
$P_1 \to I_{3} E^{[2]}$ means that $P_1$ itself is trivial.
Thus doubled trivializations for the bigerbe in question are reduced
to loop-fusion line bundles $P \to L E$ satisfying $d_2 P \cong Q$. 
On the one hand, these are clearly equivalent to loop-fusion $\wh{LG}$ lifts of
$LE \to LX$, and on the other, they are classified by those classes in
$\vH_{\lf}^2(LE; \bbZ)$ with image $c(Q)$ under $d_2$.

By Lemma~\ref{L:pa_contraction} below, the difference of any two such classes
descends to a class in $\vH^2_{\lf}(LX; \bbZ)$, and so doubled
trivializations form a torsor for the image of $\vH_{\lf}^2(LX; \bbZ)$ in
$\vH_{\lf}^2(LE; \bbZ)$, which by Theorem~\ref{T:lf_trans} is equivalent to the
image of $H^3(X; \bbZ)$ in $H^3(E; \bbZ)$.
Finally, given the conditions on $G$, it follows by the Serre spectral sequence
for $E \to X$ that $H^3(X; \bbZ) \to H^3(E; \bbZ)$ is an isomorphism in this
case, so the trivializations are classified simply by $H^3(X; \bbZ)$.
\end{proof}

It remains to show that the exactness of $d_2$ in the \v Cech-simplicial double
complex is consistent with $\pa$, which is a consequence of the following.
\begin{lem}
The homotopy chain contraction for $d_2$ in the triple complex
$(\vZ^\bullet(W^{[\bullet,\bullet]}), \delta, d_1, d_2)$ for the locally split
squares $(I_{k}E, E^k, I_{k}X, X^k)$ commutes with the product simplicial
operator $\pa$.
\label{L:pa_contraction}
\end{lem}
\begin{proof}
This follows ultimately from the existence of local sections of $E^k \to X^k$
(resp.\ $IE \to IX$) which are \blue{intertwined by} the three projection
maps $E^3 \to E^2$ and $X^3 \to X^2$ (resp.\ $I_{3} E \to I_{2} E$ and
$I_{3} X \to I_{2} X$), which we proceed to demonstrate.
In the first case, we may fix an admissible pair of covers $(\cV, \cU)$ for
$(E,X)$ and then equip $E^k$ and $X^k$ with the covers $\cV^k$ and $\cU^k$
along with the induced sections $\Et{\cU^k} \to \Et{\cV^k}$, which are then
automatically \blue{intertwined} by the projection maps. 

For the path spaces, we begin with the fact that $(IE, IX, E^3, X^3)$ is a
locally split square, so $I_{3} E$ and $I_{3} X$ admit covers $\cW$ and
$\cY$ and sections $\Et{\cY}\to \Et{\cW}$ lying over the sections $\Et{\cU^3}
\to \Et{\cV^3}$.
In general, the reparameterization maps 
$\begin{tikzcd}[sep=small] I_{3} Y \arrthree& I_{2} Y \end{tikzcd}$ 
are open, so we may equip $I_{2} E$ and $I_{2} X$ with the 
union of the three image covers $\cW' = \wt \pi_0(\cW) \cup \wt \pi_1(\cW) \cup
\wt \pi_2(\cW)$ and $\cY' = \wt \pi_0(\cY) \cup \wt \pi_1(\cY) \cup \wt
\pi_2(\cY)$, along with the induced section
\[
	\wt s' : \Et{\cY'} \to \Et{\cW'},
\]
giving a set of local sections of $I_{2} E \to I_{2} X$ which is compatible
with the three reparameterization maps $\wt \pi_i$, and which covers the local
sections of $E^2 \to X^2$.
\end{proof}

There is a simpler version of this bigerbe using based path and loop spaces, starting with the locally split square $(PE, E, PX, X)$, pulling
back a central extension $\wh{\Omega G} \to \Omega G$ to $\Omega E^{[2]}$, and omitting
the doubling conditions.
We leave the details as an exercise to the reader.

\subsection{Loop spin structures} \label{S:aside}
There is a well-known relationship between string
structures on a spin manifold $X$ of dimension $2n > 4$, and (loop) spin structures
on its loop space $LX$. 

Here, a {\em string structure} is a lift of the principal $\Spin(2n)$
bundle $E \to X$ to a principal bundle with structure group $\String(2n)$,
a 3-connected topological group covering $\Spin(2n)$ in the sequence of
ever more connected groups that form the Whitehead tower for
$O(2n),$ see for instance \cite{stolz_conjecture}.
The string group cannot be a finite dimensional Lie group (having a
subgroup with the topology of $K(\bbZ,2)$), though there are various
realizations as a 2-group \cite{Baez,string2}.
The obstruction to lifting the structure group is $\tfrac 1 2 p_1(X) \in H^4(X;
\bbZ)$ (the Pontryagin class of the $\Spin$ bundle being a refinement of the
Pontryagin class of the oriented frame bundle), and if unobstructed, string
structures are classified by $H^3(X; \bbZ)$ \cite{stolz_teichner_elliptic}.

As originally defined by Killingback in \cite{Killingback} and further developed by McLaughlin in \cite{McLaughlin}, a spin structure on $LX$
is a lift of the $L\Spin$ bundle $LE \to LX$ to the structure group $\wh{L\Spin}$, 
the fundamental $\UU(1)$ central extension of $L\Spin$.
(By analogy, as originally suggested by Atiyah in \cite{atiyah_loop_orientation}, an {\em orientation} on $LX$ is a refinement of the $L\SO(2n)$ 
bundle $LE_{\SO} \to LX$ to have structure group the connected component of the identity, 
$L_+\SO(2n) \cong L\Spin(2n)$, and therefore is typically related to a spin structure on $X$.)
The obstruction to this lift is the 3-class on $LX$ obtained by transgression of $\tfrac 1 2 p_1(X) \in H^4(X; \bbZ)$.

As defined, string structures on $X$ and spin structures on $LX$ are not
necessarily in bijection \cite{pilch}.
In fact, it was Stolz and Teichner in \cite{Stolz-Teichner2005} who first noted
the importance of the fusion structure on $LX$ and showed that string
structures on $X$ were in correspondence with what they called `stringor
bundles' on (the piecewise smooth loop space) $LX$, essentially bundles
associated to a lift $\wh{LE}$ along with a fusion condition. 
It was further proved by the authors in \cite{KM-equivalence} that string
structures in the sense of Redden \cite{Redden2011} correspond with spin
structures on the smooth loop space $LX$ which are both fusion and equivariant
for the group $\Diff^+(S^1)$ of oriented diffeomorphisms of the loop parameter,
and it was independently proved by Waldorf in \cite{MR3493404} that string
structures on $X$ exist if and only if fusion spin structures on (piecewise
smooth) $LX$ exist, using a transgression theory relating the 2-gerbe
obstructing string structures of \cite{CJMSW} and fusion gerbes on loop space.
Waldorf did not obtain a complete correspondence between string structures
and fusion loop spin structures, noting that this would necessitate additional
conditions such as equivariance with respect to thin homotopy; in the version considered here, 
it is the figure-of-eight (i.e., doubling) condition that provides the remedy.

In any case, the bigerbe formulation here leads to the following result.

\begin{cor}
There are natural bijections between the following sets:
\begin{enumerate}
[{\normalfont (i)}]
\item the set of string structures on a spin manifold $X$ of dimension $2n > 4$,
\item the set of loop-fusion spin structures on $LX$, meaning lifts of $LE \to LX$
to the structure group $\wh{L\Spin}$ such that the resultant $\UU(1)$ bundle $\wh{LE} \to LE$
is a loop-fusion bundle according to Definition~\ref{D:loop-fusion_bundle}, and
\item the set of doubled trivializations of the Brylinski-McLaughlin bigerbe $(Q, IE, E^2, IX, X^2)$.
\end{enumerate}
The sets are empty unless $\tfrac 1 2  p_1(E) = 0 \in H^4(X; \bbZ)$ and otherwise are torsors for $H^3(X; \bbZ)$.
\label{C:string_loop_spin}
\end{cor}

%
\subsection{Path bigerbes}\label{S:Path-bigerbes}

If $X$ is a path-connected and simply connected space with basepoint $b$,
from the based double path space
\begin{equation}
  Q X= PPX =\{u:[0,1]^2\longrightarrow X:u\big|_{\{0\}\times[0,1]}=u\big|_{[0,1]\times\{0\}}=b\}
\label{rbm.33}\end{equation}
there are two surjective restriction maps 
\begin{equation}
f_i:Q X\longrightarrow  PX,\
f_1u=u\big|_{[0,1]\times\{1\}}\text{ and }f_2u=u\big|_{\{1\}\times[0,1]}.
\label{rbm.34}\end{equation}
\begin{thm}\label{rbm.36} On a connected, simply connected and locally
  contractible space the end-point maps and restriction maps in
  \eqref{rbm.34} form a locally split square 
\begin{equation}
\begin{tikzcd}
	 PX \ar[d] & Q X \ar[l, "f_1"'] \ar[d, "f_2"]
	\\ X &  PX \ar[l]
\end{tikzcd}
\label{rbm.35}\end{equation}
and any class $\gamma\in H^4(X,\bbZ)$ arises from a bigerbe corresponding
to \eqref{rbm.35}.
\end{thm}

\begin{proof} The fiber product of the two copies of $PX$ is the based
loop space of $X$. The simple connectedness of $X$ implies the fiber
product of the two $f_i$ is surjective and from local contractibility it is
locally split. Since $PX$ and $QX$ are both contractible,
Theorem~\ref{T:bigerbe_representability} applies to any 4-class on $X$.
\end{proof}

Since the Eilenberg-Mac\ Lane spaces can be represented by CW complexes Theorem~\ref{rbm.36}
applies in particular to $K(\bbZ, 4)$.

\begin{thm}
There exists a universal bigerbe over $K(\bbZ, 4)$ with respect to the locally split square \eqref{rbm.35} with $X = K(\bbZ, 4)$.
\label{T:univ_bigerbe}
\end{thm}
Note the
structure of the bisimplicial space in this case, in which $*$ represents
a contractible space:
\begin{equation*}
\begin{tikzcd}[sep=small]
	{}  & {}  & {} 
	\\ K(\bbZ, 3) \ar[u,thick,dotted,no head]\ardtwo & \ar[l] \ardtwo \ar[u,thick,dotted,no head]\ast & \arltwo \ardtwo K(\bbZ, 2)\ar[u,thick,dotted,no head] \ar[r,thick,dotted,no head]& {} 
	\\ \ast \ar[d] & \ar[l] \ar[d] \ast & \arltwo \ar[d] \ast \ar[r,thick,dotted,no head] & {}
	\\ K(\bbZ,4)  & \ar[l] \ast & \arltwo K(\bbZ,3) \ar[r, thick,dotted,no head] & {}
\end{tikzcd}
\label{rbm.62}\end{equation*}

Finally, incorporation of a {\em product-bisimplicial} condition allows any
4-class to be represented as a bigerbe on any connected, locally contractible space $X$, whether simply connected
or not. 
Indeed, consider the locally split square $(IIX, IX^2, IX^2, X^4)$,
where $IIX = \set{u : [0,1]^2 \to X}$ is the free double path space and the projection maps 
are given by evaluation at both endpoints of a given path factor.
The induced bisimplicial space becomes
\begin{equation}
\begin{tikzcd}[sep=small]
	{} & {} & {} & {}
	\\ LX^2 \ardtwo \ar[u,thick,dotted,no head] & ILX \ar[l] \ardtwo \ar[u,dotted,thick, no head] & LLX \arltwo \ardtwo \ar[u, thick, dotted, no head] \ar[r, thick,dotted, no head]& {}
	\\ IX^2 \ar[d] & IIX \ar[d] \ar[l] & LIX \ar[d] \arltwo \ar[r,thick,dotted,no head] & {}
	\\ X^4 & IX^2 \ar[l] & LX^2 \arltwo \ar[r,thick,dotted,no head] & {}
\end{tikzcd}
	\label{E:path_bisimp_22}
\end{equation}
and in particular $W^{[2,2]} = LLX$ is the double free loop space of $X$.
We may view $X^4$ at the bottom as the factor $X^{2,2}$ in the bisimplicial space $X^{m,n} = X^{mn}$ of products as discussed in \S\ref{S:prod_simp_bigerbe}.
Over $X^{3,2}$ consider the locally split square and induced bisimplicial space
\begin{equation}
\begin{tikzcd}[sep=small]
	{} & {} & {} & {}
	\\ LX^3 \ardtwo \ar[u,thick,dotted,no head] & I_3LX \ar[l] \ardtwo \ar[u,dotted,thick, no head] & L_8LX \arltwo \ardtwo \ar[u, thick, dotted, no head] \ar[r, thick,dotted, no head]& {}
	\\ I_2X^3 \ar[d] & I_3I_2X \ar[d] \ar[l] & L_8I_2X \ar[d] \arltwo \ar[r,thick,dotted,no head] & {}
	\\ X^{3,2} & I_{3}X^2 \ar[l] & L_8X^2 \arltwo \ar[r,thick,dotted,no head] & {}
\end{tikzcd}
	\label{E:path_bisimp_32}
\end{equation}
and likewise for $X^{2,3}$ with factors reversed.
The bisimplicial spaces \eqref{E:path_bisimp_32} map to \eqref{E:path_bisimp_22} over the product maps 
$\begin{tikzcd}[sep=small]X^{3,2} \arrthree &X^{2,2}\end{tikzcd}$ and $\begin{tikzcd}[sep=small]X^{2,3} \arrthree &X^{2,2}\end{tikzcd}$,
and there are associated operators $\pa_1$ and $\pa_2$ on line bundles.
\begin{thm}
For a connected, locally contractible space $X$, every class in $H^4(X; \bbZ)$ is represented by a product-bisimplicial bigerbe with respect to \eqref{E:path_bisimp_22},
that is, a bigerbe $(L,IIX, IX^2,IX^2,X^4)$ having in addition trivializations of the line bundles $\pa_1 L \to L_8 LX$ and $\pa_2 L \to L L_8 X$.
\label{T:free_loop_universal}
\end{thm}
\begin{proof}
Such a bigerbe has characteristic class $G(L) \in H^4(X^{2,2}; \bbZ)$ satisfying
$\pa_1 G(L) = 0 \in H^4(X^{3,2}; \bbZ)$ and $\pa_2 G(L) = 0 \in H^4(X^{2,3};
\bbZ)$, hence by Proposition~\ref{P:prod_bisimp_bigerbe} this descends to a well-defined class 
\[
	G(L) \in H^4(X; \bbZ).
\]

Conversely, given any $\alpha \in H^4(X; \bbZ)$, let $\beta = \pa_1
\pa_2 \alpha \in H^4(X^{2,2}; \bbZ)$.
This evidently satisfies $\pa_i \beta = 0$ for $i = 1,2$, and moreover, denoting by
$\Delta_i : X^2 \hookrightarrow X^{2,2}$ for $i = 1,2$ the diagonal inclusions, 
satisfies $\Delta_i^\ast \beta = 0 \in H^4(X^2; \bbZ)$. 
Under the deformation retractions $I X^2 \simeq X^2$, the evaluation maps $IX^2 \to X^{2,2}$
become identified with these diagonal inclusions, so it follows that $\beta$ lifts to vanish in $H^4(IX^2; \bbZ)$.
Then since $IIX \simeq \ast$ is contractible, Theorem~\ref{T:bigerbe_representability} applies 
and it follows that $\beta$ is represented by a bigerbe $(L, IIX, IX^2, IX^2, X^4)$, such that 
$\pa_i L$ is trivial (as a bigerbe) for $i = 1,2$.
As a consequence of Lemma~\ref{L:prod_simp_LX}, which applies to both horizontal and vertical directions 
in the diagram \eqref{E:path_bisimp_32}, bigerbe triviality of $\pa_i L$ is equivalent to triviality of $\pa_i L$ 
as a line bundle. 
\end{proof}

\section{Multigerbes} \label{S:multi}

We end by sketching out the theory of {\em multigerbes}, the higher degree
generalization of bigerbes.
By contrast to bundle gerbes, this generalization to higher degree is
straightforward, with symmetry of the simplicial conditions replacing the need
for higher and ever more complicated associativity conditions.

Fix a degree $n \in \bbN$, where $n = 1$ and $n = 2$ correspond to bundle gerbes
and bigerbes, respectively.
To establish notation, let $e_j = (0,\ldots, 1,\ldots, 0)$ denote the $j$th standard basis
vector, and for each integer $k$ let $\ul k = (k,\ldots, k)$ denote the vector with constant entries.
For a multiindex $\alpha \in \bbN^n$ we let $\abs \alpha = \alpha_1 + \cdots + \alpha_n \in \bbN$, and
\blue{we use the partial order on multiindices where $\beta<\alpha$ if $\beta_i <\alpha_i$ for each $1 \leq i \leq n$.}
We distinguish the sets of natural numbers starting at $1$ and $0$ respectively by $\bbN_0$ and $\bbN_1$. 
\begin{added}
\begin{defn}
Fix
a set of spaces $X_\alpha$ indexed by $\alpha = (\alpha_1,\ldots,\alpha_n) \in \set{0,1}^n$
along with continuous surjective maps
\begin{equation}
	X_\alpha \to X_{\alpha - e_j}
	\quad \text{whenever $\alpha_j = 1$}.
	\label{E:lscube_str_maps}
\end{equation}
For each $\alpha \in \bbN^n$ define
\[
	X_{<\alpha} = \lim_{\beta <\alpha} X_\beta
\]
to be the limit space of the diagram formed by all spaces to which $X_\alpha$ maps along with the maps between them.

Then the set $\set{X_\alpha : \alpha \in \set{0,1}^n}$ is a {\em locally split $n$-cube} provided each of the induced maps $X_\alpha \to X_{<\alpha}$ is locally split.
\label{D:lscube}
\end{defn}
\end{added}
In particular, a locally split $1$-cube is just a locally split map $X_1 \to X_0$ and a locally
split $2$-cube is a locally split square.
\blue{
It follows, in particular by the proof of Lemma~\ref{L:ad_set_cube_exists}, that all of the maps \eqref{E:lscube_str_maps} are locally split, and that each subcube
determined by freezing some of the coefficients of $\alpha$ is likewise locally split.}

\begin{lem}
A locally split $n$-cube extends naturally by taking fiber products to a set of
spaces $\set{X_\alpha : \alpha \in \bbN_0^n}$ such that $X_{\bullet \geq \ul 1} = \set{X^\alpha : \alpha
\in \bbN_1^n}$ is an $n$-fold multisimplicial space over $X := X_{\ul 0}$; in
particular for each fixed $\alpha = (\alpha_1,\ldots, 0, \ldots, \alpha_n)$
with vanishing $j$th coordinate the sequence
\[
\begin{tikzcd}
	X_\alpha & \ar[l] X_{\alpha + e_j} & \arltwo  X_{\alpha + 2e_j} & \arlthree X_{\alpha + 3e_j} & \ar[l,thick,dotted,no head]
\end{tikzcd}
\]
is the simplicial space of fiber products of the map $X_{\alpha +e_j} \to X_{\alpha}$.
\label{L:multisimplicial_space}
\end{lem}
\begin{proof}
The proof is by induction on $n$, the case $n = 2$ having been proved as Proposition~\ref{P:fiber_quadrant}.
Assuming therefore that the result holds for $n-1$, the `hypersurfaces' $\set{X_{\alpha} : \alpha_j \equiv 0}$ are well-defined
for $1 \leq j \leq n$, and for general $\alpha \in \bbN_1^n$ define $X_\alpha$ as a subspace of $X_{\ul 1}^{\alpha_1\cdots \alpha_n}$ as follows.
For each $j$, there are $\alpha_j$ projection maps $X_{\ul 1}^{\alpha_1\cdots \alpha_n} \to X_{\ul 1}^{\alpha_1 \cdots \wh{\alpha_j} \cdots \alpha_n}$, 
where the caret denotes omission, and these may be composed with the structure map 
\[
	X_{\ul 1}^{\alpha_1\cdots \wh{\alpha_j}\cdots \alpha_n} \to X_{\ul 1 - e_j}^{\alpha_1\cdots \wh{\alpha_j}\cdots \alpha_n}.
\]
Introducing the notation $\alpha(j)$ to mean the multiindex obtained from $\alpha$ by setting $\alpha_j = 0$,
we may view $X_{\alpha(j)}$ as a subspace of $X_{\ul 1 - e_j}^{\alpha_1\cdots \wh{\alpha_k}\cdots \alpha_n}$, and then $X_\alpha$ is well-defined as the subspace of $X_{\ul 1}^{\alpha_1\cdots\alpha_n}$ 
in the mutual preimage of $X_{\alpha(j)}$ under the $\alpha_j$ maps $X_{\ul 1}^{\alpha_1\cdots\alpha_n} \to X_{\ul 1 - e_j}^{\alpha_1\cdots\wh{\alpha_j}\cdots \alpha_n}$
for each $j$.
\end{proof}

We denote the face maps of the multisimplicial space by $\pi^j_\bullet$, for $1 \leq j \leq n$.
Then for each $\alpha \in \bbN_0^n$, there are $n$ simplicial differentials defined on a line bundle $L \to X_\alpha$ by
\[
	d_j L = \bigotimes_{k=0}^{\alpha_j} (\pi^j_{k})^\ast L^{(-1)^k} \to X_{\alpha + e_j},
	\quad 1 \leq j \leq n,
\]
with the property that $d_j^2 L \to X_{\alpha+2e_j}$ is canonically trivial.
\begin{defn}
A {\em bundle $n$-multigerbe}, or simply {\em multigerbe}, defined with respect to a locally split $n$-cube $X_\alpha$ is a multisimplicial line bundle $L \to X_{\ul 2}$, meaning $d_j L$ is given a trivializing section $s_j$
inducing the canonical trivialization of $d_j^2 L$ for each $j$, and the induced trivializations
of $d_j d_k L \cong d_k d_j L$ are consistent for each pair $j \neq k$.

A set 
$(P_1,\ldots,P_n)$ of 
line bundles $P_j \to X_{\ul 2 - e_j}$ such that each $P_j$ is multisimplicial with respect to $d_i$ for $i \neq j$
determines a multisimplicial line bundle $\bigotimes_{j=1}^n d_j P_j^{(-1)^j} \to X_{\ul 2}$ in the obvious way, and a {\em trivialization} of an $n$-multigerbe $L$
consists of a set $(P_1,\ldots,P_n)$ as above along with
an isomorphism
\[
	L \cong \bigotimes_{j=1}^n d_jP_j^{(-1)^j}
\]
intertwining the multisimplicial structures of both sides.
\label{D:multigerbe}
\end{defn}

Pullbacks, products and morphisms of multigerbes are defined by generalizing in
the obvious way those same operations for bigerbes, and making use of the following result,
which follows immediately from Lemma~\ref{L:pullback_prod_lss}.

\begin{lem}
The pullback by a continuous map $X' \to X_{\ul 0}$ of a locally split $n$-cube $X_\alpha$
is a locally split $n$-cube over $X'$.
Likewise, if $X_\alpha$ and $X'_\alpha$ are locally split $n$-cubes over the same
base $X := X_{\ul 0} = X'_{\ul 0}$, then the fiber products $X_\alpha\times_{X} X'_\alpha$
form a locally split $n$-cube.
\label{L:lsncube_operations}
\end{lem}

The characteristic class is defined as before in terms of the total cohomology of a 
\v Cech-simplicial multicomplex.

\begin{added}
\begin{defn}
An admissible set of covers for a locally split $n$-cube is a set of covers $\cU_{X_\alpha} \to X_\alpha$ such that each pair $(\cU_{X_{\alpha -e_j}},\cU_{X_\alpha})$ is admissible
and each set $(\cU_{X_{\alpha -e_i-ej}},\cU_{X_{\alpha-e_i}},\cU_{X_{\alpha - e_j}},\cU_{X_\alpha})$ is an admissible set of covers for the associated split square.
\label{D:admissible_set_cube}
\end{defn}

\begin{lem}
Given any covers for the spaces in a locally split $n$-cube, there exists an admissible set refining them.
\label{L:ad_set_cube_exists}
\end{lem}
\begin{proof}
The proof is by induction on $n$ of the the stronger statement that covers of \emph{adjacent} $n$-cubes rooted at $X = X_{\ul 0}$ can be refined to a set which is simultaneously admissible
for each cube. 
Here by adjacent we mean a finite set of $n$-cubes rooted at $X$ which may share subcubes.
The cases $n = 1$ and $n = 2$ are furnished by Lemmas~\ref{Addedrbm} and 
\ref{L:admissible_set_exists}.
%

For the inductive case, we consider, for each $n$-cube, $X_{\ul 1}$ at the top mapping to the limit space $X_{<\ul 1} := \lim_{\beta<\ul 1} X_{\beta}$, and we fix a cover $\cU^0_{X_{<\ul 1}}$ 
supporting sections, denoted $t$, of the split map $p : X_{\ul 1} \to X_{<\ul 1}$.
The topmost spaces in the diagram over which the limit is taken, namely $X_{\ul 1- e_j}$, are the topmost spaces of a set of adjacent $(n-1)$-cubes rooted at $X$, so by the inductive
hypothesis there are covers $\cU^0_{X_{\bullet}}$, $\bullet < \ul 1$ refining the given covers, which constitute simultaneous admissible sets for the $(n-1)$-cubes.
In addition, at this point the covers $\cU^0_{X_{\ul 1 - e_j}}$ of the penultimate spaces support sections, $\sigma_j$, of the natural maps $\pr_j : X_{<\ul 1} \to X_{\ul 1- e_j}$ from the limit space, which participate in commutative squares along with the maps of covers involving the $\cU^0_{X_{\ul 1 - e_i - e_j}}$.
Finally, we set $\cU^0_{X_{\ul 1}}$ to be the given cover of $X_{\ul 1}$ and this setup constitutes step 0.

\noindent
\textit{Step 1}: We define $\cU^1_{X_{\ul 1}}$ to be the given cover of $X_{\ul 1}$ and then subsequently set $\cU^1_{X_{<\ul 1}} = \cU^0_{X_{<\ul 1}} \cap t^\inv(\cU^1_{X_{\ul 1}})$,
and $\cU^1_{X_{\ul 1 - e_j}} = \cU^0_{X_{\ul 1 - e_j}} \cap \sigma_j^\inv(\cU^0_{X_{< \ul 1}})$; if there are multiple $n$-cubes in which $X_{\ul 1 - e_j}$ participates, then
we take $\cU^1_{X_{\ul 1 - e_j}}$ to be the mutual refinement by the pullback of the step 1 covers of the associated limit spaces for each such $n$-cube.

\noindent
\textit{Step 2}:
Appealing to the inductive step again gives refinements, denoted $\cU^2_\bullet$, of all the spaces up to the $X_{\ul 1 - e_j}$, and then we refine the cover of the limit space by 
\[
	\cU^2_{X_{<\ul 1}} = \cU^1_{X_{<\ul 1}} \cap \lim_{\beta < \ul 1} \cU^2_{X_{\beta}},
\]
the limit here taken over the maps of covers lifting the $n$-cube maps downstairs,
and set $\cU^2_{X_{\ul 1}} = \cU^1_{X_{\ul 1}} \cap p^\inv(\cU^1_{X_{<\ul 1}})$ for the space at the top of each $n$ cube.
That this furnishes an admissible set of covers for the $n$-cube follows along the lines of the proof of Lemma~\ref{L:admissible_set_exists}, and completes the induction.
\end{proof}
\end{added}

\begin{lem}
\blue{For a fixed admissible set of covers,}
the simplicial differentials 
\[
	d_j = \sum_{k = 0}^{\alpha_j} (-1)^k (\pi^j)_{k}^\ast : \vC_{\blue{\ecU}}^\ell(X_\alpha; A) 
	\to \vC_{\blue{\ecU}}^\ell(X_{\alpha + e_j}; A)
\]
on \v Cech cochains form, along with the \v Cech differential $\delta$, an $(n+1)$-multicomplex
$\big(\vC_{\blue{\ecU}}^\bullet(X_{\bullet}; A), \delta, d_1,\ldots, d_n\big)$ with the following properties:
\begin{enumerate}
[{\normalfont (i)}]
\item \label{I:multigerbe_cech_exact} For each $j$, fixing all other indices, the complex $\big(\vC_{\blue{\ecU}}^\ell(X_{\bullet_j}; A), d_j\big)$ is exact,
and for each $k \neq j$ admits a homotopy chain contraction commuting with $d_k$. 
\item \label{I:multigerbe_cech_total} The total cohomology of 
$\big(\vC_{\blue{\ecU}}^\bullet(X_{\bullet \geq \ul 1}; A), \delta, d_1,\ldots, d_n\big)$ is isomorphic to $\vH_{\blue{\ecU}}^\bullet(X; A)$, where $X = X_{\ul 0}$.
\item \label{I:multigerbe_cech_chern} The Chern class of the line bundle $L \to X_{\ul 2}$ of a multigerbe 
is represented by a cocycle $c(L) \in \vC_{\blue{\ecU}}^1(X_{\ul 2}; \bbC^\ast)$ \blue{for some admissible set of covers,} with $d_j c(L) = 0$ for each $j$, and such a multigerbe is trivial
if and only if $c(L)$ is a coboundary in the total $(\delta,d_1,\ldots,d_n)$ complex.
\end{enumerate}
\label{L:multigerbe_cech}
\end{lem}
\begin{proof}
Here \eqref{I:multigerbe_cech_exact} is a consequence of Proposition~\ref{P:triple_complex_exactness}, since for each pair $j \neq k$, the bisimplicial
space obtained from $X_\alpha$ by freezing all but the $j$th and $k$th indices is equivalent to the one obtained from a locally split square.

Part \eqref{I:multigerbe_cech_total} follows by induction, rolling up the $(n+1)$ multicomplex into the double complex $(d_n, D_{n-1})$ where $D_{n-1}$ denotes
the total differential associated to $(\delta,d_1,\ldots,d_{n-1})$.
By exactness of $d_n$, this total cohomology is isomorphic to the total $D_{n-1}$ cohomology of the complex $\vC^\bullet(X_{\bullet(n) \geq \ul 1(n)}; A)$
where again $\alpha(n)$ denotes the index obtained from $\alpha$ by setting $\alpha_n = 0$.

Finally, part \eqref{I:multigerbe_cech_chern} is proved by a straightforward generalization of the proof of Lemma~\ref{L:Cech_bigerbe_class}.
\end{proof}

\begin{defn}
The {\em characteristic class} of a multigerbe $(L, X_\alpha)$ is the class
\[
	G(L) \in H^{n+3}(X; \bbZ),
	\quad X = X_{\ul 0},
\]
given by the Bockstein image of $[c(L)] \in
H^{n+2}\big(\vC_{\blue{\ecU}}^\bullet(X_{\bullet \geq \ul 1};\bbC^\ast), \delta,d_1,\ldots, d_n\big)$
in $\vH_{\blue{\ecU}}^{n+2}(X; \bbC^\ast)$ with respect to the isomorphism of
Lemma~\ref{L:multigerbe_cech}.\eqref{I:multigerbe_cech_total}.
\label{D:multigerbe_class}
\end{defn}

\begin{prop}
The characteristic class is natural with respect to pullback, product and
inverse operations on multigerbes, and a morphism of multigerbes induces an
equality of the (pulled back) characteristic classes on the base spaces.
It vanishes if and only if the multigerbe admits a trivialization.
Moreover, $G(L)$ transforms according to the sign representation of the
symmetric group $\Sigma_n$ acting by permutation of the indices of the locally
split $n$-cube $X_\alpha$.
\label{P:char_class_properties}
\end{prop}

\begin{proof}[Proof sketch]
As for (bi)gerbes, the naturality of the characteristic class is a consequence
of the naturality of the multicomplex $\big(\vC_{\blue{\ecU}}^\bullet(X_\bullet; A), \delta,
d_1,\ldots,d_n\big)$ and the naturality of the Chern class of the line bundle
$L \to X_{\ul 2}$, and the equivalence between vanishing of $G(L)$ and multigerbe triviality of $L$
follows from Lemma~\ref{L:multigerbe_cech}.\eqref{I:multigerbe_cech_chern}.
Finally, that $G(L)$ is odd with respect to permutations of the $n$-cube is a consequence
of the sign convention, Convention~\ref{Conv:multicomplex}, since changing the order of the differentials
in the multicomplex by a permutation $\sigma$ involves multiplying the complex by powers of $-1$, and in particular
the sign $(-1)^{\mathrm{sgn}(\sigma)}$ on the term $\vC_{\blue{\ecU}}^1(X_{\ul 2}; \bbC^\ast)$.
\end{proof}

The question of representability of a given $(n+2)$-class by a multigerbe supported by
a given locally split $n$-cube can be addressed along similar lines as for bigerbes in \S\ref{S:rep_four_class}.
Consider the multicomplex $\big(\vC_{\blue{\ecU}}^\bullet(X_\bullet; \bbZ), \delta,d_1,\ldots,d_n\big)$ truncated to involve only the spaces
in the $n$-cube, so the $X_{\alpha}$ with $\alpha \in \set{0,1}^n$.
The $(\delta, D_{1\ldots,n})$ spectral sequence of this complex (with the $d_i$ rolled up into a single differential) has $E_1$ page
consisting of the cohomology complexes
\[
\begin{tikzcd}
	H^k(X; \bbZ) \ar[r,"D_{1\ldots,n}"] 
	& \bigoplus_{\abs \alpha = 1} H^k(X_\alpha; \bbZ) \ar[r, "D_{1\ldots,n}"] 
	& \bigoplus_{\abs \alpha = 2} H^k(X_\alpha; \bbZ) \ar[r, "D_{1\ldots,n}"] & {\cdots}
\end{tikzcd}
\]
for each $k \in \bbN$.
At the bottom level, the $D_{1\ldots,n}$ differential of a class in $H^{n+2}(X;
\bbZ)$ is just the sum of the pullbacks along the $n$-cube face maps to
$\bigoplus_{\abs \alpha = 1} H^{n+2}(X_\alpha; \bbZ)$, and if this vanishes,
then we say the class {\em survives to the $E_2$ page}.
In this case the $E_2$ differential maps the class into the quotient
$\bigoplus_{\abs \alpha = 2}H^{n+1}(X_\alpha; \bbZ)/\bigoplus_{\abs \alpha =1}
H^{n+1}(X_\alpha; \bbZ)$ (the cohomology of the $E_1$ page), and we say the
class {\em survives to the $E_3$ page} if this $E_2$ differential vanishes and so on.
Provided the class survives to the $E_n$ page, the associated differential maps
it into the quotient of $\bigoplus_{\abs \alpha = n} H^2(X_\alpha; \bbZ) =
H^{2}(X_{\ul 1}; \bbZ)$ by some complicated subgroup, and this is the last
nontrivial differential of the spectral sequence, which therefore stabilizes at
$E_{n+1} = E_\infty$. 
We say the class {\em survives to $E_\infty$}, or simply {\em stabilizes}, if it survives to $E_n$
and has vanishing $E_n$ differential.

\begin{prop}
A given locally split $n$-cube $X_\alpha$ supports an $n$-multigerbe representing a given class $\alpha \in H^{n+2}(X; \bbZ)$
if and only if $\alpha$ stabilizes in the above sense.
\label{P:multigerbe_representability}
\end{prop}
We leave the details of the proof, which we claim is a relatively straightforward generalization of the proof of 
Theorem~\ref{T:bigerbe_representability}, as an exercise.

\subsection{Examples} \label{S:multiberbe_examples}
We end with some simple examples of multigerbes which are straightforward generalizations
of the $3+1$ decomposable bigerbes of \S\ref{S:decomp_bigerbes} and the path bigerbes of \S\ref{S:Path-bigerbes}.

First, suppose $(L, X_\alpha)$ is an $n$-multigerbe over $X = X_{\ul 0}$ with characteristic class $\alpha = G(L) \in H^{n+2}(X; \bbZ)$,
and let $[\rho] \in H^1(X; \bbZ) \cong \vH^0(X; \UU(1))$ be a given 1-class represented by a homotopy class of maps $\rho : X \to \UU(1)$. 
We proceed to construct a `decomposable'  $(n+1)$-multigerbe representing the class $[\rho] \cup G(L)$. 
With $\wt X \to X$ the `logarithmic' $\bbZ$-covering of $X$ associated to $\rho$ as in \S\ref{S:decomp_bigerbes}, define the $(n+1)$-cube $\wt X_\beta$ by
\[
	\wt X_{(\alpha,0)} = X_\alpha,
	\quad \wt X_{(\alpha,1)} = \wt X \times_X X_{\alpha}.
\]
Then in the induced multisimplicial simplicial space,
$
	\wt X_{\ul 2} = \wt X^{[2]} \times_X X_{\ul 2}
$
and we define the line bundle by
\[
	L^\chi = (\mathrm{pr_2}^\ast L)^{\otimes \mathrm{pr}_1^\ast\chi} \to \wt X^{[2]} \times_X X_{\ul 2}
\]
where $\chi : \wt X^{[2]} \to \bbZ$ is the fiber shift map with $\wt X \to X$
thought of as a principal $\bbZ$-bundle.

\begin{prop}
With notation as above, $(L^\chi, \wt X_\beta)$ is an $(n+1)$-multigerbe with characteristic class
\[
	G(L^\chi) = [\rho] \cup G(L) \in H^{n+3}(X; \bbZ).
\]
\label{P:decomp_multigerbe}
\end{prop}

For the generalization of the path bigerbes of \S\ref{S:Path-bigerbes}, let $X$ be a locally contractible space with a chosen basepoint.
\begin{added}
\begin{lem}
With notation $P^1 Y = PY$ and $P^0 Y = Y$, the iterated (based) path spaces
\[
	X_\alpha = P^{\alpha_1} \cdots P^{\alpha_n}X = P^{\abs \alpha} X
\]
with evaluation maps $X_{\alpha} \to X_{\alpha - e_j}$ form a locally split
$n$-cube over $X$, provided $X$ is $(n-1)$-connected.
\label{L:path_split_cube}
\end{lem}
\begin{proof}
The obstruction to the locally split condition is the surjectivity of the maps $X_\alpha \to X_{<\alpha}$.
However, it can be checked that for each $\abs\alpha \in \bbN$, the limit $X_{<\alpha} \cong \cC_\ast(S^{\abs{\alpha}-1}; X)$
can be identified with the set of basepointed maps from the $(\abs{\alpha}-1)$-sphere into $X$.
Indeed,
the $X_{\alpha-e_j} \cong \cC_\ast([0,1]^{\abs{\alpha}-1}; X)$ are cube mapping spaces of the $(\abs{\alpha}-1)$-cube, and the limit can be realized
as the subspace of the product of these in which the cubes are identified along their boundaries in a way that assembles
into an $(\abs{\alpha}-1)$-sphere, with the image of an element in $X_\alpha = \cC_\ast([0,1]^{\abs{\alpha}}, X)$ in $X_{<\alpha}$
identified with the restriction map to the boundary of the disk $[0,1]^{\abs{\alpha}} \cong D^{\abs{\alpha}}$.
Since $X$ is $(n-1)$ connected, each of these maps is surjective.
\end{proof}
\end{added}
Since in this case all the $X_\alpha$ in the $n$-cube for $\alpha \neq \ul 0$ are contractible spaces, every class in $H^{n+2}(X; \bbZ)$
survives to the $E_\infty$ page in the $(\delta,D_{1\cdots n})$ spectral sequence of the $n$-cube, so in light of Proposition~\ref{P:multigerbe_representability}
we conclude the following.
\begin{prop}
For $X$ \blue{$(n-1)$-connected} and locally contractible, every class in $H^{n+2}(X; \bbZ)$ is represented by an $n$-multigerbe
supported on the iterated path $n$-cube $X_\alpha = P^{\abs{\alpha}} X$; in particular, the multisimplicial line bundle 
of the multigerbe lives on the iterated loop space $X_{\ul 2} = \Omega^n X$.
\label{P:path_multigerbe}
\end{prop}

Finally, as in \S\ref{S:Path-bigerbes} there is a free path/loop version of this multigerbe obtained at the cost of imposing product-multisimplicial conditions. 
Indeed, the set $\set{X^{m_1,\ldots,m_n} = X^{m_1\cdots m_n} : (m_1,\ldots,m_n) \in \bbN^n}$ of $n$-fold iterated products of $X$ along with projections forms a  multisimplicial space,
with induced ``differentials'' $\pa_1,\ldots, \pa_n$ defined on functions, line bundles, gerbes, multigerbes, etc.
A {\em product-multisimplicial} multigerbe is a multigerbe $L$ over $X^{2,\ldots,2} = X^{2^n}$ such that $\pa_i L$ is a trivial multigerbe for $1 \leq i \leq n$, and then
its characteristic class descends from $H^{n+2}(X^{2^n}; \bbZ)$ to $H^{n+2}(X; \bbZ)$.
Again leaving the details of the generalization of Theorem~\ref{T:free_loop_universal} as an exercise, we claim the following result.
\begin{prop}
If $X$ is connected and locally contractible, then every class in $H^{n+2}(X; \bbZ)$ is represented by a product-multisimplicial $n$-multigerbe
supported by the iterated free path $n$-cube $X_\alpha = I^{\abs \alpha} X^{(2-\abs{\alpha})^n}$ with $X_{\ul 0} = X^{2,\ldots,2} = X^{2^n}$; in particular, the 
line bundle of the multigerbe
lives on the free loop space $L^n X$ where it satisfies an $n$-fold fusion condition as well as the multi-figure-of-eight condition that $\pa_i L \to L\cdots L_8 \cdots L X$ are trivial for each $i$.
\label{P:free_path_multigerbe}
\end{prop}

\bibliography{bigerbe}
\bibliographystyle{plain}

\end{document}